\documentclass[10pt]{amsart}
\usepackage{fancyhdr}
\usepackage{stmaryrd}
\usepackage{calrsfs}

\usepackage{amsmath}
\usepackage[nospace,noadjust]{cite}
\usepackage{amsfonts}
\usepackage{amssymb,enumerate}
\usepackage{amsthm}
\usepackage{cite}
\usepackage{comment}
\usepackage{color}
\usepackage[all]{xy}
\usepackage{hyperref}
\usepackage{lineno}

\newtheorem*{maintheorem*}{Main Theorem}
\newtheorem{theorem}{Theorem}[section]
\newtheorem{prop}[theorem]{Proposition}

\newtheorem{lemma}[theorem]{Lemma}
\newtheorem{cor}[theorem]{Corollary}

\theoremstyle{definition}
\newtheorem{definition}[theorem]{Definition}
\newtheorem{remark}[theorem]{Remark}
\newtheorem{example}[theorem]{Example}
\numberwithin{equation}{section}

\newcommand{\nn}{\mathbb{N}}
\newcommand{\pp}{\mathbb{P}}
\newcommand{\qq}{\mathbb{Q}}
\newcommand{\rr}{\mathbb{R}}
\newcommand{\zz}{\mathbb{Z}}

\newcommand{\ifi}{\mathcal{I}_\infty}
\newcommand{\iii}{\mathcal{I}}

\newcommand{\gp}{\text{gp}}

\providecommand\ldb{\llbracket}
\providecommand\rdb{\rrbracket}

\hypersetup{
	pdftitle={OHF},
	colorlinks=true,
	linkcolor=blue,
	citecolor=cyan,
	urlcolor=wine
}

\keywords{one-dimensional monoid algebra, ascending chain of principal ideals, ACCP, almost ACCP, atomicity}

\subjclass[2010]{Primary: 13F15, 13A05; Secondary: 20M13, 13F05}

\begin{document}
	
	\mbox{}
	\title{One-dimensional monoid algebras and ascending chains \\ of principal ideals}
	
	\author{Alan Bu}
	\address{Department of Mathematics, Harvard University, Cambridge, MA 02138}
	\email{abu@college.harvard.edu}
	
	\author{Felix Gotti}
	\address{Department of Mathematics\\MIT\\Cambridge, MA 02139}
	\email{fgotti@mit.edu}
	
	\author{Bangzheng Li}
	\address{Department of Mathematics\\MIT\\Cambridge, MA 02139}
	\email{liben@mit.edu}

	\author{Alex Zhao}
	\address{Department of Mathematics\\MIT\\Cambridge, MA 02139}
	\email{alexzhao@mit.edu}

\date{\today}

\begin{abstract}
	 An integral domain $R$ is called atomic if every nonzero nonunit of $R$ factors into irreducibles, while $R$ satisfies the ascending chain condition on principal ideals if every ascending chain of principal ideals of $R$ stabilizes. It is well known and not hard to verify that if an integral domain satisfies the ACCP, then it must be atomic. The converse does not hold in general, but examples are hard to come by and most of them are the result of crafty and technical constructions. Sporadic constructions of such atomic domains have appeared in the literature in the last five decades, including the first example of a finite-dimensional atomic monoid algebra not satisfying the ACCP recently constructed by the second and third authors. Here we construct the first known one-dimensional monoid algebras satisfying the almost ACCP but not the ACCP (the almost ACCP is a notion weaker than the ACCP but still stronger than atomicity). Although the two constructions we provide here are rather technical, the corresponding monoid algebras are perhaps the most elementary known examples of atomic domains not satisfying the ACCP.
\end{abstract}
\medskip

\maketitle

%\tableofcontents

\bigskip
%%%%%%%%%%%
%%%%%%%%%%%
\section{Introduction}
\label{sec:intro}

The ascent of factorization and ideal-theoretical properties from a commutative ring $R$ to its polynomial extension $R[x]$ has been a relevant topic in the development of commutative algebra, not only because it yields new commutative rings with desirable algebraic properties, but also because polynomial rings are central in classical commutative ring theory and modern related fields, including algebraic geometry and homological algebra. Two of the most standard results in this direction are the Gauss theorem that the unique factorization property ascends from a commutative ring to its polynomial extensions and Hilbert Basis Theorem. More recent results in the same direction have been given by Zaks~\cite[Theorem~2.4]{aZ80}, Anderson, Anderson, and Zafrullah~\cite[Propositions~2.5 and~5.3]{AAZ90}, Roitman~\cite[Proposition~1.1]{mR93}, and Malcolmson and Okoh~\cite[Theorem~1.9]{MO09}. In a higher level of generality, the ascent of factorization and ideal-theoretical properties from a commutative monoid-ring pair $(M, R)$ to its monoid algebra $R[M]$ (consisting of all polynomial expressions with exponents in $M$ and coefficients in $R$) has been the subject of a great deal of investigation: a treatment of the ascent of many relevant properties to monoid algebras can be found in Gilmer~\cite{rG84}, which includes a significant part of the progress made until the mid eighties, and more recent studies have been provided by Gilmer and Parker~\cite{GP74}, Kim~\cite{hK01}, Juett, Mooney, and Roberts~\cite{JMR21}, Juett and Medina~\cite{JM22}, Chang, Fadinger, and Windisch~\cite{CFW22}, and the second author~\cite{fG22}.
\smallskip

It seems that the ascent of atomicity to monoid algebras (i.e., from any pair $(M,R)$ to its corresponding monoid algebra $R[M]$) was first brought to attention by Gilmer \cite[page 189]{rG84} back in 1984 as an open problem. A cancellative and commutative monoid is called atomic if each of its non-invertible elements factors into irreducibles (also called atoms), and an integral domain is called atomic if its multiplicative monoid is atomic. The potential ascent of atomicity to polynomial extensions, which is a specialization of the same ascent problem (when the monoid of exponents is $\nn_0$), was emphasized by Anderson, Anderson, and Zafrullah \cite[Question~1]{AAZ90} in their influential paper where the bounded and finite factorization properties were introduced and first investigated. The ascent of atomicity to polynomial extensions was settled by Roitman in~\cite{mR93}, where he constructed interesting classes of atomic domains with non-atomic polynomial extensions. The dual specialized problem of whether, for a field~$F$, the monoid algebra $F[M]$ is atomic when $M$ is atomic was not settled until more recently: Coykendall and the second author in~\cite{CG19} gave a negative answer for any prime field and, even more recently, Rabinovitz and the second author in~\cite{GR23} provided a more general answer, constructing a rank-one atomic monoid $M$ such that the monoid algebra $F[M]$ is not atomic for any field $F$.
\smallskip

The construction of atomic domains that do not satisfy the ACCP is considered a challenging task and a problem that has been the subject of systematic attention in the literature since Grams~\cite{aG74} constructed the first of such integral domains back in 1974, disproving an assertion made by Cohn in~\cite{pC68} (the equivalence of atomicity and the ACCP inside the class of integral domains). Further constructions of atomic domains not satisfying the ACCP have been provided by Zaks~\cite{aZ82}, Roitman~\cite{mR93}, Boynton and Coykendall~\cite{BC19}, and the second and third authors~\cite{GL23,GL23a,GL22}. Unlike atomicity, the ACCP ascends to polynomial extensions over integral domains \cite[page~82]{GP74} (this is not the case in the presence of zero-divisors; see~\cite{HL94}). In addition, for each field $F$, a reduced, cancellative, and torsion-free monoid $M$ satisfies the ACCP if and only if its monoid algebra $F[M]$ satisfies the ACCP \cite[Theorem~13]{AJ15}. Thus, when the atomicity of any (atomic) non-ACCP monoid ascends to its monoid algebra $F[M]$ for some field $F$, one automatically obtains an atomic monoid algebra that does not satisfy the ACCP. This is how atomic monoid algebras that do not satisfy the ACCP were constructed in~\cite{GL23,GL23a}. However, the atomic monoid used in~\cite{GL23} has infinite rank and the one used in~\cite{GL23a} has rank $2$; therefore, none of them yields one-dimensional monoid algebras. The problem of whether there exists a one-dimensional atomic monoid algebra that does not satisfy the ACCP is still open\footnote{The integral domain in Grams' construction is one-dimensional but is not a monoid algebra.}, and here we resolve this problem, giving a positive answer.
\smallskip

In order to do so, we will describe two distinct constructions of one-dimensional monoid algebras satisfying the almost ACCP (a condition stronger than atomicity) but not the ACCP. An atomic monoid~$M$ satisfies the almost ACCP if each finite nonempty subset $S$ of $M$ has a common divisor $d$ in $M$ such that for some $s \in S$ every ascending chain of principal ideals of~$M$ starting at that generated by $\frac{s}d$ stabilizes, while an integral domain satisfies the almost ACCP if its multiplicative monoid satisfies the almost ACCP. The almost ACCP was recently investigated by the second and third authors in~\cite{GL23a}. The almost ACCP is an ideal-theoretical property that is stronger than atomicity but weaker than the ACCP. It turns out that the almost ACCP ascends to polynomial extensions \cite[Corollary~5.2]{GL23a} (but not to monoid algebras over fields \cite[Example~5.5]{GL23a}). Thus, the atomic domains $R$ constructed in~\cite{mR93} whose corresponding polynomial extensions $R[x]$ are not atomic are examples of atomic domains not satisfying the almost ACCP. For constructions of integral domains satisfying the almost ACCP but not the ACCP, see \cite[Example~5.7 and Theorem~5.15]{GL23a}.
\smallskip

In Theorem~\ref{thm:main}, which is the first main result of this paper, we prove that the monoid algebra $\qq[M_{3/4}]$ satisfies the almost ACCP but not the ACCP, therefore, obtaining the first one-dimensional atomic monoid algebra not satisfying the ACCP. Given a positive rational $r$, we let $M_r$ stand for the smallest additive monoid containing all the nonnegative powers of $r$; that is, $M_r := \langle r^n : n \in \nn_0 \rangle$. Observe that $M_r$ is also the smallest subsemiring of $\qq$ containing $r$, and so we can also write $M_r$ as $\nn_0[r]$. When $r^{-1} \notin \nn$ the monoid $M_r$ satisfies the almost ACCP, and when $r < 1$ the monoid $M_r$ does not satisfy the ACCP (see \cite[Example~3.8]{GL23a}). Hence the monoids $M_r$, parameterized by $r \in \qq \cap (0,1)$ with $r^{-1} \notin \nn$, are perhaps the simplest (rank-one) monoids that satisfy the almost ACCP but not the ACCP (only compared with the rank-one monoid in Grams' construction of the first atomic domain not satisfying the ACCP). Atomic and factorization properties of the monoids $M_r$ have been studied in~\cite{CGG20a,JLZ23} and the same for some generalizations of these monoids (when $r$ is algebraic) have been studied in~\cite{CG19, ABLST23}. Since the almost ACCP does not ascend to monoid algebras over fields, when $M_r$ satisfies the almost ACCP but not the ACCP, there are no standard tools at our disposal to determine whether the monoid algebra $\qq[M_r]$ satisfies the almost ACCP. It is conjectured in~\cite[Conjecture~4.11]{fG22} that $\qq[M_r]$ is atomic when $M_r$ is atomic, and our result that $\qq[M_{3/4}]$ satisfies the almost ACCP provides a piece of evidence supporting this conjecture.
\smallskip

In Theorem~\ref{thm:rank 1 half-ACCP}, which is the other main result of this paper, we construct a class of one-dimensional monoid algebras satisfying the almost ACCP but not the ACCP. Although the monoid algebra $\qq[M_{3/4}]$ is more natural than those in Theorem~\ref{thm:rank 1 half-ACCP}, the machinery we need to put together to prove that $\qq[M_{3/4}]$ satisfies the almost ACCP is heavier than the arguments we use to prove that the more esoteric monoid algebras in Theorem~\ref{thm:rank 1 half-ACCP} satisfy the almost ACCP.
\smallskip

In Section~\ref{sec:divisibility in Z[x]}, we develop the mentioned machinery and, as part of it, we establish some divisibility and factorization results in the classical polynomial ring $\zz[x]$ that we hope can be of future independent interest.

\bigskip
%%%%%%%%%%%%
%%%%%%%%%%%%
\section{Background}
\label{sec:background}

\medskip
%%%%%%%%%%%%%%%%
\subsection{General Notation}

As it is customary, we let $\zz$, $\qq$, and $\rr$ denote the sets of integers, rational numbers, and real numbers, respectively. For a subset $S$ of $\rr$ and $r \in \rr$, we set $S_{\ge r} := \{s \in S : s \ge r\}$. We let $\nn$ denote the set of positive integers, and we set $\nn_0 := \nn \cup \{0\}$. Also, we let $\pp$ denote the set of primes. For $b,c \in \zz$, we let $\ldb b,c \rdb$ be the discrete interval between $b$ and $c$; that is,
\[
	\ldb b,c \rdb := \{n \in \zz : b \le n \le c\}.
\]
Observe that $\ldb b,c \rdb$ is empty when $b > c$. For a nonzero $q \in \qq$, we let $\mathsf{n}(q)$ and $\mathsf{d}(q)$ denote, respectively, the unique $n \in \zz$ and $d \in \nn$ such that $q = n/d$ and $\gcd(n,d) = 1$. For each $p \in \pp$, we let $v_p \colon \qq^\times \to \zz$ denote the $p$-adic valuation map. 
\medskip

%%%%%%%%%%%%%%%%%%%%%%%%%%%%%%
\subsection{Commutative Monoids, Chains of Ideals, and Factorizations}

Throughout this paper, the term \emph{monoid} is reserved for a cancellative and commutative semigroup with an identity element. As we are mostly interested here in the multiplicative structure of polynomial extensions and monoid algebras, we will write monoids multiplicatively unless we specify otherwise. Let $M$ be a monoid. The \emph{Grothendieck group} of $M$ is denoted by $\gp(M)$; that is, $\gp(M)$ is the unique abelian group up to isomorphism satisfying that any abelian group containing an isomorphic image of $M$ also contains an isomorphic image of $\gp(M)$. The \emph{rank} of $M$, denoted $\text{rank} \, M$, is the rank of $\gp(M)$ as a $\zz$-module. If $\gp(M)$ is a torsion-free group, then $M$ is said to be a \emph{torsion-free} monoid. Every rank-$1$ torsion-free monoid that is not a group is isomorphic to an additive submonoid consisting of nonnegative rationals \cite[Theorem~3.12]{GGT21}. These monoids, known as \emph{Puiseux monoids}, play a relevant role in this paper. The atomic structure and the arithmetic of factorizations of Puiseux monoids and their monoid algebras have been actively investigated in recent years (see the recent papers~\cite{CJMM24,GG24} and references therein). For a subset~$S$ of $M$, we let $\langle S \rangle$ denote the smallest submonoid of $M$ containing $S$, and we say that $S$ \emph{generates} $M$ if $M = \langle S \rangle$. The group of units of $M$ is denoted by $M^\times$, and $M$ is called \emph{reduced} if the group $M^\times$ is trivial. 
\medskip

An element $a \in M \setminus M^\times$ is called an \emph{atom} (or an \emph{irreducible}) provided that the equality $a = bc$ for some $b,c \in M$ implies that either $b \in M^\times$ or $c \in M^\times$. The set consisting of all atoms of $M$ is denoted by $\mathcal{A}(M)$. An element of $M$ is called \emph{atomic} if either it is a unit or it factors into atoms. Following Cohn~\cite{pC68}, we say that $M$ is \emph{atomic} if every element of $M$ is atomic. Now assume that the monoid $M$ is atomic. The quotient semigroup $M/M^\times$ is also a monoid, which is atomic because $M$ is atomic. The free (commutative) monoid on the set of atoms of $M/M^\times$ is denoted by $\mathsf{Z}(M)$. Let $\pi \colon \mathsf{Z}(M) \to M/M^\times$ be the unique monoid homomorphism fixing the set $\mathcal{A}(M/M^\times)$. For each $b \in M$, we set
\[
	\mathsf{Z}(b) := \pi^{-1}(bM^\times) \subseteq \mathsf{Z}(M), %\quad \text{ and } \quad \mathsf{L}(b) := \{|z| : z \in \mathsf{Z}(b)\}.
\]
and we call the elements of $\mathsf{Z}(b)$ \emph{factorizations} of $b$ in $M$. Observe that because $M$ is an atomic monoid, $\mathsf{Z}(b)$ is a nonempty set for all $b \in M$.
\medskip

%\subsection{Chains of Ideals} 
A subset $I$ of $M$ is called an \emph{ideal} if the set $IM := \{bm : b \in I \text{ and } m \in M\}$ is contained in~$I$ or, equivalently, if the equality $IM = I$ holds. An ideal of the form $bM$ for some $b \in M$ is called a \emph{principal ideal}. For $b,c \in M$, we say that $c$ \emph{divides}~$b$ in~$M$ if $b = cd$ for some $d \in M$, in which case we write $c \mid_M b$: observe that $c \mid_M b$ if and only if $bM \subseteq cM$. A sequence of ideals $(I_n)_{n \ge 1}$ of $M$ is called \emph{ascending} if $I_n \subseteq I_{n+1}$ for every $n \in \nn$, while $(I_n)_{n \ge 1}$ is said to \emph{stabilize} if there exists $N \in \nn$ such that $I_n = I_N$ for every $n \ge N$. An element $b \in M$ satisfies the \emph{ascending chain condition on principal ideals} (ACCP) if every ascending chain of principal ideals whose first term is $bM$ stabilizes. The monoid~$M$ satisfies the \emph{ACCP} provided that every element of $M$ satisfies the ACCP. Following \cite{GL23a}, we say that $M$ satisfies the \emph{almost ACCP} if for every finite nonempty subset $S$ of $M$, there exists an atomic common divisor $d$ of $S$ in $M$ and an element $s \in S$ such that $\frac{s}d$ satisfies the ACCP. It is clear that every monoid that satisfies the ACCP also satisfies the almost ACCP, and it follows from \cite[Proposition~3.10(2)]{GL23a} that every monoid satisfying the almost ACCP is atomic.

\medskip
%%%%%%%%%%%%%%%%%%%%%%%%%%%%%%%%%%%%%%
\subsection{Integral Domains and Monoid Algebras}

Let $R$ be an integral domain. We let $R^* := R \setminus \{0\}$ denote the multiplicative monoid of $R$. As it is customary, we let $R^\times$ denote the group of units of $R$, and we let $\mathcal{A}(R)$ denote the set of irreducibles of $R$: note that $(R^*)^\times = R^\times$ and $\mathcal{A}(R^*) = \mathcal{A}(R)$. In addition, for $r,s \in R^*$, we observe that $r$ divides $s$ in $R^*$ if and only if $r$ divides $s$ in $R$: for simplicity, we write $r \mid_R s$ instead of $r \mid_{R^*} s$. As there is a natural inclusion-preserving bijection between the set of principal ideals of the monoid $R^*$ and the set of principal ideals of the ring $R$, we see that $R^*$ satisfies the ACCP as a monoid if and only if $R$ satisfies the ACCP as a ring. We say that $R$ is \emph{atomic} (resp., satisfies the \emph{almost ACCP}) provided that $R^*$ is atomic (resp., satisfies the almost ACCP). As we mentioned in the introduction, there are atomic domains that do not satisfy the almost ACCP, as well as integral domains that satisfy the almost ACCP but do not satisfy the ACCP. Following~\cite{AQ97}, we say that an integral domain is an \emph{AP-domain} if every irreducible element is prime (GCD-domains and, in particular, UFDs are classical examples of AP-domains).
\medskip

Let $R$ be an integral domain, and let $M$ be an additive monoid. The ring consisting of all polynomial expressions in a variable $x$ with exponents in $M$ and coefficients in $R$ is called the \emph{monoid algebra} of~$M$ over $R$. Following Gilmer~\cite{rG84}, we will denote the monoid algebra of~$M$ over $R$ by either $R[x;M]$, or simply $R[M]$ if we see no risk of ambiguity. If $M$ is torsion-free, then $R[M]$ is a integral domain and
\[
	R[M]^\times = \{ux^m : u \in R^\times \text{ and } m \in M^\times \};
\]
see \cite[Theorems~8.1 and 11.1]{rG84}. The monoid $M$ is called a \emph{totally ordered monoid}\footnote{A monoid can be turn into a totally ordered monoid if and only if it is torsion-free; see \cite[Section~3]{rG84}.} with respect to a given total order relation $\le$ on $M$ provided that $\le$ is compatible with the operation of $M$; that is, for all $b,c,d \in M$ the inequality $b < c$ guarantees that $b+d < c+d$. Assume now that $M$ is a totally ordered monoid under the order relation $\le$. Then we can write any nonzero element $f \in R[M]$ as $f := c_1 x^{m_1} + \dots + c_k x^{m_k}$ for some nonzero $c_1, \dots, c_k \in R$ and exponents $m_1, \dots, m_k \in M$ with $m_1 > \dots > m_k$. In this case, we call $\text{supp} \, f := \{m_1, \dots, m_k\}$ the \emph{support} of $f$, and we call $\deg f := m_1$ and $\text{ord} \, f := m_k$ the \emph{degree} and \emph{order} of $f$, respectively. Observe that when the monoid $M$ is $\nn_0$ under the usual order, we recover the standard notions of degree, order, and support for polynomials in $R[x]$.
\smallskip

It follows from~\cite[Proposition 8.3]{GP74} that when a monoid algebra $R[M]$ is an integral domain, $\dim R[M] \ge 1 + \dim R$. Moreover, when $R$ is a Noetherian domain and $M$ is a torsion-free monoid, it follows from \cite[Corollary~2]{jO88} that
\[
	\dim R[M] = \dim R + \text{rank} \, M.
\]
As a consequence, the one-dimensional monoid algebras over fields that are integral domains are precisely the monoid algebras of rank-one torsion-free monoids over fields (every rank-one torsion-free monoid is isomorphic to an additive submonoid of $\qq$; see, for instance, \cite[Section~24]{lF70}). %Background information on monoid algebras $R[M]$, emphasizing on the ascent of algebraic properties from $(M, R)$ to $R[M]$ and including the most significant progress on the same subject until 1984, can be found in Gilmer's book~\cite{rG84}.
\medskip

%%%%%%%%%%%%%%%%%%%%%%%%%%%%
\subsection{Symmetric and Cyclotomic Polynomials}
%For any integral domain $R$, it is clear that $R[x;\nn_0]$ yields the standard ring of polynomials $R[x]$ over $R$. 
Let $F[x_1, \dots, x_n]$ be the polynomial ring in $n$ variables over a field $F$. For each $k \in \nn_0$, let $e_k := \sum_{j_1 < \cdots < j_k} x_{j_1} \cdots x_{j_k}$ and $p_k := \sum_{i=1}^n x_i^k$ be the $k$-th elementary and the $k$-th power-sum symmetric functions on~$R$, respectively (we are assuming here that $e_0 = 1$). The Newton's identities state that
\begin{equation} \label{eq:Newton identity}
	\quad \quad ke_k = \, \sum_{i=1}^k (-1)^{i-1} e_{k-i} p_i \quad \ (\text{for all} \ k \in \ldb 1,n \rdb).
%	0		&= \!\! \sum_{i=k-n}^k \! \! (-1)^{i-1}e_{k-i}p_i   \quad \! \text{ if }  \quad k> n.
\end{equation}

Now let $P(x) := x^n + \sum_{i=0}^{n-1} c_i x^i \in F[x]$ be an $n$-degree monic polynomial with roots $r_1, \dots, r_n$ in an algebraic extension field of $F$. Then we can write $P(x) = \prod_{i=1}^n (x - r_i) = \sum_{k=0}^n c_k x^k$, %(-1)^k e_k(r_1, \dots, r_n) x^{n-k}$, 
where $c_n = 1$ and $c_{n-k} = (-1)^k e_k(r_1, \dots, r_n)$ for every $k \in \ldb 1,n \rdb$. Then we see that
\begin{equation*} \label{eq:Girard-Newton identities}
	k c_{n-k} + \sum_{j=0}^{k-1} c_{n-j} p_{k-j}(r_1, \dots, r_n) = (-1)^k ke_k(r_1, \dots, r_n) + \sum_{j=0}^{k-1} (-1)^j e_j(r_1, \dots, r_n) p_{k-j}(r_1, \dots, r_n) = 0
	%0 = \sum_{i=0}^k (-1)^{k-i} e_{k-i}(r_1, \dots, r_n) p_i= \sum_{i=0}^k c_{n-(k-i)} p_i = c_n p_k + c_{n-1} p_{k-1} + \dots + c_{n-k} p_0 = 0
\end{equation*}
for every $k \in \ldb 0,n \rdb$, where the last equality follows from~\eqref{eq:Newton identity} after plugging $i := k-j$. Therefore we obtain the following known lemma, which we will record here and reference later.

\begin{lemma} \label{lem:power sum symmetric functions}
	Let $F$ be a field, and let $P(x) := x^n + \sum_{i=0}^{n-1} c_i x^i \in F[x]$ be an $n$-degree monic polynomial with roots $r_1, \dots, r_n$ in an algebraic extension field of $F$. If~$k$ is the smallest positive integer such that $c_{n-k} \neq 0$, then $k$ is also the smallest positive integer such that $\sum_{i=1}^n r_i^k \neq 0$.
\end{lemma}

As it is customary, for every $n \in \nn$, we let $\Phi_n(x) \in \zz[x]$ denote the $n$-th cyclotomic polynomial. A polynomial $\sum_{i=0}^n c_i x^i \in \qq[x]$ is called \emph{palindromic} (or \emph{self-reciprocal}) provided that $c_i = c_{n-i}$ for every $i \in \ldb 0,n \rdb$. One can readily check that the product of two palindromic polynomials is again a palindromic polynomial. Also, it is well known that every cyclotomic polynomial is palindromic.

It is due to Kronecker~\cite{lK1857} that if $P(x)$ is a monic irreducible polynomial in $\zz[x]$ with all its roots in the complex closed unit disc, then $P(x)$ is a cyclotomic polynomial. As this is the only Kronecker's result that is relevant in the scope of this paper, we will refer to it as Kronecker's theorem. The following known lemma will be useful later (we include its proof here for the sake of completeness).

\begin{lemma} \label{lem:cyclotomic polynomial product identity}
	If $n \in \nn$ is odd and $m \in \nn_0$, then $\Phi_n\big( x^{2^m} \big)$ can be factorized as follows:
	\begin{equation} \label{eq:identity of cyclotomic polynomials}
		\Phi_n \big( x^{2^m} \big) = \prod_{i=0}^m \Phi_{n \cdot 2^i}(x).
	\end{equation}
\end{lemma}

\begin{proof}
	Fix an odd $n \in \nn$. We proceed by induction on $m$. The case when $m=0$ is trivial, and the case when $m=1$ follows from the well-known identity $\Phi_n(x^p) =  \Phi_{pn}(x) \Phi_n(x)$ whenever $p \in \pp$ such that $p \nmid n$ (after taking $p=2$). Now suppose that~\eqref{eq:identity of cyclotomic polynomials} holds for some $m \in \nn$. Since $n$ is odd, $\Phi_n(x^2) =  \Phi_{2n}(x) \Phi_n(x)$ and, therefore, $\Phi_n\big(x^{2^{m+1}} \big) = \Phi_{2n}\big(x^{2^m}\big) \Phi_n\big(x^{2^m} \big)$. Thus,
	\[
		\Phi_n \big( x^{2^{m+1}} \big) = \Phi_{2n}\big(x^{2^m}\big) \Phi_n\big(x^{2^m} \big) = \Phi_{2n}\big(x^{2^m}\big) \prod_{i=0}^m \Phi_{n \cdot 2^i}(x) = \prod_{i=0}^{m+1} \Phi_{n \cdot 2^i}(x),
	\]
	where the second equality follows from our inductive hypothesis and the third equality follows from the identity $\Phi_{2n}\big( x^{2^m}\big) = \Phi_{n \cdot 2^{m+1}}(x)$. Hence the identity~\eqref{eq:identity of cyclotomic polynomials} holds for every $m \in \nn_0$.
\end{proof}

\bigskip
%%%%%%%%%%%%%%%%%%%%%%%%%%%
%%%%%%%%%%%%%%%%%%%%%%%%%%%
\section{Divisibility and Factorizations in $\zz[x]$}
\label{sec:divisibility in Z[x]}

In this section, we establish several results on the divisibility of polynomials in $\zz[x]$ paying special attention to polynomials of the form $P(x^{2^n})$ for some $n \in \nn$. The machinery developed in this section will play a crucial role in Section~\ref{sec:a rank 1 example}, and we hope it may be of independent interest. %This will allow us to characterize, in the next section, the irreducible polynomials of $\zz[x]$ that have infinitely many non-associate divisors inside the monoid algebra $\qq[M_{1/2}]$ (such a characterization will yield some applications in factorization theory).

\medskip
%%%%%%%%%%%%%%%
\subsection{Square Lemma}

Even if a polynomial $R(x) \in \zz[x]$ is irreducible, it often happens that $R(x^2)$ is reducible in $\zz[x]$. However, under certain circumstances we can have control over the length and form of the factorization into irreducibles of $R(x^2)$. Let us illustrate this observation with an example.

\begin{example} \label{ex:x^2 + 1}
	Consider the monic irreducible polynomial $R(x) := x^2 + 4 \in \zz[x]$. After setting $P(x) := x+2$ and $Q(x) := 2$, we can write $R(x) = P(x)^2 - xQ(x)^2$. Even though $R(x)$ is irreducible, the polynomial $R(x^2) = x^4 + 4$ is reducible in $\zz[x]$. Indeed, we can write $R(x^2) = (x^2 - 2x + 2)(x^2 + 2x + 2)$, which can also be written in terms of the polynomials $P(x)$ and $Q(x)$ as follows:
	\[
		R(x^2) = P(x^2)^2 - x^2 Q(x^2)^2 = (P(x^2) - xQ(x^2))(P(x^2) + xQ(x^2)).
	\]
	Moreover, the factors $P(x^2) - xQ(x^2)$ and $P(x^2) + xQ(x^2)$ are both irreducibles in $\zz[x]$.
\end{example}

The existence of the polynomials $P(x)$ and $Q(x)$ in Example~\ref{ex:x^2 + 1} parameterizing not only $R(x)$ but also the irreducible factors of $R(x^2)$ is not a coincidence. This is the essence of the following lemma, which will be crucial throughout this paper.

%\begin{lemma}%[Square Lemma] \label{lem:square lemma}
%	Let $R(x) \in \zz[x]$ be a monic irreducible polynomial with even degree. If $R(x^2)$ is reducible, then there exist $P(x)$ and $Q(x)$ in $\zz[x]$ such that the following statements hold.
%	\begin{enumerate}
%		\item $R(x) = P(x)^2 - xQ(x)^2$ and
%		\smallskip
%		
%		\item $R(x^2) = (P(x^2) - xQ(x^2))(P(x^2) + xQ(x^2))$ is the only factorization of $R(x^2)$ in $\zz[x]$ into irreducibles.
%	\end{enumerate} 
%\end{lemma}

\begin{lemma}[Square Lemma] \label{lem:square lemma}
	Let $R(x) \in \zz[x]$ be an irreducible polynomial. If $R(x^2)$ is reducible, then there exist polynomials $P(x)$ and $Q(x)$ in $\zz[x]$ such that the following statements hold.
	\begin{enumerate}
		\item $R(x)$ is associate to $P(x)^2 - xQ(x)^2$, and
		\smallskip
		
		\item $(P(x^2) - xQ(x^2))(P(x^2) + xQ(x^2))$ is the only factorization of $R(x^2)$ in $\zz[x]$ into irreducibles.
	\end{enumerate}
	% Can we guarantee that $P(x)$ and $Q(x)$ have positive leading coefficients. Is this really needed in what follows?
\end{lemma}

\begin{proof}
	Let $n$ be the degree of $R(x)$. After replacing $R(x)$ by $-R(x)$ if necessary, we can assume that $R(x)$ has a positive leading coefficient. Assume that $R(x^2)$ is reducible in $\zz[x]$, and take nonunit polynomials $A(x)$ and $B(x)$ in $\zz[x]$ such that $R(x^2) = A(x)B(x)$. Since $R(x)$ is irreducible, both $A(x)$ and $B(x)$ are nonconstant polynomials.
	\smallskip
	
	(1) Let us proceed to argue the existence of polynomials $P(x)$ and $Q(x)$ in $\zz[x]$ both with positive leading coefficients such that $R(x)$ is associate to $P(x)^2 - xQ(x)^2$. To do so, write
	\[
		A(x) = A_1(x^2) + xA_2(x^2) \quad \text{ and } \quad B(x) = B_1(x^2) + xB_2(x^2)
	\]
	for some polynomials $A_1(x), A_2(x), B_1(x), B_2(x) \in \zz[x]$. Observe that if $A_2(x)$ were the zero polynomial, then $x A_1(x^2) B_2(x^2) = R(x^2) - A_1(x^2) B_1(x^2) \in \zz[x^2]$, which would imply that $B_2(x)$ is also the zero polynomial: however, in this case, the equality $R(x) = A_1(x)B_1(x)$ would contradict the irreducibility of $R(x)$ in $\zz[x]$ (as both $A_1(x)$ and $B_1(x)$ would be nonconstant polynomial). Hence we can assume that $A_2(x) \neq 0$ and $B_2(x) \neq 0$. In a similar way, from the equality $R(x^2) = A(x)B(x)$, we can deduce that $x \big( A_1(x^2) B_2(x^2) + A_2(x^2) B_1(x^2) \big) \in \zz[x^2]$, which means that $A_1(x^2) B_2(x^2) + A_2(x^2) B_1(x^2) = 0$. Therefore
	\begin{equation} \label{eq:SL}
		A_1(x)B_2(x) + A_2(x)B_1(x) = 0.
	\end{equation}
	This, along with the fact that $\zz[x]$ is a UFD, allows us to write $A_1(x) = a(x) b(x)$ for some $a(x)$ and $b(x)$ in $\zz[x]$ such that $a(x)$ divides $A_2(x)$ and $b(x)$ divides $B_1(x)$. Now take $c(x)$ and $d(x)$ in $\zz[x]$ such that $A_2(x) = a(x) c(x)$ and $B_1(x) = b(x)d(x)$. It follows from~\eqref{eq:SL} that $B_2(x) = -c(x)d(x)$. As a result,
	\[
		R(x^2) = \big( A_1(x^2) + xA_2(x^2) \big) \big( B_1(x^2) + xB_2(x^2) \big) = a(x^2) d(x^2) \big( b(x^2)^2 - x^2 c(x^2)^2 \big),
	\]
	which implies that $R(x) = a(x)d(x) \big( b(x)^2 - x c(x)^2 \big)$. Since $\deg b(x)^2$ is even and $\deg x c(x)^2$ is odd, $b(x)^2 - x c(x)^2$ is a nonconstant polynomial. Therefore the irreducibility of $R(x)$ in $\zz[x]$ ensures that $a(x)d(x) \in \{\pm 1\}$. Thus, $R(x) = \pm \big( b(x)^2 - x c(x)^2 \big)$. Since $R(x)$ has positive leading coefficient, we can infer that $R(x) = (-1)^n\big( b(x)^2 - x c(x)^2 \big)$ from the facts that $\deg b(x)^2$ is even and $\deg xc(x)^2$ is odd. After setting $P(x) := b(x)$ and $Q(x) := c(x)$, we obtain that $R(x)$ is associate to $P(x)^2 - x Q(x)^2$.
%	
%	This in turn implies that the leading coefficient of $b(x)^2 - x c(x)^2$ is that of $b(x)^2$, which is monic because $R(x)$ and $R(x) = a(x)d(x) \big( b(x)^2 - x c(x)^2 \big)$. Hence $a(x) d(x) = 1$. After setting $P(x) := b(x)$ and $Q(x) := c(x)$, we obtain that $R(x) = P(x)^2 - x Q(x)^2$.
	\smallskip
	
	(2) We have just proved that for any pair of nonunit polynomials $A(x)$ and $B(x)$ satisfying $R(x^2) = A(x) B(x)$ there exist $k \in \nn$ and $b(x), c(x) \in \zz[x]$ such that
	\[
		A(x) = (-1)^k (b(x^2) + xc(x^2)) \quad \text{ and } \quad B(x) = (-1)^{n-k} (b(x^2) - xc(x^2)),
	\]
	which in turn implies that $\deg A(x) = \deg B(x) = n$. Hence any nonunit polynomial dividing $R(x^2)$ in $\zz[x]$ that is not associate with $R(x^2)$ must have degree $n$. As a consequence, if $P(x)$ and $Q(x)$ are the polynomials in part~(1), then the factors on the right-hand side of the equality
	\[
		(-1)^n R(x^2) = (P(x^2) - xQ(x^2))(P(x^2) + xQ(x^2))
	\]
	must be irreducible. As $\zz[x]$ is a UFD, we conclude that $(P(x^2) - xQ(x^2))(P(x^2) + xQ(x^2))$ is the only factorization of $(-1)^n R(x^2)$ in $\zz[x]$.
	%(2) We prove now that both $A'(x) := P(x^2) - xQ(x^2)$ and $B'(x) := (-1)^n(P(x^2) + xQ(x^2))$ are irreducible polynomials in $\zz[x]$. %Then $A'(x) B'(x) = R(x^2)$ up to associates, and $\deg A'(x) = \deg B'(x)$ so $\deg A'(x) = \deg R(x)$. Suppose, towards a contradiction, that $A'(x)$ is not irreducible in $\zz[x]$, and write $A'(x) = C(x) D(x)$ for some nonconstant polynomials $C(x)$ and $D(x)$ in $\zz[x]$. Observe that $\deg A'(x) = \deg B'(x)$, and so the fact that $R(x^2) = A'(x) B'(x)$ implies that $\deg A'(x) = \deg R(x) = n$. Now observe that the argument used to establish part~(1) can be similarly used for $R(x^2) = C(x) (D(x) B'(x))$ with $C(x)$ and $D(x)B'(x)$ in place of $A'(x)$ and $B'(x)$, respectively. Therefore we obtain that $\deg C(x) = \deg D(x) + \deg B'(x)$, which implies that $\deg D(x) = 0$, which is a contradiction. Thus, $A'(x)$ must be irreducible. A similar argument can be used to show that $B'(x)$ is also irreducible.
%	
%	\smallskip
%	After replacing $x$ by $x^2$ in $R(x) = \textcolor{red}{(-1)^n(}P(x)^2 - xQ(x)^2)$, we obtain that $R(x^2) = (P(x^2) - xQ(x^2))(-1)^n(P(x^2) + xQ(x^2)) = A'(x)B'(x)$, yielding the unique factorization of $R(x^2)$ in $\zz[x]$ into irreducibles. 
%	
%	 If $A(x)$ and $B(x)$ are one of such pairs of polynomials, then $A(x)$ must be irreducible: indeed, as $A(x)$ is a primitive polynomial, any irreducible factor $A'(x)$ of $A(x)$ with $\deg A'(x) < \deg A(x)$ would allow us to write $R(x^2)$ as the product of two poynomials with different degrees, namely, $A'(x)$ and $R(x^2)/A'(x)$. 
\end{proof}

With notation as in Square Lemma, we note that the irreducibility of $R(x)$ does not give any information about the irreducibility of $R(x^2)$. Indeed, we have seen in Example~\ref{ex:x^2 + 1} that for the irreducible polynomial $R_1(x) = x^2 + 4$, the polynomial $R_1(x^2)$ is reducible, while for the irreducible polynomial $R_2(x) = x^2 + 1$, the polynomial $R_2(x^2)$ is irreducible. 
%
%These observations are illustrated in the following examples.
%
%\begin{example} \hfill
%	\begin{enumerate}
%		\item Although $R(x) := x^2 + 1 \in \zz[x]$ is a monic irreducible polynomial with even degree, the polynomial $R(x^2) := x^4 + 1$ is still irreducible in $\zz[x]$.
%		\smallskip
%		
%		\item \textcolor{red}{Comment 2: Can we find a monic irreducible polynomial $R(x)$ such that $R(x^2)$ factor into two irreducible polynomials in $\zz[x]$ having distinct degrees.}
%%		\smallskip
%%		
%%		\item \textcolor{red}{TODO: Find a monic irreducible polynomial $R(x)$ such that $R(x^2)$ factor into more than two irreducible polynomials in $\zz[x]$.}
%	\end{enumerate}
%\end{example}

\medskip
%%%%%%%%%%%%%%%%%
\subsection{Good Polynomials}

Following the usual notation from standard modular arithmetic, for any polynomials $P(x), Q(x)$, and $M(x)$ in $\zz[x]$ with $M(x) \neq 0$, we write
\[
	P(x) \equiv Q(x) \pmod{M(x)}
\]
provided that $P(x) - Q(x)$ is divisible by $M(x)$ in the ring $\zz[x]$. Polynomials $P(x) \in \zz[x]$ satisfying $P\big(x^{2^m}\big) \equiv P(x)^{2^m} \pmod{2^{m+1}}$ for some $m \in \nn$ play a central role in this section and in the arguments we use to establish our first main result, Theorem~\ref{thm:main}. Given their relevance, we reserve a term for them.

\begin{definition}
	For $n \in \nn_0$, we say that a polynomial $P(x) \in \zz[x]$ is $n$-\emph{good} if the congruence relation $P\big(x^{2^m}\big) \equiv P(x)^{2^m} \pmod{2^{m+1}}$ holds for every $m \in \ldb 0,n \rdb$.
\end{definition}

\noindent Observe that every polynomial in $\zz[x]$ is $0$-good. Also, it follows directly from the definition that if a polynomial in $\zz[x]$ is $n$-good for some $n \in \nn_0$, then it is also $k$-good for every $k \in \ldb 0,n\rdb$. This subsection is devoted to gather information related to $n$-good polynomials.

\begin{lemma} \label{lem:monomial composition keep n-goodness}
	Fix $n \in \nn$, and let $P(x)$ be a polynomial in $\zz[x]$. Then, for each $\ell \in \nn$, the polynomial $P(x^\ell)$ is $n$-good if and only if $P(x)$ is $n$-good.
\end{lemma}

\begin{proof}
	Fix $\ell \in \nn$. For each $m \in \ldb 0, n \rdb$, set $Q_m(x) := P\big( x^{2^m} \big) - P(x)^{2^m}$, and then observe that the fact that $Q_m(x)$ and $Q_m(x^\ell)$ have the same set of nonzero coefficients implies that $2^{m+1} \mid_{\zz[x]} Q_m(x)$ if and only if $2^{m+1} \mid_{\zz[x]} Q_m(x^\ell)$, which in turn implies that $P\big( x^{2^m} \big) \equiv P(x)^{2^m}  \pmod{2^{m+1}}$ if and only if $P\big( \big( x^{2^m} \big)^\ell\big) = P\big( (x^\ell)^{2^m} \big) \equiv P(x^\ell)^{2^m}  \pmod{2^{m+1}}$. Therefore $P(x^\ell)$ is $n$-good if and only if $P(x)$ is $n$-good.
%	
%	Since $P(x)$ is $n$-good, we can write $P\big(x^{2^m}\big) = P(x)^{2^m} + 2^{m+1} Q(x)$ for some $Q[x] \in \zz[x]$. Using this equality, we obtain
%	\[
%		P\big( \big( x^{2^m}\big)^\ell \big) = P\big( ( x^\ell )^{2^m} \big) = P(x^\ell)^{2^m} + 2^{m+1} Q(x^\ell) \equiv P(x^\ell)^{2^m}  \pmod{2^{m+1}}.
%	\]
%	Thus, we conclude that $P(x^\ell)$ is also an $n$-good polynomial.
\end{proof}

%\begin{lemma} \label{lem:monomial composition keep n-goodness}
%	For $n \in \nn$, let $P(x)$ be an $n$-good polynomial in $\zz[x]$. Then $P(x^\ell)$ is also $n$-good for every $\ell \in \nn$.
%\end{lemma}
%
%\begin{proof}
%	Fix $\ell \in \nn$ and $m \in \ldb 0, n \rdb$. Since $P(x)$ is $n$-good, we can write $P\big(x^{2^m}\big) = P(x)^{2^m} + 2^{m+1} Q(x)$ for some $Q[x] \in \zz[x]$. Using this equality, we obtain
%	\[
%		P\big( \big( x^{2^m}\big)^\ell \big) = P\big( ( x^\ell )^{2^m} \big) = P(x^\ell)^{2^m} + 2^{m+1} Q(x^\ell) \equiv P(x^\ell)^{2^m}  \pmod{2^{m+1}}.
%	\]
%	Thus, we conclude that $P(x^\ell)$ is also an $n$-good polynomial.
%\end{proof}

\begin{lemma} \label{lem:n-good polynomial I}
	For $n \in \nn$, let $P(x)$ be an $n$-good polynomial in $\zz[x]$. Then $P(x^2) \equiv P(x)^2 \pmod{2^{n+1}}$.% for every $k \in \zz$.
\end{lemma}

\begin{proof}
	As $P(x)$ is an $n$-good polynomial, $P(x)^{2^{n-1}} \equiv P(x^{2^{n-1}}) \pmod{2^n}$ and, therefore, we can write $P(x)^{2^{n-1}} = P\big( x^{2^{n-1}} \big) + 2^n Q(x)$ for some $Q(x) \in \zz[x]$. This, along with the fact that $P(x)$ is an $n$-good polynomial, ensures that
	\[
		P\big( \big(x^{2^{n-1}}\big)^2 \big) \equiv \big( P(x)^{2^{n-1}} \big)^2 = \big( P\big( x^{2^{n-1}} \big) + 2^n Q(x) \big)^2 \equiv P\big( x^{2^{n-1}} \big)^2 \pmod{2^{n+1}}.
	\]
	After replacing $x^{2^{n-1}}$ by $x$ in the congruence relation $P\big( \big(x^{2^{n-1}}\big)^2 \big) \equiv P\big( x^{2^{n-1}} \big)^2 \pmod{2^{n+1}}$, we obtain that $P(x^2) \equiv P(x)^2 \pmod{2^{n+1}}$.
\end{proof}

The following lemma will be useful often, so we name it.

\begin{lemma}[Match-and-Lift Lemma] \label{lem:n-good polynomial II}
	For $n \in \nn_0$, let $P(x)$ and $Q(x)$ be two $n$-good polynomials in $\zz[x]$. If $P(x) \equiv Q(x) \pmod{2}$, then $P(x) \equiv Q(x) \pmod{2^{n+1}}$.
\end{lemma}

\begin{proof}
	We proceed by induction on $n \in \nn_0$. The base case $n = 0$ follows immediately. Now assume that the statement of the lemma holds with $n$ replaced by $n-1 \in \nn_0$. Let $P(x)$ and $Q(x)$ be two $n$-good polynomials in $\zz[x]$, and suppose that $P(x) \equiv Q(x) \pmod{2}$. Since $P(x)$ and $Q(x)$ are $n$-good, they are also $(n-1)$-good, and it follows from our induction hypothesis that $P(x) \equiv Q(x) \pmod{2^n}$. Now write $Q(x) = P(x) + 2^n R(x)$ for some $R(x) \in \zz[x]$. Using that both $P(x)$ and $Q(x)$ are $n$-good, we obtain that 
	\[
		P\big(x^{2^n}\big) \equiv P(x)^{2^n} \equiv (P(x) + 2^n R(x))^{2^n} = Q(x)^{2^n} \equiv Q\big(x^{2^n}\big) \pmod{2^{n+1}}.
	\]
	We can now replace $x^{2^n}$ by $x$ in $P\big(x^{2^n} \big) \equiv Q\big(x^{2^n} \big) \pmod{2^{n+1}}$ to obtain that $P(x) \equiv Q(x) \pmod{2^{n+1}}$. Hence the statement of the lemma also holds for $n$, which completes our inductive argument.
\end{proof}

Now we provide a version of the part~(1) of Square Lemma for $n$-good polynomials.

\begin{lemma} \label{lem:Square Lemma for n-good polynomials}
	For $n \in \nn_0$, let $P(x)$ be a monic $n$-good polynomial in $\zz[x]$ with even degree. If $P(x) = A(x^2) - xB(x^2)$ for some polynomials $A(x)$ and $B(x)$ in $\zz[x]$, then $2^n \mid_{\zz[x]} B(x)$.
\end{lemma}

\begin{proof}
	Let $d$ be the degree of $P(x)$, and take $b_0, \dots, b_d$ and $c_0, \dots, c_{2d}$ in $\zz$ such that $P(x) = \sum_{i=0}^d b_i x^i$ and $P(x)^2 = \sum_{i=0}^{2d} c_i x^i$. Now write $P(x) = A(x^2) - xB(x^2)$ for polynomials $A(x)$ and $B(x)$ in $\zz[x]$. Suppose, by way of contradiction, that $2^n \nmid_{\zz[x]} B(x)$. This implies that $2^n \nmid_{\zz[x]} xB(x^2)$, and so there is a maximum odd index $m$ such that $2^n \nmid b_m$. Observe that %Then consider the monomial $q_{d+m}x^{d+m}$ summand of $P(x)^2$. Note that
	\begin{equation} \label{eq:auxx}
			c_{d+m} = \sum_{i=m}^d b_i b_{d+m-i} = 2 \sum_{i=m}^{(d+m-1)/2} b_i b_{d+m-i} \equiv 2b_m b_d \pmod{2^{n+1}},
	\end{equation}
	where the congruence follows from the fact that $2^{n+1} \mid 2 b_i b_{d+m-i}$ for every $i \in \ldb m+1, (d+m-1)/2 \rdb$ (by the maximality of $m$). As $b_d = 1$ and $2^n \nmid b_m$, we see that $2^{n+1} \nmid 2 b_m b_d$, and so $2^{n+1} \nmid c_{d+m}$ in light of~\eqref{eq:auxx}. However, since $d+m$ is odd, the congruence relation $P(x)^2 \equiv P(x^2) \pmod{2^{n+1}}$ guarantees that $c_{d+m} \equiv 0 \pmod{2^{n+1}}$, which brings us to a contradiction.
\end{proof}

%We define a polynomial $P(x)$ to be \emph{nice} if it is monic, irreducible, has even degree, and $|P(0)| = 1$.

\medskip
%%%%%%%%%%%%%%%%%%%%%%%%%%%%%%
\subsection{Splitting Sequences and Nice Polynomials}

%Let $R$ be an integral domain. 
In this subsection we introduce and study splitting sequences, which are objects that we need in order to establish the first two main results of this paper, Theorem~\ref{thm:irreducible polynomials that are not atomic in $Q[M_1/2]$} and Theorem~\ref{thm:main}. For a polynomial $P(x)$ in $\zz[x]$, we call $P(x^2)$ the \emph{square lifting} of $P(x)$.
\smallskip

Let $P(x)$ be an irreducible polynomial in $\zz[x]$. Starting from $P(x)$ we will construct inductively a sequence $(P_n(x))_{n \ge 0}$ of polynomials with no two consecutive terms both reducible in $\zz[x]$. We start by setting $P_0(x) := P(x)$. For the inductive step, suppose that we have already found $P_0(x), \dots, P_n(x)$ for some $n \in \nn_0$ such that $P_{i-1}(x)$ and $P_i(x)$ are not both reducible in $\zz[x]$ for any $i \in \ldb 1,n \rdb$. If $P_n(x)$ is irreducible, then let $P_{n+1}(x)$ be the square lifting of $P_n(x)$; that is, set $P_{n+1}(x) := P_n(x^2)$. Otherwise, $P_n(x)$ is reducible and so $n \ge 1$, whence it follows from the inductive hypothesis that $P_{n-1}(x)$ is irreducible. Since $P_n(x) = P_{n-1}(x^2)$, Square Lemma guarantees that $P_n(x)$ factors into two irreducible polynomials. If $P_{n+1}(x)$ is one such irreducible polynomial, then we say that $P_{n+1}(x)$ is a \emph{splitting} of $P_n(x)$. Observe that still there are no two consecutive reducible polynomials in the finite sequence $P_0(x), \dots, P_{n+1}(x) \in \zz[x]$. This concludes our inductive argument, yielding a sequence $(P_n(x))_{n \ge 0}$ of polynomials in $\zz[x]$ with $P_0(x) = P(x)$ and no two consecutive reducible terms.

\begin{definition}
	With notation as in the previous paragraph, we call $(P_n(x))_{n \ge 0}$ a \emph{splitting sequence} of the polynomial $P(x)$ or simply a \emph{splitting sequence}.
\end{definition}

It is worth emphasizing that, in some of the arguments we will provide, it is more convenient or natural to consider splitting sequences enumerated by $\nn$ instead of $\nn_0$, which means that we will often consider a splitting sequence $(P_n(x))_{n \ge 1}$ of an irreducible polynomial $P(x) \in \zz[x]$, tacitly assuming that $P(x)$ is the first term of the sequence, that is, $P(x) = P_1(x)$. 

Hence every irreducible polynomial in $\zz[x]$ has a splitting sequence. We encapsulate some initial observations about splitting sequences in the following lemma.

\begin{lemma} \label{lem:spliting sequences}
	For a splitting sequence $(P_n(x))_{n \ge 0}$, the following statements hold. 
	\begin{enumerate}
		
		\item For every $n \in \nn$, 
		\begin{itemize}
			\item $\deg P_n(x) = 2 \deg P_{n-1}(x)$ if $P_n(x)$ is the square lifting of $P_{n-1}(x)$, and
			\smallskip
			
			\item $\deg P_n(x) = \frac12 \deg P_{n-1}(x)$ if $P_n(x)$ is a splitting of $P_{n-1}(x)$.
		\end{itemize}
		\smallskip
		
		\item If $P_0(x)$ has even degree, then every $P_n(x)$ has even degree.
	\end{enumerate}
\end{lemma}

\begin{proof}
	(1) Fix $n \in \nn$. Clearly, $\deg P_n(x) = 2 \deg P_{n-1}(x)$ when $P_n(x)$ is the square lifting of $P_{n-1}(x)$. Assume, on the other hand, that $P_n(x)$ is a splitting of $P_{n-1}(x)$, which must be reducible. In this case, $n \ge 2$ and, as $(P_n(x))_{n \ge 0}$ does not have two consecutive reducible polynomials, $P_{n-2}(x)$ is irreducible. Then it follows from part~(2) of Square Lemma that $ \deg P_n(x) = \frac12 \deg P_{n-2}(x^2) = \frac12 \deg P_{n-1}(x)$.
	\smallskip
	
	(2) This follows from part~(1) and the fact that the sequence $(P_n(x))_{n \ge 0}$ cannot contain two consecutive terms that are reducible polynomials in $\zz[x]$.
\end{proof}

Let $(P_n(x))_{n \ge 0}$ be a splitting sequence. For each $n \in \nn_0$, choose $s_n$ from the $2$-letter alphabet $\{L,S\}$ as follows: $s_n = L$ if $P_n(x)$ is irreducible in $\zz[x]$ and $s_n = S$ otherwise (the letters $L$ and $S$ stand for `lift' and `split', respectively). Now concatenate all the terms in the sequence $(s_n)_{n \ge 0}$ into an infinite binary string~$\mathfrak s$.

\begin{definition}
    With notation as in the previous paragraph, we call $\mathfrak{s}$ the \emph{binary string} of the splitting sequence $(P_n(x))_{n \ge 0}$.
\end{definition}

% We observe that an irreducible polynomial $P(x)$ in $\zz[x]$ may have splitting sequences whose corresponding binary strings are different: indeed, different choices of the second term of the splitting sequence of $P(x)$ may result in different values of the corresponding binary string.

Let us take a look at some examples.

\begin{example}\label{ex:splitting sequences and binary strings} \hfill
	\begin{enumerate}
		\item Let $(P_n(x))_{\ge 1}$ be a splitting sequence of the irreducible monomial $x$. Then $P_1(x) = x$ and $P_2(x) = x^2$. We can argue inductively that $P_{2n-1}(x) \in \{\pm x\}$ and $P_{2n}(x) = x^2$ for every $n \in \nn$. Hence the binary string of $(P_n(x))_{n \ge 1}$ is the string that starts with a copy of $L$ and alternates between $L$ and $S$.
		\smallskip 
		
		\item Fix $p \in \pp$, and consider the irreducible binomial $x+p$. Let $(P_n(x))_{n \ge 0}$ be a splitting sequence of $x+p$. By Eisenstein's criterion, $x^{2^n} + p$ is irreducible in $\zz[x]$ for every $n \in \nn$. Hence the equality $P_n(x) = x^{2^n} + p$ holds for every $n \in \nn_0$, and so $x+p$ has a unique splitting sequence, whose binary string has a copy of $L$ in every position.
		\smallskip
		
		\item For an even positive integer $k$, let $(P_n(x))_{n \ge 0}$ be a splitting sequence of the $k$-th cyclotomic polynomial $\Phi_k(x)$. For each $n \in \nn_0$, the equality $\Phi_k\big(x^{2^n} \big) = \Phi_{2^nk}(x)$ implies that $\Phi_k\big(x^{2^n} \big)$ is irreducible, and so $P_n(x) = \Phi_{2^nk}(x)$. Hence $\Phi_k(x)$ has a unique splitting sequence, whose binary string has a copy of $L$ in every position.
	\end{enumerate}
\end{example}

We are especially interested in the splitting sequences of polynomials whose leading and constant coefficients belong to $\{\pm 1\}$.

\begin{definition}
	A polynomial in $\zz[x]$ is \emph{nice} if it is irreducible and both its leading and constant coefficients belong to $\{\pm 1\}$.
\end{definition}

We proceed to collect some preliminary remarks about the connection between splitting sequences and nice polynomials.

\begin{lemma}  \label{lem:splitting sequences and NP}
	Let $(P_n(x))_{n \ge 1}$ be a splitting sequence such that $P_1(x)$ is not a monomial. Then the following statements hold.
	\begin{enumerate}
		\item If $P_1(x)$ is a nice polynomial, then for each $n \in \nn$ the polynomial $P_n(x)$ is nice if and only if it is irreducible.
		\smallskip
				
		\item If the binary string of $(P_n(x))_{n \ge 1}$ contains infinitely many copies of $S$, then $P_1(x)$ is a nice polynomial (and so for each $n \in \nn$, the polynomial $P_n(x)$ is nice if and only if it is irreducible).
	\end{enumerate}
\end{lemma}

\begin{proof}
	(1) This part follows immediately from the fact that the property of having the leading and constant coefficients in $\{\pm 1\}$ is preserved by both taking square lifting and choosing a splitting. %Therefore the equality $|P_1(0)| = 1$ implies that $|P_n(0)| = 1$ for every $n \in \nn$. This implies that for each $n \in \nn$ the polynomial $P_n(x)$ is nice if and only if it is monic and irreducible.
	\smallskip
	
	(2) Suppose that the binary string of $(P_n(x))_{n \ge 1}$ contains infinitely many copies of $S$. First, observe that taking square lifting respects both the leading and constant coefficients of any polynomial. Now fix $n \in \nn$ such that $P_n(x)$ is not irreducible in $\zz[x]$. If $P_n(x) = A(x)B(x)$ is the splitting in part~(2) of Square Lemma, then for every $k \in \nn$ the coefficients of $x^k$ in both $A(x)$ and $B(x)$ have the same magnitude. Thus, the magnitudes of the leading and constant coefficients of $P_{n+1}(x)$ are the square roots of the magnitudes of the leading and constant coefficients of $P_n(x)$, respectively. Since $P_1(x)$ is not a monomial, the fact that the binary string of $(P_n(x))_{n \ge 1}$ contains infinitely many copies of~$S$ guarantees that both the leading and constant coefficients of $P_1(x)$ belong to $\{\pm 1\}$, whence $P_1(x)$ is a nice polynomial. The last parenthetical observation is an immediate consequence of part~(1).
\end{proof}

%\textcolor{red}{TODO: Does the binary string of $P(x)$ depend on the chosen splitting sequence?} %describing the degrees of $P_i(x)$. 
%For each $n \in \nn$, let $a_n$ be the index of the $n$-th $S$ in the string $\mathfrak s$ if such an $S$ exists, and set $b_n := a_{n+1} - a_n$. Observe that no two $S$'s are consecutive in $\mathfrak s$ because, after splitting, both of the resulting factors are irreducible. Thus, $b_n \ge 2$ for every $n \in \nn$. 
In the following lemmas, we learn how to decode certain divisibility aspects of $(P_n(x))_{n \ge 1}$ from its binary string.

\begin{lemma} \label{lem:splitting dividing composition power}
	Let $(P_n(x))_{n \ge 0}$ be a splitting sequence, and let $s_0 s_1 \ldots$ be the binary string of $(P_n(x))_{n \ge 0}$. Fix $m \in \nn_0$ such that the polynomial $P_m(x)$ is irreducible. Then $P_{m+k}(x) \mid_{\zz[x]} P_m\big(x^{2^\ell} \big)$ for all $k \in \nn_0$, where $\ell = |\{i \in \ldb m, m+k-1 \rdb : s_i = L\}|$. 
\end{lemma}

\begin{proof}
	Since $P_m(x)$ is irreducible, $(P_n(x))_{n \ge m}$ is also a splitting sequence and, therefore, we can safely assume that $m=0$. For each $k \in \nn_0$, set $\ell_k = |\{i \in \ldb 0, k-1 \rdb : s_i = L\}|$. We proceed by induction.
	
	 If $k=0$, then $\ell_k = 0$, and so it trivially holds that $P_0(x) \mid_{\zz[x]} P_0\big(x^{2^{\ell_0}} \big)$. In addition, if $k=1$, then $\ell_k = 1$. In this case, $P_1(x) = P_0(x^2) = P_0\big( x^{2^{\ell_1}}\big)$, and so $P_1(x) \mid_{\zz[x]} P_0\big( x^{2^{\ell_1}}\big)$, also in a trivial way. For the inductive step, fix $k \ge 1$, and suppose that $P_j(x) \mid_{\zz[x]} P_0\big( x^{2^{\ell_j}}\big)$ for every $j \in \ldb 0, k \rdb$. We split the rest of the proof into two cases.
	 \smallskip
	 
	 \textsc{Case 1:} $s_k = L$. In this case, $\ell_{k+1} = \ell_k + 1$ and $P_{k+1}(x) = P_k(x^2)$. By the induction hypothesis, $P_k(x) \mid_{\zz[x]} P_0\big( x^{2^{\ell_k}} \big)$, and so $P_{k+1}(x) = P_k(x^2) \mid_{\zz[x]} P_0\big( x^{2^{\ell_k + 1}} \big) = P_0\big( x^{2^{\ell_{k+1}}} \big)$.
	 \smallskip
	 
	 \textsc{Case 2:} $s_k = S$. In this case, $k \ge 1$ and $s_{k-1} = L$. Therefore $\ell_{k+1} = \ell_{k-1} + 1$. Observe that $P_{k+1}(x)$ is a splitting of $P_k(x)$, which is in turn a square lifting of the irreducible polynomial $P_{k-1}(x)$. Thus, $P_{k+1}(x) \mid_{\zz[x]} P_k(x) = P_{k-1}(x^2)$. By the induction hypothesis, $P_{k-1}(x) \mid_{\zz[x]} P_0\big( x^{2^{\ell_{k-1}}} \big)$, which implies that  $P_{k-1}(x^2) \mid_{\zz[x]} P_0\big( x^{2^{\ell_{k-1}+1}} \big) = P_0\big( x^{2^{\ell_{k+1}}} \big)$. Hence $P_{k+1}(x) \mid_{\zz[x]} P_0\big( x^{2^{\ell_{k+1}}} \big)$.
\end{proof}

\begin{lemma} \label{lem:string feature I}
	Let $(P_n(x))_{n \ge 1}$ be the splitting sequence of a nice polynomial, and suppose that the binary string $s_1 s_2 \ldots$ of $(P_n(x))_{n \ge 1}$ satisfies $s_i s_{i+1} = LS$ for some $i \in \nn$. If $P_i(x)$ is $(k-1)$-good, and $P_{i+2}(x)$ is $k$-good for some $k \in \nn$, then the following statements hold.
	\begin{enumerate}
		\item $P_i(x)$ is also $k$-good.
		\smallskip
		
		\item $P_i(x) \equiv P_{i+2}(x) \pmod{2^{k+1}}$.
	\end{enumerate}
\end{lemma}

\begin{proof}
	(1) As $s_i = L$, the polynomial $P_i(x)$ is irreducible while its square lifting $P_i(x^2)$ is reducible. Then it follows from Lemma~\ref{lem:splitting sequences and NP} that $P_i(x)$ is a nice polynomial. Thus, after dropping the first $i-1$ terms of the sequence $(P_n(x))_{n \ge 1}$, we can assume that $i=1$. %$P_1(x) = P_i(x)$, and $P_2(x) = P_{i+2}(x)$. 
	
	 Since $P_1(x)$ is $(k-1)$-good, $P_1\big( x^{2^{k-1}} \big) \equiv P_1(x)^{2^{k-1}} \pmod{2^k}$ holds. On the other hand, the equality $s_2 = S$ means that $P_2(x)$ is reducible, and so in light of Square Lemma we can take $A(x), B(x) \in \zz[x]$ such that $P_1(x)$ is associate to $A(x)^2 - xB(x)^2$ and $P_3(x) \in \{ A(x^2) \pm xB(x^2) \}$. Because $P_1(x)$ is associate to $A(x)^2 - xB(x)^2$, it follows that $P_1(x) \equiv A(x)^2 - xB(x)^2 \pmod{2}$. Also, after replacing $B(x)$ by $-B(x)$ if necessary, we can assume that $P_3(x) = A(x^2) - xB(x^2)$. Using the fact that $P(x)^2 \equiv P(x^2) \pmod{2}$ for all $P(x) \in \zz[x]$, we obtain
	\[
		P_1(x) \equiv A(x)^2 - xB(x)^2 \equiv A(x^2) - xB(x^2) = P_3(x) \pmod{2}.
	\]
	This, in tandem with the fact that $P_1(x)$ and $P_3(x)$ are both $(k-1)$-good polynomials, implies that $P_1(x) \equiv P_3(x) \pmod{2^k}$ by virtue of the Match-and-Lift Lemma. Now we need to establish the following claim.
	\smallskip
	
	\noindent \textsc{Claim.} $\deg P_3(x)$ is even.
	\smallskip
	
	\noindent \textsc{Proof of Claim.} Assume for the sake of contradiction that $P_3(x)$ has odd degree. As $P_3(x)$ is $k$-good, it is in particular $1$-good and, therefore, $P_3(x^2) \equiv P_3(x)^2 \pmod{4}$. Write $P_3(x) = Q(x^2) + xR(x^2)$ for some polynomials $Q(x)$ and $R(x)$ in $\zz[x]$. Since $P_3(x)$ is nice, the constant coefficient of $Q(x^2)$ belongs to $\{ \pm 1 \}$: let $d_1$ be the maximum degree of a monomial of $Q(x^2)$ whose coefficient is odd. Because $P_3(x)$ has odd degree, $\deg P_3(x) = \deg xR(x^2)$: set $d_2 := \deg xR(x^2)$. Observe now that the only summand in the right-hand side of the equality
	\begin{equation} \label{auxin}
		P_3(x)^2 = Q(x^2)^2 + 2xQ(x^2)R(x^2) + x^2R(x^2)^2
	\end{equation}
	that involves monomials of odd degree is $2xQ(x^2)R(x^2)$. In addition, we see that the coefficient of the monomial of degree $d_1 + d_2$ in $2xQ(x^2)R(x^2)$ belongs to $2 + 4\zz$. Now the fact that neither $Q(x^2)^2$ nor $x^2R(x^2)^2$ in~\eqref{auxin} involve monomials of odd degree implies that $P_3(x)^2$ has a monomial of odd degree (namely, $d_1 + d_2$) whose coefficient belongs to $2 + 4\zz$, and so the same is true for the polynomial $P_3(x)^2 - P_3(x^2)$ because $P_3(x^2)$ involves no monomial of odd degree. However, this contradicts that $P_3(x^2) \equiv P_3(x)^2 \pmod{4}$.
	\smallskip

	In light of the established claim, $P_3(x)$ is a monic $k$-good polynomial with even degree. Thus, it follows from Lemma~\ref{lem:Square Lemma for n-good polynomials} that  $2^k \mid_{\zz[x]} B(x)$, which in turn implies that $2^k \mid_{\zz[x]} B(x^n)$ for every $n \in \nn$. Using again the fact that $P_3(x)$ is $k$-good, we see that for each $n \in \ldb 0,k \rdb$,
	\[
		A\big( x^{2^{n+1}} \big) - x^{2^n} B\big( x^{2^{n+1}} \big) \equiv P_3\big(x^{2^n}\big) \equiv P_3(x)^{2^n} \equiv (A(x^2) - xB(x^2))^{2^n} \equiv A(x^2)^{2^n} \pmod{2^{n+1}},
	\]
	where the last congruence relation holds because $2^n \mid_{\zz[x]} B(x^2)$. Thus, for each $n \in \ldb 0, k \rdb$, the congruence $A\big( x^{2^{n+1}} \big) - x^{2^n} B\big( x^{2^{n+1}} \big) \equiv A(x^2)^{2^n} \pmod{2^{n+1}}$ holds, and from this (after replacing $x^2$ by $x$) we obtain
	\begin{equation} \label{eq:tempp}
		A\big(x^{2^n}\big) - x^{2^{n-1}} B\big(x^{2^n}\big) \equiv A(x)^{2^n} \pmod{2^{n+1}}.
	\end{equation}

	We proceed to verify that $A\big(x^{2^{k-1}}\big) \equiv A(x)^{2^{k-1}}\pmod{2^k}$. This congruence clearly holds if $k=1$. Therefore we assume (in the scope of this paragraph) that $k \ge 2$. %, it is clear that $A\big( x^{2^{k-1}}\big) \equiv A(x)^{2^{k-1}} \pmod{2^k}$. On the other hand, when $k \ge 2$, after carrying out the obvious substitution in~\eqref{eq:tempp}, we see that 
	Taking $n$ to be $k-1$ in~\eqref{eq:tempp}, we see that
	\[
		A\big( x^{2^{k-1}}\big) - x^{2^{k-2}} B\big( x^{2^{k-1}} \big) \equiv A(x)^{2^{k-1}} \pmod{2^k},
	\]
	and so the fact that $2^k \mid_{\zz[x]} B\big( x^{2^{k-1}}\big)$ ensures that $A\big( x^{2^{k-1}}\big) \equiv A(x)^{2^{k-1}} \pmod{2^k}$, which is the congruence we wished to verify.
 
    Take a polynomial $D(x)$ in $\zz[x]$ such that $A(x)^{2^{k-1}} = A\big(x^{2^{k-1}}\big) + 2^k D(x)$. Using this and~\eqref{eq:tempp}, we observe that
	\[
		A\big(x^{2^k}\big) - x^{2^{k-1}} B\big(x^{2^k}\big) \equiv A(x)^{2^k} = \big( A\big(x^{2^{k-1}}\big) + 2^k D(x) \big)^2 \equiv A\big(x^{2^{k-1}}\big)^2 \pmod{2^{k+1}}.
	\]
	Now we can replace $x^{2^{k-1}}$ by $x$ in $A\big(x^{2^k}\big) - x^{2^{k-1}} B\big(x^{2^k}\big) \equiv A\big(x^{2^{k-1}}\big)^2 \pmod{2^{k+1}}$ to obtain the congruence relation $A(x^2) - xB(x^2) \equiv A(x)^2 \pmod{2^{k+1}}$. In addition, $2^k \mid_{\zz[x]} B(x)$ implies that  $2^{k+1} \mid_{\zz[x]} xB(x)^2 = A(x)^2 - P_1(x)$. After putting the last two statements together, we find that
	\begin{equation} \label{eq:part 2 of lemma}
		P_3(x) = A(x^2) - xB(x^2) \equiv A(x)^2 \equiv P_1(x) \pmod{2^{k+1}}.
	\end{equation}
	Because of this, we can finally infer that $P_1(x)$ is $k$-good from the fact that $P_3(x)$ is $k$-good.
	\smallskip
	
	(2) We have already seen this in~\eqref{eq:part 2 of lemma}.
\end{proof}

\begin{lemma}  \label{lem:string feature II}
	Let $(P_n(x))_{n \ge 0}$ be the splitting sequence of a nice polynomial, and suppose that the binary string $s_0 s_1 \ldots$ of $(P_n(x))_{n \ge 0}$ satisfies $s_{i-1} s_i s_{i+1} = LLS$ for some $i \in \nn$. If $P_i(x)$ and $P_{i+2}(x)$ are $k$-good for some $k \in \nn$, then $P_i(x)$ is $(k+1)$-good.
\end{lemma}

\begin{proof}
	We can reuse the argument given in the first paragraph of the proof of Lemma~\ref{lem:string feature I} to assume that $i=1$. In this case, $s_0 s_1 s_2 = LLS$ and the polynomials $P_1(x)$ and $P_3(x)$ are $k$-good. %For simplicity, let $P_1(x) = P_i(x)$, and $P_2(x) = P_{i+2}(x)$. 
	
	It follows from part~(2) of Lemma~\ref{lem:string feature I} that $P_1(x) \equiv P_3(x) \pmod{2^{k+1}}$. In addition, the equality $s_1 s_2 = LS$ guarantees that $P_1(x)$ is irreducible and $P_1(x^2)$ is reducible, and so by Square Lemma we can take $A(x)$ and $B(x)$ in $\zz[x]$ such that $P_1(x) = A(x)^2 - xB(x)^2$ and $P_3(x) = A(x^2) - xB(x^2)$ (as we did in the proof of Lemma~\ref{lem:string feature I}). Since $s_0 = L$, we see that $P_1(x) = P_0(x^2)$. Therefore
	\[
		A(x^2) - xB(x^2) = P_3(x) \equiv P_1(x) = P_0(x^2) \pmod {2^{k+1}}.
	\]
	As $2^{k+1} \mid_{\zz[x]} xB(x^2) + (P_0(x^2) - A(x^2))$, the fact that the polynomials $xB(x^2)$ and $P_0(x^2) - A(x^2)$ have disjoint supports implies that $2^{k+1}$ divides in $\zz$ every coefficient of $xB(x^2)$. Thus, $2^{k+1} \mid_{\zz[x]} xB(x^2)$. Take $S(x) \in \zz[x]$ such that $xB(x^2) = 2^{k+1}S(x)$. Since $P_1(x) \equiv P_3(x) = A(x^2) - xB(x^2) \equiv A(x^2) \pmod{2^{k+1}}$, we can also take $T(x) \in \zz[x]$ such that $P_1(x) + 2^{k+1}T(x) = A(x^2)$. Then
	\[
		P_1(x^2) = A(x^2)^2 - x^2 B(x^2)^2 = (P_1(x) + 2^{k+1}T(x))^2 - (2^{k+1}S(x))^2 \equiv P_1(x)^2 \pmod{2^{k+2}}.
	\]
	Replacing $x$ by $x^{2^k}$ in $P_1(x^2) \equiv P_1(x)^2 \pmod{2^{k+2}}$, we obtain $P_1\big(x^{2^k}\big)^2 \equiv P_1\big(x^{2^{k+1}}\big) \pmod{2^{k+2}}$. As $P_1(x)$ is $k$-good, we can take $U(x) \in \zz[x]$ such that $P_1(x)^{2^k} = P_1\big(x^{2^k}\big) + 2^{k+1}U(x)$. After putting together the last two statements, we see that
	\[
		P_1(x)^{2^{k+1}} \equiv \big( P_1(x^{2^k}) + 2^{k+1} U(x) \big)^2 \equiv P_1\big(x^{2^k}\big)^ 2  \equiv P_1\big(x^{2^{k+1}}\big) \pmod{2^{k+2}}.
	\]
	This, together with the fact that $P_1(x)$ is $k$-good, allows us to conclude that $P_1(x)$ is a $(k+1)$-good polynomial.
\end{proof}

Let $(P_n(x))_{n \ge 1}$ be a splitting sequence of a nice polynomial, and let $\mathfrak{s} = s_1 s_2 \ldots$ be the binary string of $(P_n(x))_{n \ge 1}$. For each $n \in \nn$, let $a_n$ be the index of the $n$-th $S$ in the string $\mathfrak s$ if such an $S$ exists. Now set $a_0 = 0$ and $b_n := a_n - a_{n-1}$ for each $n \in \nn$. Observe that, by virtue of Square Lemma, no two $S$'s are consecutive in $\mathfrak s$. Thus, $b_n \ge 2$ for every $n \in \nn$.

\begin{definition}
	With notation as in the previous paragraph and $k \in \nn_0$, we say that a nice polynomial $P(x) \in \zz[x]$ is $k$-\emph{special} if $P(x)$ has a splitting sequence $(P_n(x))_{n \ge 1}$ such that there exist $k$ distinct indices $i_1, \dots, i_k \in \nn$ with $b_{i_n} > 2$ for every $n \in \ldb 1,k \rdb$.
\end{definition} 

%\color{red}
%
%The argument as is does actually fail here ($P_{m+2}$ does not actually have to be $n$-good), but we claim that this is just a result of faulty definitions. We modify the paragraph before Definition 3.16 as follows: Set $a_0 = 0$ and define $b_n := a_n - a_{n-1}$ instead. We define $k$-special polynomials analogously.
%
%The reason we make this modification is because $k$-specialness should be interpreted as “there remain at least $k$ LLS operations in the splitting sequence of this polynomial.” With the current definition, if the sequence were LSLSLLSLSLLS, you would detect two LLS as desired (at the previous $n = 2,4$) but if the sequence were LLSLSLSLLS, you would miss the first LLS because no S appears before it, and we would only claim the polynomial is $1$-special (at the previous $n = 3$) when in principle this polynomial should be $2$-special since it has two LLS subsequences. The proposed modification fixes this issue.
%
%This fixes the weird indexing issues that arise such as the $[1, \ell-2]$, $[1, m+1]$, $[2, m+1]$ kinds of boundary index shifts and should also fix the issue with $P_{m+2}$. Again, the underlying proof works, just the definition as written doesn’t capture the correct understanding of $k$-specialness. With the fix in the indices the rest of the proof should fix naturally.
%
%\color{black}

Observe that for $k \in \nn_0$, a nice polynomial $P(x) \in \zz[x]$ is $k$-special if there are at least $k$ copies of the string $LLS$ in one of the binary strings of $P(x)$. It follows directly from the definition that a $k$-special polynomial is also $j$-special for every $j \in \ldb 0, k \rdb$. We can use Lemmas~\ref{lem:string feature I} and~\ref{lem:string feature II} to argue the following proposition.

%\newpage
\begin{prop} \label{prop:nice polynomials}
	For a nice polynomial $P(x) \in \zz[x]$, the following statements hold.
	\begin{enumerate}
		\item If $P(x)$ is $k$-special, then $P(x)$ is $k$-good.
		\smallskip
		
		\item There exists $k \in \nn$ such that $P(x)$ is not $k$-special.
	\end{enumerate}
\end{prop}

\begin{proof}
	(1) We proceed by induction on~$k$. The base case $k=0$ is trivial as every polynomial in $\zz[x]$ is $0$-good. Now assume that the statement of the proposition holds when $k = n$ for some $n \in \nn_0$. Suppose that $P(x)$ is $(n+1)$-special. Let $(P_n(x))_{n \ge 1}$ be a splitting sequence of $P(x)$ making it an $(n+1)$-special polynomial, and let $\mathfrak{s} = s_1 s_2 \ldots$ be the binary string of $(P_n(x))_{n \ge 1}$. Because the sequence $(P_n(x))_{n \ge 1}$ makes $P(x)$ an $(n+1)$-special polynomial, $\mathfrak{s}$ must contain at least one copy of $LLS$ as a substring. Thus, there is a minimum index $m \in \nn_{\ge 2}$ such that $s_{m-1} s_m s_{m+1} = LLS$. Now we let $\ell \in \nn$ be the minimum index such that $s_i = L$ for every $i \in \ldb \ell, m \rdb$. The minimality of both~$\ell$ and~$m$, along with the fact that $\mathfrak{s}$ cannot contain a copy of $SS$ as a substring, ensures that either $\mathfrak{s}$ starts with a copy of the sting $LLS$ or $\ell \ge 3$ and the first $\ell-1$ letters of $\mathfrak{s}$ alternate between $L$ and $S$, with $s_1 = L$ because $P(x)$ is irreducible (i.e., $s_1 s_2 \dots s_{\ell-1} = LSLS \dots LS$). 
	
	For each $i \in \nn$, it follows from Lemma~\ref{lem:splitting sequences and NP} that $P_i(x)$ is a nice polynomial if $s_i = L$. Thus, for each $i \in \ldb 1,m+2 \rdb$ with $s_i = L$, the minimality of $m$ guarantees that $P_i(x)$ is $(n+1)$-special when $i \in \ldb 1, m-1 \rdb$ and $P_i(x)$ is $n$-special when $i \in \ldb m,m+2 \rdb$. %the polynomial $P_i(x)$ is $n$-special for every $i \in \ldb 1, m+1 \rdb$ with $s_i = L$ (more precisely, when $s_i = L$, the polynomial $P_i(x)$ is $(n+1)$-special if $i \in \ldb 1, \ell-2 \rdb$ and $n$-special if $i \in \ldb \ell, m \rdb$). 
	It follows then from our induction hypothesis that, for each $i \in \ldb 1, m+2 \rdb$, the polynomial $P_i(x)$ is $n$-good provided that $s_i = L$. On the other hand, if $s_i = S$ for some $i \in \ldb 2,m+1 \rdb$, then $s_{i-1} = L$ and so $P_i(x) = P_{i-1}(x^2)$: as $P_{i-1}(x)$ is $n$-good, it follows from Lemma~\ref{lem:monomial composition keep n-goodness} that $P_i(x)$ is also $n$-good. Thus, $P_i(x)$ is $n$-good for every $i \in \ldb 1, m+2 \rdb$. In particular, $P_m(x)$ and $P_{m+2}(x)$ are $n$-good polynomials. %\textcolor{red}{(TODO: Need to argue that $P_{m+2}(x)$ is actually $n$-good, as it may not be $n$-special))}. 
	Therefore $P_m(x)$ is $(n+1)$-good in light of Lemma~\ref{lem:string feature II}. This, together with the equality $P_m(x) = P_\ell\big(x^{2^{m-\ell}}\big)$, implies that $P_\ell(x)$ is also $(n+1)$-good by virtue of Lemma~\ref{lem:monomial composition keep n-goodness}. Thus, Lemma~\ref{lem:string feature I} guarantees that $P_{\ell-2}(x), P_{\ell-4}(x), \dots, P_1(x)$ are all $(n+1)$-good. In particular, $P(x)$ is $(n+1)$-good, which concludes our inductive argument.
	\smallskip
	
	(2) Suppose, by way of contradiction, that $P(x)$ is $k$-special for every $k \in \nn$. Then part~(1) ensures that $P(x)$ is $k$-good for every $k \in \nn$. Thus, it follows from Lemma~\ref{lem:n-good polynomial I} that $2^{k+1} \mid P(x)^2 - P(x^2)$ for every $k \in \nn$, which implies that $P(x)^2 = P(x^2)$. Now set $c := P(2)$. An immediate inductive argument can be used to shows that $P(2^{2^n}) = c^{2^n}$ for every $n \in \nn$. As a consequence, $P(x) = c^{\log_2 x} = x^{\log_2 c}$ for infinitely many $x$. Hence $c$ must be a power of $2$ and, therefore, $P(x) = x^j$ for some $j \in \nn$. In this case, $|P(0)| \neq 1$, contradicting that $P(x)$ is nice.
\end{proof}

\medskip
%%%%%%%%%%%%%%%%%%%%%%%%%%%%%
\subsection{Connection with Cyclotomic Polynomials}

In this subsection we establish some connections between the splitting sequences introduced in the previous subsection and cyclotomic polynomials.

\begin{lemma} \label{lem:lemma 13}
	% Let $P(x)$ be an irreducible polynomial in $\zz[x]$ with splitting sequence $(P_n(x))_{n \ge 1}$ whose binary string contains infinitely many copies of $S$. %$s_1 s_2 \ldots$. If $s_k = S$ for infinitely many indices $k$ and 
	% If the set $\{\deg P_n(x) : n \in \nn\}$ is bounded, then $P(x) \in \{ \pm x, \pm \Phi_n(x) : n \in 2 \nn_0 + 1\}$. %for some odd integer $n$.

    Let $P(x)$ be an irreducible polynomial in $\zz[x]$, and let $(P_n(x))_{n \ge 1}$ be a splitting sequence of $P(x)$. If $\{\deg P_n(x) : n \in \nn\}$ is a bounded set, then the following statements hold.
    \begin{enumerate}
        \item The binary string of $(P_n(x))_{n \ge 1}$ contains infinitely many copies of $S$.
        \smallskip

        \item $P(x) \in \{ \pm x, \pm \Phi_n(x) : n \in 2 \nn_0 + 1\}$.
    \end{enumerate}
\end{lemma}

\begin{proof}
    Let $\mathfrak{s} := s_1 s_2 \ldots$ be the binary string of $(P_n(x))_{n \ge 1}$.
    \smallskip

    (1) Observe that if there existed $N \in \nn$ large enough so that $s_n = L$ for every $n \ge N$, then the inequality $P_{n+1}(x) > P_n(x)$ would hold for any $n \ge N$, which would contradict that $\{\deg P_n(x) : n \in \nn\}$ is a bounded set.
    \smallskip
    
	(2) After replacing $P_n(x)$ by its monic associate for every $n \in \nn$, one can assume that every term of the sequence $(P_n(x))_{n \ge 1}$ is monic. %We have seen in part~(1) of Example~\ref{ex:splitting sequences and binary strings} that the binary string of any splitting sequence of $x$ is that starting with a copy of $L$ and alternating between $L$ and $S$. 
    Assume that $P(x) \neq x$. Now it follows from Lemma~\ref{lem:splitting sequences and NP} that, for each $n \in \nn$, the polynomial $P_n(x)$ is nice if and only if $s_n = L$. Thus, $|P_n(0)| = 1$ for every $n \in \nn$. In particular, $P(x)$ is a nice polynomial and $|P(0)| = 1$.
 
	Let us argue first that the underlying set of $(P_n(x))_{n \ge 1}$ is finite. As $\mathfrak{s}$ does not have $SS$ as a substring, $s_i = L$ for infinitely many indices~$i$. Let ~$k$ be the largest magnitude of any root of $P(x)$. Since the magnitude of the product of all the roots of $P(x)$ equals $1$, we see that $k \ge 1$. Notice that the roots of $P(x^2)$ are the complex square roots of those of $P(x)$, with two square roots for each root of $P(x)$. As $k \ge 1$, the sequence $\big(k^{1/2^n}\big)_{n \ge 0}$ is decreasing. Thus, the largest magnitude of any root of a polynomial in the sequence $(P_n(x))_{n \ge 1}$ is upper bounded by $k$. Fix $n \in \nn$ and write
	\[
		P_n(x) = x^{d_n} + \sum_{j=0}^{d_n - 1} c_{n,j} x^j = \prod_{i=1}^{d_n}(x - r_i) %= x^{d_n} + c_{n,1} x^{d-1} + c_{n,2} x^{d-2} + \cdots + c_{n,d_n},
	\]
	for some $c_{n,0}, \dots, c_{n, d_n - 1} \in \zz$, setting $d_n := \deg P_n(x)$ and letting $r_1, \dots, r_n$ denote the complex roots of $P_n(x)$. Now observe that, for every $j \in \ldb 1, d_n \rdb$,
	\[
		|c_{n,d_n-j}| = \bigg{|}  \sum_{1 \le i_1 < \dots < i_j \le d_n} (-1)^j r_{i_1} \cdots r_{i_j} \bigg{|} \ \le \! \! \sum_{1 \le i_1 < \dots < i_j \le d_n} |r_{i_1} \cdots r_{i_j}|  \le \binom{d_n}{j} k^j.
	\]
	Since the underlying set of the sequence $(d_n)_{n \ge 1}$ is bounded by hypothesis, the set
    \[
        D := \bigg\{ \binom{d_n}{j} : n \in \nn \text{ and } j \in \ldb 0, d_n \rdb \bigg\}
    \]
    is also bounded. Therefore the set of all possible magnitudes of coefficients of polynomials in $(P_n(x))_{n \ge 1}$ is bounded by $k^d \max D$, where $d := \max \{d_n : n \in \nn\}$. Because each coefficient of any term of $(P_n(x))_{n \ge 1}$ is an integer, the fact that $(d_n)_{n \in \nn}$ is bounded guarantees that the underlying set of $(P_n(x))_{n \ge 1}$ is finite, as desired.
	\smallskip

    (2) Let us argue now that $P(x)$ is a cyclotomic polynomial. As the underlying set of the sequence $(P_n(x))_{n \ge 1}$ is finite, and $\mathfrak{s}$ does not contain $SS$ as a substring, we can pick indices $i,j \in \nn$ with $i < j$ such that $P_i(x)$ is irreducible and $P_i(x) = P_j(x)$. Set
	\[
		\ell := |\{w \in \ldb i,j-1 \rdb : s_w = L\}|.
	\]
	Since $i+1 < j$ and $\mathfrak{s}$ does not contain $SS$ as a substring, $\ell \ge 1$. %terms of $(P_n(x))_{n \ge 1}$ which are equal. Let $P(x)$ and $Q(x)$ be the earlier and the latter of such two polynomials. 
	It follows from Lemma~\ref{lem:splitting dividing composition power} that $P_j(x) \mid_{\zz[x]} P_i(x^{2^\ell})$. %, where $\ell$ is the number of lifts between $P_i(x)$ and $P_j(x)$ in the sequence. It is clear that $\ell \ge 1$. 
	Thus, $P_i(x) \mid_{\zz[x]} P_i(x^{2^\ell})$. Let $r$ be the root of $P_i(x)$ of largest magnitude and set $m := |r|$. Since the magnitude of the product of all the roots of $P_i(x)$ is $1$, the inequality $m \ge 1$ holds. If $s$ is a root of $P_i(x^{2^\ell})$, then $s^{2^\ell}$ is a root of $P_i(x)$, and so $|s|^{2^\ell} = |s^{2^\ell}| \le m$. Thus, every root of $P_i(x^{2^\ell})$ has magnitude at most $m^{1/2^\ell}$. Now, as $P_i(x) \mid_{\zz[x]} P_i(x^{2^\ell})$, it follows that $r$ is also a root of $P_i(x^{2^\ell})$. Therefore $m = |r| \le m^{1/2^\ell}$. Since $m \ge 1$, we obtain that $m = 1$. Hence every root of $P_i(x)$ lies on the closed unit disk, and so it follows from Kronecker's theorem that $P_i(x)$ is a cyclotomic polynomial. Therefore $r$ is a root of unity. If $i=1$, then $P(x) = P_i(x)$, and we are done. Otherwise, it follows from Lemma~\ref{lem:splitting dividing composition power} that $P_i(x) \mid_{\zz[x]} P\big( x^{2^{\ell'}} \big)$, where $\ell' := |\{w \in \ldb 1,i-1 \rdb : s_w = L\}|$. Hence the root of unity $r^{2^{\ell'}}$ is a root of $P(x)$, and so the fact that $P(x)$ is monic and irreducible implies that $P(x) = \Phi_n(x)$ for some $n \in \nn$.
	
	Finally, we only need to verify that $n$ is odd. Indeed, we have seen in part~(3) of Example~\ref{ex:splitting sequences and binary strings} that for even $k \in \nn$ the cyclotomic polynomial $\Phi_k(x)$ has only one splitting sequence, whose binary string does not contain any copy of $S$.
%	
%	To do so, write $n = 2^a b$ for some $a \in \nn_0$ and $b \in \nn$ such that $2 \nmid b$. Observe that if $a \ge 1$, then for every $t \in \nn$ the fact that $P\big( x^{2^t}) = \Phi_{2^ab}\big( x^{2^t}\big) = \Phi_{2^{a+t}b}(x)$ implies that $P\big( x^{2^t} \big)$ is irreducible. Hence the only splitting sequence of $P(x)$ would be $\big( P\big( x^{2^n}\big)\big)_{n \ge 1}$, contradicting the existing of a splitting sequence of $R(x)$ with a binary string containing infinitely many copies of~$S$.
\end{proof}

\bigskip
%%%%%%%%%%%%%%%%%%%%%%%%%
%%%%%%%%%%%%%%%%%%%%%%%%%
\section{Applications to Factorization Theory}
\label{sec:a rank 1 example}

\smallskip
%%%%%%%%%%%%%%%%%%%%%%%%%%%%%%%%
\subsection{Factorizing Polynomials with Dyadic Exponents}

In this subsection, we determine the irreducible polynomials in $\zz[x]$ that are not atomic in the monoid algebra $\qq[M_{1/2}]$, where $M_{1/2}$ denotes the valuation monoid $\nn_0\big[\frac12 \big]$. First, we show that $\qq[M_{1/2}]$ is an AP-domain, and to do so we need the following lemma.

\begin{lemma} \label{lem:irreducible elements in $Q[M_{1/2}]$}
	A nonzero polynomial expression $P(x) \in \qq[M_{1/2}]$ is irreducible in $\qq[M_{1/2}]$ if and only if $P\big( x^{2^n} \big)$ is irreducible in $\qq[x]$ for every $n \in \nn$ such that $P\big( x^{2^n}\big) \in \qq[x]$.
\end{lemma}

\begin{proof}
	For the direct implication, suppose that $P(x)$ is irreducible in $\qq[M_{1/2}]$. Take $n \in \nn$ such that $P\big( x^{2^n}\big) \in \qq[x]$ and write $P\big( x^{2^n}\big) = A(x) B(x)$ for some $A(x), B(x) \in \qq[x]$. Since $P(x) = A\big( x^{1/2^n}\big) B\big( x^{1/2^n}\big)$ in $\qq[M_{1/2}]$, the irreducibility of $P(x)$ guarantees that either $A(x)$ or $B(x)$ is constant. Thus, $P\big( x^{2^n}\big)$ is irreducible in $\qq[x]$.
	
	For the reverse implication, suppose that for each $n \in \nn$, the fact that $P\big( x^{2^n}\big) \in \qq[x]$ implies that $P\big( x^{2^n}\big)$ is irreducible in $\qq[x]$. Write $P(x) = A(x) B(x)$ for some $A(x), B(x) \in \qq[M_{1/2}]$. After taking $n \in \nn$ such that $A\big( x^{2^n} \big)$ and $B\big( x^{2^n} \big)$ both belong to $\qq[x]$, we see that $P\big( x^{2^n} \big) = A\big( x^{2^n} \big) B\big( x^{2^n} \big) \in \qq[x]$. Thus, $P\big( x^{2^n}\big)$ is irreducible in $\qq[x]$, which implies that either $A(x)$ or $B(x)$ is constant. Hence $P(x)$ is irreducible in $\qq[M_{1/2}]$.
\end{proof}

\begin{prop}
	The monoid algebra $\qq[M_{1/2}]$ is an AP-domain.
\end{prop}

\begin{proof}
	Let $A(x)$ be an irreducible polynomial expression in $\qq[M_{1/2}]$. %Assume that such a $P(x)$ minimizes $\min \mathsf{L}(P(x))$ and set $\ell := \min \mathsf{L}(P(x))$. 
	Because the multiplicative monoid $M := \big\{ qx^m : q \in \qq \setminus \{0\} \text{ and } m \in M_{1/2} \big\}$ is a divisor-closed submonoid of $\qq[M_{1/2}]^*$ whose reduced monoid is isomorphic to $M_{1/2}$, the fact that $M_{1/2}$ is antimatter implies that $M$ is antimatter. Hence no monomial of $\qq[M_{1/2}]$ is atomic. Thus, $A(0) \neq 0$. Now let $P(x)$ and $Q(x)$ be nonzero polynomial expressions in $\qq[M_{1/2}]$ such that $A(x) \mid_{\qq[M_{1/2}]} P(x) Q(x)$. Take $B(x) \in \qq[M_{1/2}]$ such that $A(x) B(x) = P(x) Q(x)$, and then take $m \in \nn$ such that $A\big( x^{2^m} \big), B\big( x^{2^m} \big), P\big( x^{2^m} \big)$, and $Q\big( x^{2^m}\big)$ all belong to $\qq[x]$. Hence $A\big(x^{2^m}\big) \mid_{\qq[x]} P\big( x^{2^m }\big) Q\big( x^{2^m}\big)$. It follows from Lemma~\ref{lem:irreducible elements in $Q[M_{1/2}]$} that $A\big( x^{2^m}\big)$ is irreducible in $\qq[x]$, so it must be prime because $\qq[x]$ is an AP-domain. Then we can assume that $A\big(x^{2^m}\big) \mid_{\qq[x]} P\big( x^{2^m }\big)$. Take $B'(x) \in \qq[x]$ such that $A\big( x^{2^m}\big) B'(x) = P\big( x^{2^m}\big)$ or, equivalently, $A(x) B'\big(x^{1/2^m} \big) = P(x)$. Therefore $A(x) \mid_{\qq[M_{1/2}]} P(x)$, and we can conclude that $A(x)$ is prime. Hence $\qq[M_{1/2}]$ is an AP-domain.
\end{proof}

\begin{cor} \label{cor:atomic element characterization in $Q[M_{1/2}]$}
	For a polynomial expression $P(x) \in \qq[M_{1/2}]$, the following statements are equivalent.
	\begin{enumerate}
		\item[(a)] $P(x)$ has finitely many non-associate divisors.
		\smallskip
		
		\item[(b)] $P(x)$ satisfies the ACCP.
		\smallskip
		
		\item[(c)] $P(x)$ is atomic.
	\end{enumerate}
\end{cor}

\begin{proof}
	(a) $\Rightarrow$ (b): This is clear.
	\smallskip
	
	(b) $\Rightarrow$(c): This follows from~\cite[Lemma~3.4]{GL23a}.
	\smallskip
	
	(c) $\Rightarrow$ (a): Assume that $P(x)$ is atomic, and write $P(x) = A_1(x) \cdots A_\ell(x)$ for some irreducibles $A_1(x), \dots, A_\ell(x)$ in $\qq[M_{1/2}]$. Since $\qq[M_{1/2}]$ is an AP-domain, $A_1(x), \dots, A_\ell(x)$ are primes and, therefore, for every divisor $Q(x)$ of $P(x)$ in $\qq[M_{1/2}]$ there exists $I \subseteq \ldb 1,\ell \rdb$ such that $Q(x)$ is associate to $\prod_{i \in I} A_i(x)$ in $\qq[M_{1/2}]$. Thus, $P(x)$ as finitely many non-associate divisors.
\end{proof}

Following Grams and Warner~\cite{GW75}, we say that an integral domain is \emph{irreducible-divisor-finite} (or an \emph{IDF-domain}) if every nonzero element has only finitely many irreducible divisors up to associates. As the following example illustrates, $\qq[M_{1/2}]$ is not an IDF-domain.

\begin{example}\cite[Example~5.5]{fG22}.
	For each $n \in \nn$, we can write
	\begin{equation} \label{eq:factorization into cyclotomics}
		x-1 = \big( x^{1/2^n}\big)^{2^n} - 1 = \prod_{j=0}^n \Phi_{2^j}\big( x^{1/2^n}\big).
	\end{equation}
	It follows from Lemma~\ref{lem:irreducible elements in $Q[M_{1/2}]$} that all the factors in the rightmost expression of~\eqref{eq:factorization into cyclotomics} except the one corresponding to $j=0$ are irreducible in $\qq[M_{1/2}]$. Hence the binomial $x-1$ has infinitely many non-associate irreducible divisors in $\qq[M_{1/2}]$.
\end{example}
\smallskip

We say that a nonconstant polynomial expression $P(x) \in \zz[M_{1/2}]$ is \emph{quasi-irreducible} provided that $P\big(x^{2^n}\big)$ is irreducible in $\zz[x]$ for some $n \in \nn_0$. It follows directly from the definition that every nonconstant polynomial $P(x)$ that is irreducible in $\zz[x]$ must be quasi-irreducible as an element of the monoid algebra $\zz[M_{1/2}]$. 

We proceed to assign a tree to each quasi-irreducible polynomial expression. Fix a quasi-irreducible $P(x) \in \zz[M_{1/2}]$. If $P(x)$ is irreducible in $\qq[M_{1/2}]$, then the \emph{splitting tree} of $P(x)$ consists of one node, itself. Now suppose that $P(x)$ is reducible in $\qq[M_{1/2}]$. In light of Gauss's lemma, we can pick nonconstant polynomials $A(x), B(x) \in \zz[x]$ and $m \in \nn_0$ such that $P\big( x^{2^m}\big) = A(x) B(x)$. Assume that~$m$ has been chosen as small as possible. This being the case, the fact that $P(x)$ is quasi-irreducible implies that $P\big( x^{2^{m-1}}\big)$ is an irreducible polynomial in $\zz[x]$. Thus, it follows from Square Lemma that $A(x)$ and $B(x)$ are the irreducible polynomials in the unique factorization of $P\big( x^{2^m}\big)$ in $\zz[x]$. Now set $A'(x) := A(x^{2^{-m}})$ and $B'(x) := B(x^{2^{-m}})$. Therefore $A'(x)$ and $B'(x)$ are quasi-irreducibles in $\zz[M_{1/2}]$ satisfying $P(x) = A'(x) B'(x)$. Then the \emph{splitting tree} of $P(x)$ is the binary tree having root $P(x)$ with children $A'(x)$ and $B'(x)$, and the children of $A'(x)$ and $B'(x)$ are inductively determined by the splitting trees of $A'(x)$ and $B'(x)$.

We proceed to characterize the irreducible polynomials in $\zz[x]$ that are not atomic in the monoid algebra $\qq[M_{1/2}]$.

\begin{prop} \label{prop:characterization of polynomials in Z[x] not atomic in Q[M_{1/2}]}
	If an irreducible polynomial in $\zz[x]$ is not atomic in $\qq[M_{1/2}]$, then it has a splitting sequence whose binary string has infinitely many copies of $S$.
\end{prop}

\begin{proof}
	Let $P(x)$ be an irreducible polynomial in $\zz[x]$ that is not atomic in $\qq[M_{1/2}]$. As $P(x)$ is not atomic in $\qq[M_{1/2}]$, it is nonconstant. Let $\mathfrak{t}$ be the splitting tree of $P(x)$. First, notice that the leaves of $\mathfrak{t}$ cannot split, and so every leaf of $\mathfrak{t}$ is irreducible in $\qq[M_{1/2}]$. Since every non-leaf is the product of its two children, if $\mathfrak{t}$ were finite, then we could proceed by induction on the height of $\mathfrak{t}$ to argue that $P(x)$ equals the product of all the leaves of $\mathfrak{t}$: however, this is not possible because $P(x)$ is not atomic in $\qq[M_{1/2}]$. Therefore $\mathfrak{t}$ must be an infinite tree, and so K\"onig's lemma guarantees the existence of an infinite path starting at the root of $\mathfrak{t}$. Let $Q_n(x)$ be the $n$-th polynomial in such infinite path, where $Q_1(x) = P(x)$. For each $n \in \nn$, the fact that $Q_n(x)$ was not taken from a leaf of~$\mathfrak{t}$ guarantees that $Q_n(x)$ is reducible in $\qq[M_{1/2}]$, and so by virtue of Lemma~\ref{lem:irreducible elements in $Q[M_{1/2}]$} there exists a minimum $e_n \in \nn$ such that $Q_n\big( x^{2^{e_n}}\big)$ is reducible in $\qq[x]$, and so in $\zz[x]$ because of Gauss's lemma. %Since $Q_n(x)$ is quasi-irreducible, the minimality of $e_n$ ensures that $Q_n\big( x^{2^{e_n-1}}\big)$ is an irreducible polynomial in $\zz[x]$...
	
	We proceed to construct inductively a splitting sequence $(P_n(x))_{n \ge 0}$ of $P(x)$ and a strictly increasing sequence $(m_k)_{k \ge 1}$ of positive integers such that $P_{m_k}(x)$ is reducible in $\zz[x]$ for every $k \in \nn$. For the base case, set $P_0(x) := P(x)$ and then set $P_j(x) := P(x^{2^j})$ for every $j \in \ldb 1, e_1 \rdb$. Note that $P_j(x)$ is the square lifting of $P_{j-1}(x)$ for every $j \in \ldb 1, e_1 \rdb$. Also, set $m_1 := e_1$, and observe that $P_{m_1}(x) = Q_1\big( x^{2^{e_1}}\big)$ is reducible in $\zz[x]$. 
 
    For the inductive step, assume that for some $k \in \nn$ we have chosen $P_0(x), \dots, P_{m_k}(x) \in \zz[x]$ and $m_1, \dots, m_k \in \nn$ with $m_1 < \cdots < m_k$ in such a way that $P_j(x)$ is the square lifting (resp., a splitting) of $P_{j-1}(x)$ if $P_{j-1}(x)$ is irreducible (resp., reducible) for every $j \in \ldb 1, m_k \rdb$ and $P_{m_i}(x) = Q_i\big(x^{2^{e_i}} \big)$ for every $i \in \ldb 1,k \rdb$. As a consequence, we can write $P_{m_k}(x) = Q_k\big(x^{2^{e_k}} \big) = A_{k+1}(x) B_{k+1}(x)$, where $A_{k+1}(x)$ and $B_{k+1}(x)$ are the irreducibles of $\zz[x]$ resulting from Square Lemma, and we can assume that $Q_{k+1}(x) = A\big( x^{2^{-e_k}}\big) \in \zz[M_{1/2}]$. Since $A_{k+1}\big( x^{2^{e_{k+1} - e_k}}\big) = Q_{k+1}\big( x^{2^{e_{k+1}}}\big)$ is reducible in $\zz[x]$, the fact that $A_{k+1}(x)$ is irreducible in $\zz[x]$ implies that $n_k := e_{k+1} - e_k \ge 1$. Now set $P_{n+j}(x) := A_{k+1}\big( x^{2^{j-1}}\big)$ for every $j \in \ldb 1, n_k + 1 \rdb$. It is clear that $P_{m_k+1}(x)$ is a splitting of $P_{m_k}(x)$, namely, $A_{k+1}(x)$. In addition, observe that $P_{m_k+j+1}(x)$ is the square lifting of $P_{m_k+j}(x)$ for every $j \in \ldb 1, n_k \rdb$, and $P_{m_k + n_k + 1}(x) = Q_{k+1}\big( x^{2^{e_{k+1}}}\big)$. 
    
    Having settled our inductive step, we can guarantee the existence of a splitting sequence $(P_n(x))_{n \ge 0}$ of $P(x)$ and a strictly increasing sequence $(m_k)_{k \ge 1}$ of positive integers such that $P_{m_k}(x)$ is reduced for every $k \in \nn$. Thus, $(P_n(x))_{n \ge 0}$ is a splitting sequence of $P(x)$ whose binary string contains infinitely many copies of~$S$.
\end{proof}

%\newpage
%The following lemma is used in Theorem~\ref{thm:irreducible polynomials that are not atomic in $Q[M_1/2]$}.
%
%\begin{lemma} \label{lem:sufficient conditions for nice polynomial}
%	Let $P(x)$ be an irreducible polynomial in $\zz[x]$. If $P(x)$ has a splitting sequence whose binary string has infinitely many copies of $S$, then $P(x)$ is associate to a nice polynomial.
%\end{lemma}
%
%\begin{proof}
%	Let $(P_n(x))_{n \ge 1}$ be a splitting sequence of $P(x)$ whose binary string contains infinitely many copies of $S$. First, observe that taking square lifting respects both the leading and the constant coefficients of any polynomial. Now fix $n \in \nn$ such that $P_n(x)$ is not irreducible. If $P_n(x) = A(x)B(x)$ is the splitting in part~(2) of Square Lemma, then the coefficients of $x^k$ in both $A(x)$ and $B(x)$ have the same magnitude for every $k \in \nn$. Thus the magnitude of the leading and constant coefficients of $P_{n+1}(x)$ are the squares of the magnitudes of the leading and constant coefficients of $P_n(x)$, respectively. Therefore the fact that the binary sequence of $(P_n(x))_{n \ge 1}$ contains infinitely many copies of $S$ guarantees that both the leading and the constant coefficient of $P(x)$ belong to $\{\pm 1\}$, whence $P(x)$ is associate to a nice polynomial.
%\end{proof}

We are in a position now to determine the irreducible polynomials in $\zz[x]$ that are not atomic in $\qq[M_{1/2}]$.

\begin{theorem} \label{thm:irreducible polynomials that are not atomic in $Q[M_1/2]$}
	If an irreducible polynomial $P(x) \in \zz[x]$ is not atomic in the monoid algebra $\qq[M_{1/2}]$, then $P(x) \in \{\pm x, \pm \Phi_n(x) : n \in 2\nn_0 + 1\}$. %or $P(x) = \Phi_n(x)$ for an odd integer $n$.
\end{theorem}

\begin{proof}
	Suppose, by way of contradiction, that $P(x) \notin \{\pm x, \pm \Phi_n(x) : n \in 2\nn_0 + 1\}$. Thus, $P(x)$ is not a monomial. Since $P(x)$ is not atomic in $\qq[M_{1/2}]$, Proposition~\ref{prop:characterization of polynomials in Z[x] not atomic in Q[M_{1/2}]} guarantees the existence of a splitting sequence $(P_n(x))_{n \ge 1}$ whose binary sequence $\mathfrak{s} = s_1 s_2 \ldots$ has infinitely many copies of $S$. Since $P(x)$ is not a monomial, it follows from Lemma~\ref{lem:splitting sequences and NP} that $P(x)$ is a nice polynomial.
%	We proceed to argue that $P(x)$ is a nice polynomial. First, observe that taking square lifting respects both the leading and the constant coefficients of any polynomial. Now fix $n \in \nn$ such that $P_n(x)$ is not irreducible in $\zz[x]$. If $P_n(x) = A(x)B(x)$ is the splitting in part~(2) of Square Lemma, then the coefficients of $x^k$ in both $A(x)$ and $B(x)$ have the same magnitude for every $k \in \nn$. Thus, the magnitudes of the leading and constant coefficients of $P_{n+1}(x)$ are the square roots of the magnitudes of the leading and constant coefficients of $P_n(x)$, respectively. Therefore the fact that the binary sequence of $(P_n(x))_{n \ge 1}$ contains infinitely many copies of $S$ guarantees that both the leading and the constant coefficient of $P(x)$ belong to $\{\pm 1\}$, whence $P(x)$ is a nice polynomial. %Then Lemma~\ref{lem:sufficient conditions for nice polynomial} allows us to assume that $P(x)$ is a nice polynomial (after replacing $P(x)$ by $-P(x)$ if necessary). 
Thus, it follows from Lemma~\ref{lem:lemma 13} that the set $\{\deg P_n(x) : n \in \nn\}$ is not bounded. This in turn enforces the existence of infinitely many copies of the string LLS in $\mathfrak{s}$. As a result, $P(x)$ is $k$-special for infinitely many $k \in \nn$, which implies that $P(x)$ is $k$-special for every $k \in \nn$. However, the fact that $P(x)$ is a nice polynomial that is $k$-special for every $k \in \nn$ contradicts part~(2) of Proposition~\ref{prop:nice polynomials}.
\end{proof}

\medskip
%%%%%%%%%%%%%%%%%%%%%%%%%%%%%%%%%%%%
\subsection{Ascending Chain of Principal Ideals in $\qq[M_{3/4}]$}

The machinery we have developed so far can be applied to gain a better understanding of ascending chains of principal ideals inside the monoid algebra $\qq[M_{3/4}]$. In this direction, we prove that the monoid algebra $\qq[M_{3/4}]$ satisfies the almost ACCP, which gives a partial answer to \cite[Conjecture~4.12]{fG22}.

\begin{lemma} \label{lem:factorizing polynomial expressions in $Q[M_{1/2}]$}
	Let $P(x)$ be a nonzero polynomial expression in $\qq[M_{1/2}]$ %with $\emph{ord} \, P(x) = 0$, 
	and let $N$ be the minimum nonnegative integer such that $P\big( x^{2^N} \big) \in \qq[x]$. If the infimum of the set
	\[
		\big\{\deg Q(x) : Q_n(x) \text{ is not a monomial and divides }P(x) \text{ in } \qq[M_{1/2}] \big\}
	\]
	is zero, then the following statements hold.
	\begin{enumerate}
		\item If $\emph{ord} \, P(x) = 0$, then $P\big(x^{2^N}\big) = \prod_{i=1}^t \Phi_{n_i}(x)$ for some $n_1, \dots, n_t \in 2\nn_0 + 1$.
		\smallskip
		
		\item If $n \in \nn_{\ge N}$, then $P\big( x^{2^n}\big) = x^{k_n} \prod_{i=1}^t \Phi_{n_i}(x)$ for some $k_n \in \nn_0$ and $n_1, \dots, n_t \in \nn$. %$P\big( x^{2^N}\big) = x^{2^N \emph{ord} \, P(x)} \prod_{i=1}^t \Phi_{n_i}(x)$ for some $n_1, \dots, n_t \in \nn$.
	\end{enumerate}
\end{lemma}

\begin{proof}
	(1) Set $P'(x) := P\big( x^{2^N} \big)$. Let $(Q_n(x))_{n \ge 1}$ be a sequence consisting of monic polynomial expressions in $\qq[M_{1/2}] \setminus \qq$ dividing $P(x)$ in $\qq[M_{1/2}]$ with $\lim_{n \to \infty} \deg Q_n(x) = 0$. Clearly, one can assume that $\text{ord} \, Q_n = 0$ for every $n \in \nn$. For each $n \in \nn$, let $M_n$ be the minimum nonnegative integer such that both polynomial expressions $Q_n(x)$ and $P(x)/Q_n(x)$ belong to $\qq\big[x^{1/2^{M_n}}\big]$, and set
	\[
		P''(x) := P\big( x^{2^{M_n}}\big) \quad \text{ and } \quad Q_n'(x) := Q_n\big( x^{2^{M_n}}\big).
	\]
	It is clear that $M_n \ge N$ and $Q'_n(x) \mid_{\qq[x]} P''(x)$ for every $n \in \nn$.
	\smallskip
	
	Assume, by way of contradiction, that one of the factors in the only factorization of $P'(x)$ into irreducibles in $\qq[x]$ is not associate to $\Phi_{2n+1}(x)$ for any $n \in \nn_0$. Then we can write
    \[
        P'(x) = P_1'(x)^m P_2'(x)
    \]
    in $\qq[x]$ for some $m \in \nn$ and $P'_1(x), P'_2(x) \in \qq[x]$ such that $P_1'(x) \in \mathcal{A}(\zz[x]) \setminus \{\pm \Phi_{2n+1}(x) : n \in \nn_0\}$ and $P'_1(x) \nmid_{\qq[x]} P'_2(x)$. Furthermore, we can assume that the leading coefficient of $P'_1(x)$ is positive. Now set
	\[
		P_1(x) := P'_1\big( x^{1/2^N} \big) \quad \text{ and } \quad P_2(x) := P'_2\big( x^{1/2^N} \big).
	\]
	Observe that $P_1(x)^m P_2(x) = P(x)$ in $\qq[M_{1/2}]$, %Then take the preimages of $P_1'(x)$ and $P_2'(x)$ to obtain $P_1(x)$ and $P_2(x)$, fixed irrespective of $Q(x)$. 
	%Now fix $n \in \nn$, and then take the smallest nonnegative integer~$M$ such that $P_0(x), Q_n(x) \in \qq\big[ x^{1/2^M} \big]$. It is clear that $M \ge N$. 
	and so $P''(x) = P_1'\big(x^{2^{M_n-N}}\big)^m P_2'\big(x^{2^{M_n-N}}\big)$. As $P'_1(x)^m$ and $P'_2(x)$ are coprime in $\qq[x]$ it follows that $P'_1\big(x^{2^{M_n-N}}\big)^m$ and $P'_2\big(x^{2^{M_n-N}}\big)$ are also coprime in $\qq[x]$ (to see this, it suffices to compose any B\'ezout's identity for $P'_1(x)^m$ and $P'_2(x)$ with $x^{2^{M_n - N}}$). Take
	\[
		Q_{n,1}'(x) := \text{gcd}_{\qq[x]}\big( Q_n'(x), P'_1\big(x^{2^{M_n-N}}\big)^m \big) \quad \text{ and } \quad Q_{n,2}'(x) := \text{gcd}_{\qq[x]}\big( Q_n'(x),P'_2\big(x^{2^{M_n-N}}\big)\big).
	\]
	%to be the greatest monic common divisors of the sets $\{Q_n'(x), P'_1(x^{2^{M_n-N}})^m \}$ and $\{Q_n'(x),P'_2(x^{2^{M_n-N}}) \}$ in $\qq[x]$, respectively. 
	Since $P'_1\big(x^{2^{M_n-N}}\big)^m$ and $P'_2\big(x^{2^{M_n-N}}\big)$ are coprime with product $P''(x)$, it follows from the divisibility relation $Q'_n(x) \mid_{\qq[x]} P''(x)$ that $Q_n'(x) = Q_{n,1}'(x) Q_{n,2}'(x)$ in $\qq[x]$. %\textcolor{red}{Ben: Why this equation holds? Maybe we can argue $P'_1(x^{2^{M - N}})$ and $P'_2(x^{2^{M - N}})$ are coprime?} 
	Now set
	\[
		Q_{n,1}(x) := Q'_{n,1}\big(x^{1/2^{M_n}} \big) \quad \text{ and } \quad Q_{n,2}(x) := Q'_{n,2}\big(x^{1/2^{M_n}} \big),
	\]
	and note that $Q_n(x) = Q_{n,1}(x) Q_{n,2}(x)$ in $\qq[M_{1/2}]$. After writing $Q'_{n,1}(x) Q(x) = P'_1\big( x^{2^{M_n-N}}\big)^m$ for some $Q(x) \in \qq[x]$, we see that
	\[
		Q_{n,1}\big(x^{2^N} \big) Q\big(x^{1/2^{M_n-N}} \big) = Q'_{n,1}\big(x^{1/2^{M_n-N}} \big) Q\big(x^{1/2^{M_n-N}} \big) = P'_1(x)^m,
	\]
	and so $Q_{n,1}\big(x^{2^N} \big)$ divides $P'_1(x)^m$ in $\qq[M_{1/2}]$. %Hence $P'_1(x)^m$ has infinitely many non-associate divisors in $\qq[M_{1/2}]$. 
	Because $Q'_{n,1}(x) \mid_{\qq[x]} Q'_n(x)$, it follows that $\deg Q_{n,1}(x) = \deg Q'_{n,1} \big( x^{1/2^{M_n}}\big) \le \deg Q'_n\big( x^{1/2^{M_n}}\big) = \deg Q_n(x)$. Thus, from the fact that $\lim_{n \to \infty} \deg Q_n(x) = 0$, we can now infer that $\lim_{n \to \infty} \deg Q_{n,1}\big( x^{2^N}\big) = 0$ (as $N$ does not depend on $n$). As a consequence, $\big\{ Q_{n,1}\big(x^{2^N} \big) : n \in \nn \big\}$ contains infinitely many non-associate divisors of $P'_1(x)^m$ in $\qq[M_{1/2}]$, and so it follows from Corollary~\ref{cor:atomic element characterization in $Q[M_{1/2}]$} that $P'_1(x)^m$ is not atomic in $\qq[M_{1/2}]$. This implies that $P'_1(x)$ is not atomic in $\qq[M_{1/2}]$. Hence Theorem~\ref{thm:irreducible polynomials that are not atomic in $Q[M_1/2]$} ensures that $P'_1(x) = x$, which contradicts that $\text{ord} \, P'(x) = 0$. % Therefore Claim 1 is established. %\textcolor{red}{Ben: Why? For example, why $Q_{n, 1}(x^{2^N})$ cannot all be $x^2 - x + 1$?} Thus, $P'_1(x)^m$ is not atomic in $\qq[M_{1/2}]$ in light of Corollary~\ref{cor:atomic element characterization in $Q[M_{1/2}]$}. This implies that $P'_1(x)$ is not atomic $\qq[M_{1/2}]$. Hence it follows from Theorem~\ref{thm:irreducible polynomials that are not atomic in $Q[M_1/2]$} that $P'_1(x) = x$, which contradicts that $\text{ord} \, P'(x) = 0$. Therefore Claim 1 is established.
	%the infimum of the sequence $(\deg Q_{n,1}(x))_{n \ge 1}$ is positive, which is not possible because $\lim_{n \to \infty} \deg Q_n(x) = 0$. Therefore the Claim 1 is established.
	\smallskip

% Ben'n question: Why can't $Q_n$ (in the previous paragraph) be $1$?
% Alan-Alex's answer: The lemma statement discusses the set of degrees of all *non-monomial* divisors of $P$, so $Q$ cannot be 1 by definition. It would be helpful to clarify that $Q_n$ are not monomials in the beginning of the proof.
    
	(2) Set $P'(x) := P(x)/x^q$, where $q := \text{ord} \, P(x)$, and let $m$ be the minimum nonnegative integer such that $P'\big(x^{2^m}\big) \in \qq[x]$. It is clear that $m \le N$. Now observe that if $(Q_n(x))_{n \ge 1}$ is a sequence of non-monomials in $\qq[M_{1/2}] \setminus \qq$ dividing $P(x)$ in $\qq[M_{1/2}]$ with $\lim_{n \to \infty} \deg Q_n(x) = 0$, then $(Q_n(x)/x^{\text{ord} Q_n(x)})_{n \ge 1}$ is a sequence of non-monomials in $\qq[M_{1/2}]$ dividing $P'(x)$ in $\qq[M_{1/2}]$ with $\lim_{n \to \infty} \deg \big( Q_n(x)/x^{\text{ord} Q_n(x)} \big) = 0$. Hence $P'(x)$ also satisfies the hypothesis of the lemma. Since $\text{ord} \, P'(x) = 0$, by part~(1) we can write $P\big( x^{2^m}\big) = x^{2^mq} \prod_{i=1}^s \Phi_{m_i}(x)$ for some $m_1, \dots, m_s \in 2\nn_0 + 1$. Now if $n \in \nn_{\ge N}$, then $2^nq \in \nn_0$ and $P\big( x^{2^n}\big) = x^{2^nq} \prod_{i=1}^s \Phi_{m_i}\big(x^{2^{n-m}}\big)$, and so we are done in light of Lemma~\ref{lem:cyclotomic polynomial product identity}.
\end{proof}

Nonzero polynomial expressions in $\qq[M_{3/4}]$ having a monomial whose exponent satisfies the ACCP in $M_{3/4}$ play an essential role in the proof of the main result of this section, Theorem~\ref{thm:main}. We call such polynomial expressions \emph{ACCP-supported}. As the following proposition indicates, every ascending chain of principal ideals in $\qq[M_{3/4}]$ starting at an ACCP-supported polynomial expression must stabilize.

\begin{prop} \label{prop:ACCP-supported PE are ACCP}
	Every ACCP-supported polynomial expression in $\qq[M_{3/4}]$ satisfies the ACCP.
\end{prop}

\begin{proof}
	Suppose, towards a contradiction, that $(P_n(x) \qq[M_{3/4}])_{n \ge 0}$ is a non-stabilizing ascending chain of principal ideals with $P_0(x)$ being an ACCP-supported polynomial expression in $\qq[M_{3/4}]$. Assume, without loss of generality, that $P_n(x)$ is monic for every $n \in \nn_0$. Furthermore, we can assume that
    \[
        Q_{n+1}(x) := \frac{P_n(x)}{P_{n+1}(x)}
    \]
    is not constant for any $n \in \nn_0$. Since $P_0(x)$ is an ACCP-supported polynomial expression, only finitely many terms of the sequence $(Q_n(x))_{n \ge 1}$ can be monomials. Hence, after dropping finitely many initial terms from the sequence $(P_n(x))_{n \ge 0}$, we can assume that $Q_n(x)$ is not a monomial for any $n \in \nn$. Since the sequence $(\deg P_n(x))_{n \ge 0}$ is decreasing, it must converge, whence $\lim_{n \to \infty} \deg Q_n(x) = 0$. For each $n \in \nn$, set
	\[
		o_n := \text{ord} \, Q_n(x) \quad \text{ and } \quad q_n := \min \big(\text{supp} \, Q_n(x) \setminus \{o_n\} \big)
	\]
	and observe that $q_n$ is well defined because $Q_n(x)$ is not a monomial. Note that both sequences $(o_n)_{n \ge 1}$ and $(q_n)_{n \ge 1}$ converge to zero.
	
	Take $N$ to be the smallest nonnegative integer such that $P_0(x) \in \qq\big[ x^{1/{2^N}} \big]$ and, for each $n \in \nn$, take $M_n$ to be the smallest nonnegative integer such that both $Q_n(x)$ and $P_0(x)/Q_n(x)$ belong to $\qq\big[x^{1/2^{M_n}}\big]$.  Clearly, $M_n \ge N$. Now, for each $n \in \nn$, set $p := \text{ord} \, P_0\big( x^{2^{M_n}}\big)$ and $o'_n := \text{ord} \, Q_n\big( x^{2^{M_n}}\big)$, set
	\[
		P'(x) := P_0\big( x^{2^{M_n}}\big)/x^p \quad \text{ and } \quad Q_n'(x) := Q_n\big( x^{2^{M_n}}\big)/x^{o'_n},
	\]
	and observe that both $Q'_n(x)$ and $P'(x)$ are monic polynomials in $\qq[x]$ such that $Q'_n(x) \mid_{\qq[x]} P'(x)$.

	Because $\big(Q_n\big( x^{2^N} \big) \big)_{n \ge 1}$ is a sequence with $\inf\big\{\!\deg Q_n\big( x^{2^N}\big) : n \in \nn \big\} = 0$ whose terms divide $P_0\big(x^{2^N} \big)$ in $\qq[M_{1/2}]$, it follows from Lemma~\ref{lem:factorizing polynomial expressions in $Q[M_{1/2}]$} that
	\begin{equation} \label{eq:almost-product of cyclotomic}
		P_0\big(x^{2^{M_n}} \big) = x^p \prod_{i=1}^t \Phi_{m_i}(x)
	\end{equation}
	for some $m_1, \dots, m_t \in \nn$. Now fix $n \in \nn$. It follows from~\eqref{eq:almost-product of cyclotomic} that both $P'(x)$ and $Q'_n(x)$ can be written as products of cyclotomic polynomials. Indeed, after permuting the cyclotomic factors in~\eqref{eq:almost-product of cyclotomic}, we can assume that $Q'_n(x) = \prod_{i=1}^j \Phi_{m_i}(x)$ for some $j \in \ldb 1,t \rdb$. Because cyclotomic polynomials are palindromic and $Q'_n(x)$ is a product of cyclotomic polynomials, $Q'_n(x)$ is palindromic.
 
    Set $d_n := \deg Q'_n(x)$, and let $r_1, \dots, r_{d_n}$ be the roots of $Q'_n(x)$. Let $k_n$ be the smallest positive integer such that the coefficient of $x^{d_n-k_n}$ in $Q'_n(x)$ is nonzero. Then in light of Lemma~\ref{lem:power sum symmetric functions}, we see that $k_n$ is also the smallest positive integer such that $\sum_{i=1}^{d_n} r_i^{k_n}  \neq 0$. As $Q'_n(x)$ is palindromic, the equality $k_n = \text{ord}(Q'_n(x) - Q'_n(0))$ holds. Now take $u,v,w \in \nn_0$ such that $P_0\big(x^{2^N}\big) \mid_{\zz[x]} (x^u - 1)^v x^w$ (observe that the exponents $u,v,w$ do not depend on $n$). Since $Q_n\big(x^{2^{M_n}}\big) \mid_{\qq[x]} P_0\big( x^{2^{M_n}}\big)$, we see that $Q'_n(x) x^{o'_n} \mid_{\qq[x]}  \big( x^{2^{M_n-N}u} - 1 \big)^v x^{2^{M_n-N}w}$ and, therefore, $Q'_n(x) \mid_{\qq[x]} \big( x^{2^{M_n-N}u} - 1 \big)^v$.
	\medskip
	
	\noindent \textsc{Claim.} $k_n \mid 2^{M_n-N} u$.
	\smallskip
	
	\noindent \textsc{Proof of Claim.} Set $\ell := 2^{M_n-N} u$. Now suppose, for the sake of a contradiction, that $k_n \nmid \ell$. Set $k'_n = \gcd(k_n, \ell) < k_n$. Since  $\prod_{i=1}^j \Phi_{m_i}(x) = Q'_n(x) \mid_{\qq[x]} \big(x^\ell - 1\big)^u$ , for each $i \in \ldb 1,j \rdb$, the identity $x^\ell - 1 = \prod_{d \mid \ell} \Phi_d(x)$ ensures that $m_i \mid \ell$, and so $\gcd(k'_n, m_i) = \gcd(k_n, m_i)$. Now fix $m \in \{m_1, \dots, m_j\}$, let~$r$ be a primitive $m$-th root of unity, and let $r^{e_1}, \dots, r^{e_{\varphi(m)}}$ be the roots of the cyclotomic polynomial $\Phi_m(x)$. After taking $k \in \nn$ and setting $g := \gcd(m,k)$, one sees that $r^k$ is a primitive $(m/g)$-th root of unity, and so the fact that $\gcd(e_i, m/g) = \gcd(e_i,m) = 1$ for every $i \in \ldb 1, \varphi(m) \rdb$ ensures that $r^{ke_1}, \dots, r^{ke_{\varphi(m)}}$ are roots of $\Phi_{m/g}(x)$. Therefore the map $f \colon (\zz/m\zz)^\times \to (\zz/(m/g)\zz)^\times$ defined by $j + m\zz \mapsto j + (m/g)\zz$ is a well-defined group homomorphism. The homomorphism $f$ is clearly surjective. This, along with the fact that all the preimages of $f$ have the same size, guarantees that $r^{ke_1}, \dots, r^{ke_{\varphi(m)}}$ are the roots of the polynomial $\Phi_{m/g}(x)^{\varphi(m)/\varphi(m/g)}$, accounting for repetitions. As a result, from the fact that $\gcd(k'_n, m_i) = \gcd(k_n, m_i)$ for every $i \in \ldb 1,j \rdb$ we can infer that the underlying multisets of the finite sequences $(r_i^{k'_n})_{i=1}^{d_n}$ and $(r_i^{k_n})_{i=1}^{d_n}$ are equal. Consequently, $\sum_{i=1}^{d_n} r_i^{k'_n} = \sum_{i=1}^{d_n} r_i^{k_n} \neq 0$, and so the fact that $k_n' < k_n$ contradicts the minimality of $k_n$. The claim is now established.
	\smallskip
	
	We are now ready to conclude the proof. It follows from the claim that $k_n \mid 2^{M_n - N} u$ for every $n \in \nn$. On the other hand, for each $n \in \nn$, it follows from the equality $k_n = \text{ord}(Q'_n(x) - Q'_n(0))$ that $k_n = 2^{M_n}(q_n - o_n)$, and so $v_3(k_n) = v_3(q_n - o_n)$. Because $\lim_{n \to \infty} q_n = \lim_{n \to \infty} o_n = 0$, the fact that $M_{3/4}$ is atomic with $\mathcal{A}(M_{3/4}) = \big\{ \big( \frac34 \big)^n : n \in \nn_0 \big\}$ allows us to take $n_0 \in \nn$ large enough so that $q_{n_0}$ and $o_{n_0}$ can both be written as a finite sum of irreducibles in the set $\big\{ \big( \frac34 \big)^i : i > v_3(u) \big\}$. As a consequence, $v_3(k_{n_0}) = v_3(q_{n_0} - o_{n_0}) > v_3(u)$. However, this contradicts that $k_{n_0} \mid 2^{M_{n_0} - N} u$.
\end{proof}

We are in a position to prove that the monoid algebra $\qq[M_{3/4}]$ is atomic.

\begin{theorem} \label{thm:main}
	The monoid algebra $\qq[M_{3/4}]$ satisfies the almost ACCP but does not satisfy the ACCP.
\end{theorem}

\begin{proof}
	It was proved in \cite[Example~3.2]{GL23a} that $M_{3/4}$ satisfies the almost ACCP. To argue that the monoid algebra $\qq[M_{3/4}]$ also satisfies the almost ACCP, fix nonconstant polynomials $P_1(x), \dots, P_k(x) \in \qq[M_{3/4}]$. Now set $S := \bigcup_{i=1}^k \text{supp} \, P_i(x)$. Because $M_{3/4}$ satisfies the almost ACCP, we can pick a common divisor $d \in M_{3/4}$ of $S$ such that $s-d$ satisfies the ACCP in $M_{3/4}$ for some $s \in S$. Take $j \in \ldb 1,k \rdb$ such that $s \in \text{supp} \, P_j(x)$. It is clear that $x^d$ is a common divisor of $P_1(x), \dots, P_k(x)$ in $\qq[M_{3/4}]$. In addition, $s-d \in \text{supp} \, P_j(x)/x^d$, and so $P_j(x)/x^d$ is an ACCP-supported polynomial expression in $\qq[M_{3/4}]$. Then it follows from Proposition~\ref{prop:ACCP-supported PE are ACCP} that $P_j(x)/x^d$ satisfies the ACCP in $\qq[M_{3/4}]$. Hence we can conclude that $\qq[M_{3/4}]$ satisfies the almost ACCP.
\end{proof}

Thus, the following corollary holds.

\begin{cor}
	The one-dimensional monoid algebra $\qq[M_{3/4}]$ is atomic but does not satisfy the ACCP.
\end{cor}

\bigskip
%%%%%%%%%%%%%%%%%%%%%%%%%%%%%%%%%%
%%%%%%%%%%%%%%%%%%%%%%%%%%%%%%%%%%
\section{More Almost ACCP One-Dimensional Monoid Algebras}
\label{sec:another almost ACCP monoid algebra}

The primary purpose of this section is to construct a rank-one torsion-free monoid~$M$ that is universal in the following way: for every field $F$ the monoid algebra $F[M]$ satisfies the almost ACCP but not the ACCP. We split our work into three parts: first, we construct the monoid $M$, then we introduce some functions on $M$ with some convenient additive properties, and finally we prove the universal property of $M$.

\subsection{The Exponent Monoid} Let $(p_n)_{n \ge 1}$ and $(q_n)_{n \ge 1}$ be two strictly increasing sequences consisting of primes such that $100 < q_1$ and $q_n < p_n < q_{n+1}$ for every $n \in \nn$. Now, set
\[
    \alpha_n := \frac{q_n}{p_n}
\]
for every $n \in \nn$, and assume that the terms of the sequences $(p_n)_{n \ge 1}$ and $(q_n)_{n \ge 1}$ have been chosen large enough so that $\alpha_1 < \frac1{200}$ and also that $\alpha_{n+1} < \frac12 \alpha_n$ for every $n \in \nn$. Therefore $(\alpha_n)_{n \ge 1}$ is a strictly decreasing sequence such that
\begin{equation*}
	\sum_{n \in \nn} \alpha_n < \sum_{n \in \nn} \frac1{2^{n-1}}\alpha_1 = 2 \alpha_1 < \frac{1}{100},
 % \quad \text{ and } \quad \frac{1}{100} > \alpha_n > \alpha_{n+1} \ \text{ for every } \, n \in \nn.
\end{equation*}
Following~\cite[Section~3]{GL23}, we let $N$ be the atomization with respect to $(p_n)_{n \ge 1}$ of the monoid generated by the sequence $(q_n)_{n \ge 1}$ at this generating sequence; that is,
\begin{equation} \label{eq:the Puiseux monoid N}
 	N := \big\langle \alpha_n : n \in \nn \big\rangle.
\end{equation}
The following lemma provides an insight into the atomic structure of $N$ and also provides a canonical sum decomposition of the elements of $N$ that will be useful later.

\begin{lemma} \label{lem:atomicity and canonical decomposition of N}
	Let $N$ be the monoid in~\eqref{eq:the Puiseux monoid N}, and let $N_0$ be the (numerical) submonoid of $N$ generated by the set $\{q_n : n \in \nn\}$. Then the following statements hold.
	\begin{enumerate}
		\item[(1)] $N$ satisfies the ACCP.
		\smallskip
		
		\item[(2)] $\mathcal{A}(N) = \{\alpha_n : n \in \nn\}$.
		\smallskip
		
		\item[(3)] Each element $q \in N$ can be written uniquely as
		\begin{equation} \label{eq:canonical lifting decomposition}
			q = \nu(q) + \sum_{n \in \nn} c_n(q) \alpha_n,
		\end{equation}
		where $\nu(q) \in N_0$ and the map $\nn \to \nn_0$ defined via the assignment $n \mapsto c_n(q)$ has finite support and satisfies that $0 \le c_n(q) < p_n$ for every $n \in \nn$.
		\smallskip
		
		\item[(4)] For each element $q \in N$ the following conditions are equivalent, and each of these conditions implies that $|\mathsf{Z}(q)| = 1$.
		\smallskip
		\begin{enumerate}
			\item[(a)] $\nu(q) = 0$.
			\smallskip
			
			\item[(b)] Every factorization of $q$ has at most $p_n - 1$ copies of the atom $\alpha_n$ for every $n \in \nn$.
			\smallskip
			
			\item[(c)] There is a factorization of $q$ with at most $p_n - 1$ copies of the atom $\alpha_n$ for every $n \in \nn$.
		\end{enumerate}
		\smallskip
		
		\item[(5)] For any $q,r \in N$, the following inequality $\nu(q + r) \ge \nu(q) + \nu(r)$ holds.
	\end{enumerate}
\end{lemma}

\begin{proof}
	(1) First, observe that every submonoid of a reduced monoid satisfying the ACCP must satisfy the ACCP. Since $N$ is a submonoid of the Puiseux monoid $\big\langle \frac1{p_n} : n \in \nn \big\rangle$ and the latter satisfies the ACCP by \cite[Proposition~4.2.2]{fG22}, it follows that $N$ satisfies the ACCP.
	\smallskip
	
	(2) Since $N$ is reduced, it follows that $\mathcal{A}(N) \subseteq \{\alpha_n : n \in \nn\}$. The reverse inclusion follows immediately from the fact that $\gcd(\mathsf{d}(\alpha_m), \mathsf{d}(\alpha_n)) = \gcd(p_m, p_n) = 1$ for all distinct $m,n \in \nn$.
	\smallskip
	
	(3) Fix $q \in N$, and let us check that $q$ can be written as in~\eqref{eq:canonical lifting decomposition}. If $q=0$, it is clear that it has a trivial and unique sum decomposition as in~\eqref{eq:canonical lifting decomposition}. Assume, therefore, that $q > 0$. Since $N$ is atomic, we can write $q = \sum_{n \in \nn} c'_n \alpha_n$ for coefficients $c'_1, c'_2, \dots$ in $\nn_0$ such that $c'_n = 0$ for almost all $n \in \nn$. For each $n \in \nn$, take $b_n, c_n \in \nn_0$ with $c_n < p_n$ such that $c'_n = b_n p_n + c_n$. Then after setting $\nu(q) := \sum_{k \in \nn} b_k p_k \alpha_k = \sum_{k \in \nn} b_k q_k \in N_0$ (note that $b_k = 0$ for almost all $k \in \nn$) and $c_n(q) := c_n$ for every $n \in \nn$, we obtain a sum decomposition of $q$ as specified in~\eqref{eq:canonical lifting decomposition}. For the uniqueness of the sum decomposition, notice that if $\nu'(q) + \sum_{n \in \nn} c'_n(q) \alpha_n$ is also a sum decomposition of $q$ as specified in~\eqref{eq:canonical lifting decomposition}, then for each $n \in \nn$, after applying $p_n$-adic valuation to both sides of $\nu(q) - \nu'(q) + \sum_{n \in \nn} (c_n(q) - c'_n(q)) \alpha_n = 0$, the inequality $|c_n(q) - c'_n(q)| < p_n$ ensures that $c'_n(q) = c_n(q)$: this clearly implies that both sum decompositions are the same.
	\smallskip

    (4) We proceed to argue that the three conditions stated in part~(4) are equivalent.
	\smallskip
 
	(a) $\Rightarrow$ (b): Observe that if $q$ has a factorization having at least $p_n$ copies of the atom $\alpha_n$ for some $n \in \nn$, then we can proceed as we did in the proof of the existence of the sum decomposition~\eqref{eq:canonical lifting decomposition} to obtain that the same sum decomposition of $q$ satisfies $\nu(q) \ge q_n > 0$.
	\smallskip
	
	(b) $\Rightarrow$ (c): This follows trivially from the fact that $N$ is atomic.
	\smallskip
	
	(c) $\Rightarrow$ (a): Clearly, if $\sum_{n \in \nn} c_n \alpha_n$ is a factorization of $q$ in $N$ with $c_n \le p_n - 1$ for every $n \in \nn$, then it can be interpreted as the unique sum decomposition of $q$ in~\eqref{eq:canonical lifting decomposition} and, therefore, the equality $\nu(q) = 0$ must hold.
	\smallskip
	
	Now let us show that condition~(c) implies that $|\mathsf{Z}(q)| = 1$. To do so, take $\sum_{i=1}^n c_i \alpha_i \in \mathsf{Z}(q)$ with $c_i \le p_i - 1$ for every $i \in \nn$. Then $\sum_{i=1}^n c_i \alpha_i$ must be the sum decomposition of $q$ in~\eqref{eq:canonical lifting decomposition}, which satisfies $\nu(q) = 0$.  On the other hand, observe that if $q$ has a factorization having at least $p_n$ copies of the atom $\alpha_n$ for some $n \in \nn$, then we can proceed as we did in the proof of the existence of the sum decomposition~\eqref{eq:canonical lifting decomposition} to obtain that the same sum decomposition of $q$ satisfies $\nu(q) \ge q_n > 0$. Thus, by uniqueness of the sum decomposition in~\eqref{eq:canonical lifting decomposition}, every other factorization of $q$ must contain at most $p_n - 1$ copies of the atom $\alpha_n$, and so it must be the same as $\sum_{i=1}^n c_i \alpha_i$. 
	\smallskip
	
	(5) Fix $q,r \in N$. Let the canonical sum decomposition of $q$ be as in~\eqref{eq:canonical lifting decomposition}, and let the right-hand side of the equality
	\[
		r = \nu(r) + \sum_{n \in \nn} c_n(r) \alpha_n
	\]
	be the canonical sum decomposition~\eqref{eq:canonical lifting decomposition} of $r$. For each $n \in \nn$, write $c_n(q) + c_n(r) = d_n p_n + r_n$ for some $d_n, r_n \in \nn_0$ with $r_n < p_n $. Now observe that from
	\[
		q+r = \nu(q) + \nu(r) + \sum_{n \in \nn} (c_n(q) + c_n(r)) \alpha_n
	\]
	we can obtain the canonical sum decomposition of $q+r$ if, for each $n \in \nn$, we replace the coefficient $c_n(q) + c_n(r)$ by $r_n$ and add $\sum_{n \in \nn} d_n p_n \alpha_n = \sum_{n \in \nn} d_n q_n \in N_0$ to the isolated coefficient $\nu(q) + \nu(r)$. As a consequence, $\nu(q+r) = \nu(q) + \nu(r) + \sum_{n \in \nn} d_n p_n \alpha_n \ge \nu(q) + \nu(r)$, as desired.
	%
	%In the language of~\cite[Subsection~3.1]{GR23}, the monoid $N$ in~\eqref{eq:sequence of alpha_n} is the lifting monoid of $\langle q_n : n \in \nn \rangle$ with respect to the lifting map defined by the assignments $q_n \mapsto (p_n, \nn_0)$, and so it follows from \cite[Proposition~3.2]{GR23} that each element $\alpha \in N$ can be uniquely written as in the statement of the lemma (it is not difficult to argue the existence and uniqueness of the canonical sum decomposition directly).
\end{proof}
\smallskip

We call the right-hand side of the equality~\eqref{eq:canonical lifting decomposition} the \emph{canonical decomposition} of $q$ in $N$. Thus, with notation as in Lemma~\ref{lem:atomicity and canonical decomposition of N}, we can define the functions $\nu, \sigma \colon N \to \nn_0$ via the assignments
\[
	\nu \colon q \mapsto \nu(q) \quad \text{and} \quad \sigma \colon q \mapsto \sum_{n \in \nn} c_n(q).
\]
%
%\begin{lemma} \label{lem:additivity of nu}
%	For any $q,r \in N$, the following inequality $\nu(q + r) \ge \nu(q) + \nu(r)$ holds.
%\end{lemma}
%
%\begin{proof}
%	Fix $q,r \in N$. Let the canonical decomposition of $q$ be as in~\eqref{eq:canonical lifting decomposition}, and let the right-hand side of the equality
%	\[
%		r = \nu(r) + \sum_{n \in \nn} c_n(r) \alpha_n
%	\]
%	be the canonical decomposition of $r$ in $N$. For each $n \in \nn$, write $c_n(q) + c_n(r) = d_n p_n + r_n$ for some $d_n, r_n \in \nn_0$ with $r_n < p_n $. Now observe that from
%	\[
%		q+r = \nu(q) + \nu(r) + \sum_{n \in \nn} (c_n(q) + c_n(r)) \alpha_n
%	\]
%	we can obtain the canonical decomposition of $q+r$ in $N$ if, for each $n \in \nn$, we replace the coefficient $c_n(q) + c_n(r)$ by $r_n$ and add $\sum_{n \in \nn} d_n p_n \alpha_n \in \nn_0$ to the isolated coefficient $\nu(q) + \nu(r)$. As a consequence, $\nu(q+r) = \nu(q) + \nu(r) + \sum_{n \in \nn} d_n p_n \alpha_n \ge \nu(q) + \nu(r)$, as desired.
%\end{proof}

We proceed to introduce some subsets of rationals, defined in terms of $\mathcal{A}(N)$, which are also needed in the main construction of this section. These are the following sets:

\begin{equation*}
	A_1 :=  \bigg\{ \sum_{i=1}^\ell \alpha_{j_i} : \ell \in \nn \text{ and } j_1, j_2, \dots,j_\ell \in \nn \text{ with } j_1 < \dots < j_\ell \bigg\},
\end{equation*}
\begin{equation*}
	A_2 := \bigg\{ \alpha_{j_k} \, + \, \sum_{i=1}^\ell \alpha_{j_i} : \ell \in \nn, \ k \in \ldb 1, \ell \rdb, \text{ and distinct }j_1, j_2, \dots,j_\ell \in \nn \bigg\},
\end{equation*}
\begin{equation*}
	A_3 := \bigg\{ 2\alpha_{j_k} \! + \, \sum_{i=1}^\ell \alpha_{j_i} : \ell \in \nn, \ k \in \ldb 1, \ell \rdb, \text{ and distinct }j_1,j_2,\ldots,j_\ell\in \nn \bigg\},
\end{equation*}
\begin{equation*}
	B=\bigg\{ \beta_0 := 1,\beta_\ell := 1 - \sum_{i = 1}^\ell \alpha_i : \ell \in \nn \bigg\}.
\end{equation*}
Since $\sum_{n \in \nn} \alpha_n < \frac1{100}$, it is clear that $A_1 \cup A_2 \cup A_3 \subset \big(0,\frac12\big)$ and $B \subset \big(\frac12, 1\big]$. We encapsulate some further observations about the sets $A_1, A_2$, and $A_3$ in the following lemma.

\begin{lemma} \label{lem:sets A_1, A_2, A_3}
	The following statements hold.
	\begin{enumerate}
		\item $A_1, A_2$, and $A_3$ are mutually disjoint sets.
		\smallskip

        \item $B$ and $\langle A_1 \sqcup A_2 \sqcup A_3 \rangle$ are disjoint sets.
        \smallskip
		
		\item $N = A_1 \sqcup \langle A_2 \sqcup A_3 \rangle$.
	\end{enumerate}
\end{lemma}

\begin{proof}
	(1) Observe that the inequality $\min\{p_n : n \in \nn\} = p_1 > 3$ guarantees that, for each $i \in \ldb 1,3 \rdb$ and any element $q \in A_i$, the equality $\nu(q) = 0$ holds, and there exists $j \in \nn$ such that $c_j(q) = i$ while $c_n(q) \in \{0,1\}$ for any $n \neq j$. Thus, the uniqueness of the canonical decomposition of $q$ in $N$ ensures that $A_1$, $A_2$, and $A_3$ are mutually disjoint.
	\smallskip

    (2) Suppose, by way of contradiction, that $\beta_n \in \langle A_1 \sqcup A_2 \sqcup A_3 \rangle$ for some $n \in \nn_0$. In this case, $1 \in N$. Also, note that for each $n \in \nn$ the fact that the $p_n$-adic valuation of $1$ is nonnegative implies that $c_n(1) = 0$. Therefore $1 = \nu(1)$, and so $1 = \nu(1) \in \langle q_n : n \in \nn \rangle$. However, the smallest positive integer of the numerical monoid $\langle q_n : n \in \nn \rangle$ is $q_1$, and so the inequality $q_1 > 1$ generates a contradiction.
	\smallskip
 
	(3) To see that $A_1$ and $\langle A_2 \sqcup A_3 \rangle$ are disjoint we use the uniqueness of the canonical decomposition~\eqref{eq:canonical lifting decomposition}: first observe that for each $q \in A_1$, the equality $\nu(q) = 0$ holds and $c_n(q) \in \{0,1\}$ for every $n \in \nn$, while for each $r \in \langle A_2 \sqcup A_3 \rangle$, either $\nu(r) > 0$ or $c_n(r) \ge 2$ for some $n \in \nn$. 
 
    Let us argue now that $N \subseteq A_1 \sqcup \langle A_2 \sqcup A_3 \rangle$. The fact that $p_n \ge 2$ for every $n \in \nn$ implies that $q_j = p_j \alpha_j \in \langle 2\alpha_j, 3 \alpha_j \rangle \subset \langle A_2 \sqcup A_3 \rangle$ for every $j \in \nn$, and so $N_0 \subseteq \langle A_2 \sqcup A_3 \rangle$. Hence we only need to verify that each $q \in N$ with $\nu(q) = 0$ belongs to $A_1 \sqcup \langle A_2 \sqcup A_3 \rangle$. This follows immediately by induction on $\sum_{n \in \nn} c_n(q)$ as if $c_n(q) \in \{0,1\}$ for every $n \in \nn$ then $q \in A_1$ and, otherwise, there exists an index $k \in \nn$ such that $q = (q - 2 \alpha_k) + 2 \alpha_k \in A_1 \sqcup \langle A_2 \sqcup A_3 \rangle + A_2 \subseteq A_1 \sqcup \langle A_2 \sqcup A_3 \rangle$ by the inductive hypothesis.
 % This also follows from the uniqueness of the canonical decomposition as in the canonical decomposition of a nonzero element $q \in \langle A_2 \sqcup A_3 \rangle$, either $\nu(q) > 0$ or $c_n(q) \ge 2$ for some $n \in \nn$, while in the canonical decomposition of an element $q' \in A_1$, the equality $\nu(q') = 0$ holds and $c_n(q') \in \{0,1\}$ for every $n \in \nn$.
\end{proof}

The Puiseux monoid that will play a crucial role in the forthcoming construction of one-dimensional monoid algebras that satisfy the almost ACCP is that generated by the set $A_2 \cup A_3 \cup B$. Therefore we set
\begin{equation} \label{eq:main monoid}
	A := A_2 \sqcup A_3 \quad \text{ and } \quad M := \langle A \cup B \rangle.
\end{equation}
As the following proposition indicates, the monoid $M$ is atomic but does not satisfy the ACCP.

\begin{prop}\label{prop:M atomic} %\label{lem:M not ACCP}
	For the monoid $M$ in~\eqref{eq:main monoid}, the following statements hold.
	\begin{enumerate}
		\item $M$ is atomic with $\mathcal{A}(M) = A \cup B$.
		\smallskip
		
		\item $M$ does not satisfy the ACCP.
	\end{enumerate}
\end{prop}

\begin{proof}
	(1) It suffices to show that $A \cup B \subset \mathcal{A}(M)$. Now if $b := a_1 + \dots + a_\ell$ for some $a_1, \dots, a_\ell \in A$ and $\ell \ge 2$, then one of the following conditions must hold:
	\begin{itemize}
		\item $\nu(b) \ge 1$;
		\smallskip
		
		\item $c_i(b) \ge 4$ for some $i \in \nn$;
		\smallskip
		
		\item $c_i(b) \ge 2$ and $c_j(b) \ge 2$ for some $i,j \in \nn$ with $i \neq j$.
	\end{itemize}
	On the other hand, observe that none of the elements of $A$ satisfies any of these three conditions, whence $\alpha \notin \langle A \setminus \{\alpha\} \rangle$ for any $\alpha \in A$. In addition, none of the elements of $B$ can divide any element of $A$ in~$M$ because $\sup A < \frac12 < \inf B$. Thus, we can conclude that $A \subseteq \mathcal{A}(M)$. 
	
	Considering the canonical decomposition of each $b' \in \langle A \rangle$, we see that either $\nu(b') \ge 1$ or $c_n(b') \ge 2$ for some $n \in \nn$. Thus, $\sum_{i=n}^{n+d} \alpha_i \notin \langle A \rangle$ for any $n \in \nn$ and $d \in \nn_0$. Now fix $k \in \nn_0$, and let us argue that $\beta_k \in \mathcal{A}(M)$. It follows from part~(2) of Lemma~\ref{lem:sets A_1, A_2, A_3} that $\beta_n \notin \langle A \rangle$ for any $n \in \nn_0$. Now since $(\beta_n)_{n \ge 0}$ is a strictly decreasing sequence whose underlying set, $B$, is contained in $(\frac 12, 1]$, if we write $\beta_k$ as a sum of elements of $A \cup B$, then at most one term $\beta_\ell$ (and exactly one because $\beta_k \notin \langle A \rangle$) can appear as a summand, in which case, $\ell \ge k$ and we see that $\sum_{i=k+1}^\ell \alpha_i \in \langle A \rangle$, which can only happen if $\ell = k$. Therefore $\beta_k \in \mathcal{A}(M)$, as desired. Hence the inclusion $B \subseteq \mathcal{A}(M)$ also holds, and so $M$ is atomic with $\mathcal{A}(M) = A \cup B$.
	\smallskip
	
	(2) This part follows immediately: indeed, the equality $2 \beta_n = 2 \beta_{n+1} + 2 \alpha_{n+1}$ holds for every $n \in \nn_0$, and this implies that the ascending chain of principal ideals $(2 \beta_n + M)_{n \ge 0}$ of $M$ does not stabilize.
\end{proof}

%For the next lemma, recall that $\{\alpha_n : n \in \nn\}$ is the set of atoms of the atomic submonoid $N$ of $\qq_{\ge 0}$ introduced in~\eqref{eq:the Puiseux monoid N}, and also that $\alpha_n = \frac{q_n}{p_n}$ for every $n \in \nn$.

We conclude this section proving the following lemma.

\begin{lemma} \label{lem:generating set of A}
	%Let the sequence $(\alpha_n)_{n \ge 1}$ be as in~\eqref{eq:sequence of alpha_n}. 
	For any $b_1, \dots, b_\ell \in \nn_0$, if $\max \{b_1, \dots, b_\ell\} \ge 2$, then $\sum_{i=1}^\ell b_i \alpha_i \in \langle A \rangle$.
\end{lemma}

%\begin{lemma} \label{lem:generating set of A}
%	With notation as in \eqref{eq:canonical lifting decomposition}, the following equality holds:
%	\[
%		N \setminus \langle A \rangle = \big\{q \in N : \nu(q) = 0 \ \text{and} \ c_n(q) \in \{0,1\} \text{ for every } n \in \nn \big\}.
%	\]
%\end{lemma}
%
%\begin{proof}
%	Set $T : \big\{q \in N : \nu(q) > 0 \ \text{or} \ c_n(q) \in \{0,1\} \text{ for every } n \in \nn \big\}$. By definition, $T \subseteq N$. Now fix $q \in T$, and let 
%	\[
%		q = \sum_{n \in \nn} c_n(q) \alpha_n
%	\]
%	be the canonical decomposition of $q$ in $N$: as $q \in T$, it follows that $c_n(q) \in \{0,1\}$ for every $n \in \nn$. TODO: finish...
%\end{proof}

\begin{proof}
	Let $U_N$ be the subset of $N$ consisting of all the elements having a unique factorization; that is, $U_N := \{q \in N : |\mathsf{Z}(q)| = 1\}$. Now consider the map $\omega \colon U_N \to \nn_0$ defined by
	\[
		\omega\bigg( \sum_{n \in \nn} c_n \alpha_n \bigg) = \sum_{n \in \nn} \max\{c_n - 1, 0\}
	\]
	for every $ \sum_{n \in \nn} c_n \alpha_n \in U_N$. Since $U_N$ consists of elements having exactly one factorization, the map~$\omega$ is well defined.
%	
%	First observe that for any $c_1, \dots, c_\ell \in \nn_0$, if $c_j \ge p_1$ then $c_j \ge 3$,$\max \{c_1, \dots, c_\ell\} \ge p_1$, then... Suppose that Set $S := \big\{ \sum_{i=1}^\ell c_i \alpha_i : c_1, \dots, c_\ell \in \nn_0 \text{ and } \max \{c_1, \dots, c_\ell\} \ge 2 \big\}$, and define the map $\omega \colon S \to \nn_0$ as follows: for each $a = \sum_{i=1}^\ell c_i \alpha_i  \in S$, where $c_1, \dots, c_\ell \in \nn_0$ are such that $\max \{c_1, \dots, c_\ell\} \ge 2$, set
%	\[
%		\omega(a) := \sum_{i=1}^n \max\{c_i - 1, 0\}.
%	\]
	We now prove by induction on $\omega(q)$ that the set
	\[
		S := \bigg\{ \sum_{n \in \nn} c_n \alpha_n \in U_N : \max\{c_n : n \in \nn \} \ge 2 \bigg\}
	\]
	is a subset of $\langle A \rangle$. If $\omega(q) = 1$ for some $q \in S$, then it is clear that $q \in A_2 \subseteq \langle A \rangle$. Also, if $\omega(q) = 2$ for some $q \in S$, then either $q \in A_3 \subseteq \langle A \rangle$ or $q \in A_2 + A_2 \subseteq \langle A \rangle$. Now suppose there exists $k \in \nn$ with $k \ge 3$ such that $r \in \langle A \rangle$ for all $r \in S$ with $\omega(r) < k$, and take $q := \sum_{n \in \nn} c_n \alpha_n \in S$ with $\omega(q) = k$. Fix an index $i \in \nn$ such that $c_i \ge 2$. From a combination of the inequalities $c_i \ge 2$ and $\omega(q) \ge 3$ we infer that $q - 2 \alpha_i \in S$. In addition, $\omega(q - 2 \alpha_i) < \omega(q) = k$, and so our induction hypothesis ensures that $q - 2 \alpha_i \in \langle A \rangle$. Thus, $q = 2 \alpha_i + (q - 2 \alpha_i) \in A_2 + \langle A \rangle \subseteq \langle A \rangle$. Therefore $S \subseteq \langle A \rangle$.
	
	% Now, set $N_0 := \langle q_n : n \in \nn \rangle$. Let us verify that $N_0 \subseteq \langle A \rangle$. First, notice that $2\alpha_n, 3 \alpha_n \in \langle A \rangle$ for every $n \in \nn$. Thus, for each $n \in \nn$, the equality $q_n := p_n \alpha_n$, along with the fact that $p_n \in \langle 2,3 \rangle$, ensures that $q_n = p_n \alpha_n \in \langle 2\alpha_n, 3 \alpha_n \rangle \subseteq \langle A \rangle$. Since $\{q_n : n \in \nn\}$ is a generating set of $N_0$, the inclusion $N_0 \subseteq \langle A \rangle$ must hold.
    
	Finally, pick $b_1, \dots, b_\ell \in \nn_0$ such that $\max \{b_1, \dots, b_\ell\} \ge 2$, and let us argue that $q := \sum_{i=1}^\ell b_i \alpha_i \in \langle A \rangle$. If $\nu(q) = 0$, then it follows from part~(3) of Lemma~\ref{lem:atomicity and canonical decomposition of N} that $q$ is an element of $N$ having a unique factorization, and so $q \in S \subseteq \langle A \rangle$. Thus, we assume that $\nu(q) > 0$. As $q - \nu(q)$ is an element of~$N$ having a unique factorization (also by part~(3) of Lemma~\ref{lem:atomicity and canonical decomposition of N}), if the only factorization of $q - \nu(q)$ in~$N$ repeats at least one atom, then $q - \nu(q) \in S \subseteq \langle A \rangle$, which implies that $q \in \nu(q) + \langle A \rangle \subseteq N_0 + \langle A \rangle \subseteq \langle A \rangle$ (the inclusion $N_0 \subseteq \langle A \rangle$ was already argued in the proof of part~(3) of Lemma~\ref{lem:sets A_1, A_2, A_3}). Therefore suppose that the only factorization of $q - \nu(q)$ in $N$ is the sum of distinct atoms. Now write $\nu(q) = \sum_{n \in \nn} d_n q_n$, where $d_n = 0$ for all but finitely many $n \in \nn$, and take an index $j$ such that $d_j > 0$. Thus,
	\begin{equation} \label{eq:lemm max coeff aux 1}
		\nu(q) - q_j \in N_0 \subseteq \langle A \rangle.
	\end{equation}
	Since $q - \nu(q)$ is the sum of distinct atoms of $N$, we see that $2\alpha_j + (q - \nu(q)) \in A_2 \cup A_3$. Moreover, because $p_j - 2 \in \langle 2,3 \rangle$, we see that
	\begin{equation} \label{eq:lem max coeff aux 2}
		q_j + (q - \nu(q)) = (p_j - 2) \alpha_j + \big(2\alpha_j + (q - \nu(q)) \big) \in \langle 2\alpha_j, 3 \alpha_j \rangle + A_2 \cup A_3 \subseteq \langle A \rangle.
	\end{equation}
	Finally, we can use \eqref{eq:lemm max coeff aux 1} and~\eqref{eq:lem max coeff aux 2} to obtain that $q = \big( \nu(q) - q_j \big) + \big( q_j + (q - \nu(q))\big) \in \langle A \rangle + \langle A \rangle \subseteq \langle A \rangle$.
\end{proof}

\smallskip
%%%%%%%%%%%%%%%%%%%%%%%%%
\subsection{Integer-Valued Functions on $M$}

Before proving that $F[M]$ satisfies the almost ACCP for every field $F$, which we shall do in the next section, some further preliminaries are needed. We proceed to introduce, for each $n \in \nn_0 \cup \{\infty\}$, an integer-valued function $\iii_n$ on $M$. These functions will play a crucial role in the proof of our primary result. We start by the case $n=\infty$. Recall that $\beta_0 := 1$ and $\beta_n = 1 - \sum_{i = 1}^n \alpha_i$ for every $n \in \nn$.

\begin{definition}
	The map $\ifi \colon M \to \nn_0$ defined by
	\begin{equation*}
		\ifi(b) := \max \Big\{ m \in \nn_0 : \text{there exist } j_1, j_2, \dots, j_m \in \nn_0 \text{ such that } \sum_{i=1}^m \beta_{j_i} \mid_M b \Big\}
	\end{equation*}
	is called the \emph{indicator map} of $M$ and, for each $b \in M$, the value $\ifi(b)$ is called the \emph{indicator of $b$ at infinity}.
\end{definition}

We are using here the convention that a summation of real numbers over an empty set is $0$. This, along with the fact that $\beta_n > \frac12$ for all $n \in \nn_0$, guarantees that the indicator map in the previous definition is well defined. In the following lemma we will encapsulate some useful properties of the map $\ifi$. %Moreover, as the following lemma indicate the same indicator map is superadditive.

\begin{lemma} \label{lem:indicator infty} %\label{lem:indicator infty sum} \label{lem:indicator infty > 1}} \label{lem:indicator infty = 1}
	Let $M$ be as in~\eqref{eq:main monoid}. Then the following statements hold.
	\begin{enumerate}
		\item $\ifi(b+c) \ge \ifi(b) + \ifi(c)$ for all $b,c \in M$.
		\smallskip
		
		\item If $\ifi(b)\ge 2$ for some $b \in M$, then for all sufficiently large $n \in \nn$ we have $\beta_n \mid_M b$.
		\smallskip
		
		\item If $\ifi(b)=1$ for some $b \in M \setminus B$, then for all sufficiently large $n \in \nn$ we have $\beta_n \mid_M b$.
	\end{enumerate}
\end{lemma}

\begin{proof}
	(1) Fix $b,c \in M$, and then set $m := \ifi(b)$ and $n := \ifi(c)$. Now take nonnegative integers $j_1, \dots, j_m$ and $k_1, \dots, k_n$ such that $\sum_{i = 1}^m \beta_{j_i} \mid_M b$ and $\sum_{i=1}^n \beta_{k_i} \mid_M c$. Then $\sum_{i = 1}^m \beta_{j_i} + \sum_{i = 1}^n \beta_{k_i}$ divides $b+c$ in~$M$ and, therefore, $\ifi(b+c) \ge m+n = \ifi(b) + \ifi(c)$.
	\smallskip
	
	(2) Suppose now that the inequality $\ifi(b)\ge 2$ holds for some $b \in M$. Then there exist nonnegative integers $m_1, \dots, m_t$ with $t \ge 2$ such that $\sum_{i=1}^t \beta_{m_i} \mid_M b$. Now take $n \in \nn$ such that the inequality $n \ge \max\{m_1, \dots, m_t\} + 1$ holds, and then write
    \[
        \beta_{m_1} = \beta_n + \sum_{i=m_1+1}^n \alpha_i \quad \text{ and } \quad \beta_{m_2} = \beta_n + \sum_{i=m_2+1}^n \alpha_i.
    \]
    Observe that $2 \beta_n + \big( \sum_{i=m_1+1}^n \alpha_i + \sum_{i=m_2+1}^n \alpha_i \big) \mid_M b$. On the other hand, as there are two copies of $\alpha_n$ in the sum $\sum_{i=m_1+1}^n \alpha_i + \sum_{i=m_2+1}^n \alpha_i$, Lemma~\ref{lem:generating set of A} ensures that $\sum_{i=m_1+1}^n \alpha_i + \sum_{i=m_2+1}^n \alpha_i \in \langle A \rangle \subseteq M$. As a result, $2 \beta_n \mid_M b$, which implies that $\beta_n \mid_M b$.
	\smallskip
	
	(3) Finally, suppose that $\ifi(b) = 1$ for some $b \in M$. Take $i \in \nn$ such that $\beta_i \mid_M b$, and then write $b = \beta_i + a$ for some nonzero $a \in \langle A \rangle$. Fix $j \in \nn$ such that $\alpha_j$ divides $a$ in $\langle A \rangle$. Then $\beta_i + \alpha_j \mid_M b$, and proceeding as we did to argue part~(2), we will obtain that $\beta_n \mid_M b$ for each $n \in \nn$ such that $n > \max\{i, j\}$.
\end{proof}

%\begin{lemma}\label{lem:indicator infty sum}
%	Let $M$ be as in~\eqref{eq:main monoid}. Then $\ifi(b+c) \ge \ifi(b) + \ifi(c)$ for all $b,c \in M$.
%\end{lemma}
%
%\begin{proof}
%	Fix $b,c \in M$, and set $m := \ifi(b)$ and $n := \ifi(c)$. Now take nonnegative integers $j_1, \dots, j_m$ and $k_1, \dots, k_n$ such that $\sum_{i = 1}^m \beta_{j_i} \mid_M b$ and $\sum_{i=1}^n \beta_{k_i} \mid_M b$. Then $\sum_{i = 1}^m \beta_{j_i} + \sum_{i = 1}^n \beta_{k_i}$ divides $b+c$ in~$M$ and, therefore, $\ifi(b+c) \ge m+n = \ifi(b) + \ifi(c)$.
%\end{proof}
%
%\begin{lemma}\label{lem:indicator infty > 1}
%	If $\ifi(x)\ge 2$, then for all sufficiently large $n \in \nn$ we have $\beta_n \mid_M x$.
%\end{lemma}
%
%\begin{proof}
%	There exist $\beta_{m_1},\beta_{m_2},\ldots,\beta_{m_t}(t\ge 2)\text{ such that } (\beta_{m_1}+\beta_{m_2}+\cdots+\beta_{m_t})\mid_M x$, then for all $ n\ge \max(m_1,m_2,\ldots,m_t)+1$, we have $2\beta_n\mid_M x$ by \ref{lem:generating set of A}, which implies $\beta_n\mid_M x$.
%\end{proof}
%
%\begin{lemma}\label{lem:indicator infty = 1}
%	If $\ifi(x)=1$ and $x \not\in B$, then for all sufficiently large $n \in \nn$ we have $\beta_n \mid_M x$.
%\end{lemma}
%
%\begin{proof}
%	Write $x=\beta_i+a$ for some $a\in \langle A \rangle\setminus \{ 0 \}$ that contains $\alpha_j$ in its decomposition, then for all $ n > \max(i, j)$ we have $\beta_n\mid_M x$ by \ref{lem:generating set of A}.
%\end{proof}

Before introducing the map $\iii_n$ for every $n \in \nn$, we need to argue the following lemma.

\begin{lemma} \label{lem:existence and uniqueness of s to define I_n}
    For each $b \in M$ and $n \in \nn$, there exists a unique $s \in \zz$ such that 
	\begin{equation} \label{eq:conditions for s}
		- \frac{p_n-1}{2} \le s \le \frac{p_n-1}{2} \quad \text{ and } \quad v_{p_n}(b - s\alpha_n) \ge 0.
	\end{equation}
\end{lemma}

\begin{proof}
Fix $b \in M$ and $n \in \nn$. Since $M$ is generated by the set $A \cup B$, we can write $b = c_0 + \sum_{i=1}^\ell c_i \alpha_i$ for some index $\ell \in \nn$ with $\ell \ge n$ and coefficients $c_0 \in \nn_0$ and $c_1, \dots, c_n \in \zz$. From the fact that $\alpha_n = \frac{q_n}{p_n}$ (with $\gcd(p_n, q_n) = 1$), we obtain that for each $s \in \zz$, the inequality $v_{p_n} (b - s \alpha_n) \ge 0$ holds if and only if $p_n \mid c_n-s$, which immediately implies the existence and uniqueness of $s \in \zz$ satisfying the conditions in~\eqref{eq:conditions for s}.
    % Observe that $v_{p_n}(b - s\alpha_n) \ge 0$ if and only if $p_n$ does not divide the denominator of
    % \[
    %     b - s \alpha_n = \frac{p_n \mathsf{n}(b) - s \mathsf{d}(b)q_n}{\mathsf{d}(b)p_n}.
    % \]
    % If $p_n \mid \mathsf{d}(b)$, then we see that $p_n \nmid \mathsf{d}(b - s \alpha_n)$ if and only if $p_n \mid \mathsf{n}(b) - s \frac{\mathsf{d}(b)}{p_n} q_n$, and the fact that $\gcd\big(p_n, \frac{\mathsf{d}(b)}{p_n} q_n \big) = 1$ ensures that $\big\{ s \frac{\mathsf{d}(b)}{p_n} q_n : - \frac{p_n - 1}2 \le s \le \frac{p_n - 1}2 \big\}$ is a complete residue system module $p_n$, which in turn ensures the existence and uniqueness of $s \in \zz$ with $- \frac{p_n - 1}2 \le s \le \frac{p_n - 1}2$ such that $p_n \mid \mathsf{n}(b) - s \frac{\mathsf{d}(b)}{p_n} q_n$. On the other hand, if $p_n \nmid \mathsf{d}(b)$, then we see that $p_n \nmid \mathsf{d}(b - s \alpha_n)$ if and only if $p_n \mid p_n \mathsf{n}(b) - s \mathsf{d}(b) q_n$, which happens precisely when $s=0$ because $\gcd(p_n, \mathsf{d}(b)q_n) = 1$.
\end{proof}

We are in a position to introduce the map $\iii_n \colon M \to \zz$ for every $n \in \nn$.

\begin{definition}
	For each $n \in \nn$, the \emph{indicator map at} $n$ is the map $\iii_n \colon M \to \zz$ defined as follows: for $b \in M$, we let $\iii_n(b)$ be the unique $s \in \zz$ such that both conditions in \eqref{eq:conditions for s} hold.
	% \begin{equation*}
	% 	- \frac{p_n-1}{2} \le s \le \frac{p_n-1}{2} \quad \text{ and } \quad v_{p_n}(b - s\alpha_n) \ge 0.
	% \end{equation*}
\end{definition}

In light of Lemma~\ref{lem:existence and uniqueness of s to define I_n}, for each $n \in \nn$, the map $\iii_n$ is well defined. As we did with the map $\ifi$, we collect some properties of the maps $\iii_n$ in the following lemma.

\begin{lemma} \label{lem:indicator n} %\label{lem:indicator n inequality} \label{lem:indicator n sum}
	Let $M$ be as in~\eqref{eq:main monoid}. Then the following statements hold.
	\begin{enumerate}
		\item For each $n \in \nn$, if $b \in M$ and $b \le \frac{q_n}{100}$, then $\lvert \iii_n(b) \rvert \le \frac{p_n}6$.
		\smallskip
		
		\item For each $n \in \nn$, if $b,c \in M$ and $b+c \le \frac{q_n}{100}$, then $\iii_n(b+c) = \iii_n(b) + \iii_n(c)$.
        \smallskip

        \item For each $n \in \nn$, if $b \in \langle A \rangle \subset N$ and $b \le \frac{q_n}{100}$, then $\iii_n(b) = c_n(b)$.
	\end{enumerate}
\end{lemma}

\begin{proof}
	(1) Take $n \in \nn$ and $b \in M$ such that $b \le \frac{q_n}{100}$. Now write $b = (\beta_{m_1} + \dots +\beta_{m_k}) + a$ for some $m_1, \dots, m_k \in \nn_0$ and $a \in \langle A \rangle$. Since
	\[
		\iii_n(b) \equiv \iii_n(\beta_{m_1} + \cdots + \beta_{m_k}) + \iii_n(a) \pmod{p_n},
	\]
	we only need to show that $-\frac{p_n}6 \le \iii_n(\beta_{m_1} + \dots + \beta_{m_k}) \le 0$ and $0 \le \iii_n(a) \le \frac{p_n}6$. We first argue that
	\begin{equation} \label{eq:inequality for i_n(beta)}
		-\frac{p_n}6 \le \iii_n(\beta_{m_1} + \dots +\beta_{m_k}) \le 0.
	\end{equation}
	Towards this end, observe that $m := |\{i \in \ldb 1, k \rdb : m_i \ge n \}|$ is the number of indices $m_i$ (with $i \in \ldb 1,k \rdb$) such that $-\alpha_n$ is a summand of $\beta_{m_i}$. Thus, for any $s \in \zz$, the inequality $v_{p_n} \big( \big(\sum_{i=1}^k \beta_{m_i}\big) - s\alpha_n \big) \ge 0$ holds if and only if $p_n \mid -m - s$, and so
	\begin{equation} \label{eq:i_n congruent to -m mod p_n}
		\iii_n(\beta_{m_1} + \dots + \beta_{m_k}) \equiv - m \pmod{p_n}.
	\end{equation}
	On the other hand, $k < 2 (\beta_{m_1} + \dots + \beta_{m_k}) < 2b$ (this is because $\frac12 < \beta_i$ for every $i \in \nn_0$), from which we obtain that
	\begin{equation*}
		0 \le m \le k < 2b \le \frac{2q_n}{100} < \frac{q_n}6 < \frac{p_n}6.
	\end{equation*}
    Therefore $-\frac{p_n - 1}2 \le -m, \iii_n(\beta_{m_1} + \dots + \beta_{m_k}) \le \frac{p_n - 1}2$, and so the congruence \eqref{eq:i_n congruent to -m mod p_n} guarantees that $\iii_n(\beta_{m_1} + \dots + \beta_{m_k}) = - m \in \big[\!-\!\frac{p_n}6, 0\big]$, which is~\eqref{eq:inequality for i_n(beta)}.
    
	Now we argue $0 \le \iii_n(a) \le \frac{p_n}6$. To do this, write $a = c_1 \alpha_1 + \dots + c_\ell \alpha_\ell$ for some $\ell \in \nn$ with $\ell \ge n$ and $c_1, \dots, c_\ell \in \nn_0$. Then $\iii_n(a) \equiv c_n \pmod{p_n}$, and so we are done because
	\begin{equation*}
		0 \le c_n \le \left\lfloor \frac{b}{\alpha_n} \right\rfloor \le \frac{b}{\alpha_n} = \frac{b}{q_n/p_n}\le \frac{q_n/100}{q_n/p_n} = \frac{p_n}{100} < \frac{p_n}{6}.
	\end{equation*}
	\smallskip
	
	(2) Now take $n \in \nn$ and $b,c \in M$ such that $b+c \le \frac{q_n}{100}$. As $\iii_n(b+c) \equiv \iii_n(b) + \iii_n(c) \pmod{p_n}$, the only way that the equality $\iii_n(b+c) = \iii_n(b) + \iii_n(c)$ does not hold is that $\lvert \iii_n(b) + \iii_n(c) \rvert > \frac{p_n-1}{2}$. However, $\lvert \iii_n(b) + \iii_n(c) \rvert \le \lvert \iii_n(b) \rvert + \lvert \iii_n(c) \rvert \le \frac{p_n}3$, where the last inequality follows from part~(1) and, on the other hand, we can see that $\frac{p_n}{3} < \frac{p_n-1}{2}$ because $p_n > 100$.
    \smallskip

    (3) Let $b = \nu(b) + \sum_{n \in \nn} c_n(b) \alpha_n$ be the canonical decomposition of $b$ in $N$, which was introduced in~\eqref{eq:canonical lifting decomposition}. Observe that $c_n(b) \le \frac{p_n-1}2$ as otherwise $b \ge c_n(b)\alpha_n > \frac{p_n-1}{2p_n}q_n > \frac{q_n}{100}$. This observation, along with the fact that $v_{p_n}(b - s\alpha_n) \ge 0$ if and only if $p_n \mid c_n(b) - s$, guarantees that $\iii_n(b) = c_n(b)$, as desired.
\end{proof}

\medskip
%%%%%%%%%%%%%%%%%%%
\subsection{$F[M]$ Satisfies the Almost ACCP}

Let $M$ be the Puiseux monoid introduced in~\eqref{eq:main monoid}, and let~$F$ be a field, which will be fixed for the rest of this section. It turns out that the monoid algebra $F[M]$ satisfies the almost ACCP but not the ACCP, and we dedicate the rest of this section to give a proof of this fact. Let us start by defining the map $\ifi \colon F[M]^* \to \nn_0$ as follows:
\[
	\ifi(f) := \min \big\{ \ifi(s) : s \in \text{supp}\, f \big\}
\]
for all $f \in F[M]^*$. Also, for each $n \in \nn$, set
\[
    W_n := \Big\{f \in F[M]^* : \deg f \le \frac{q_n}{100} \Big\}.
\]
Clearly, $(W_n)_{n \ge 1}$ is an ascending chain of subsets of $F[M]$ whose union is $F[M]$. For each $n \in \nn$, define the map $\iii_n \colon W_n \to \zz$ as follows:
\[
	\iii_n(f) := \max \big\{ \iii_n(s) : s \in \text{supp}\, f \big\}
\]
for all $f \in W_n$. %Here is a comment on our choice of notation.

\begin{remark}
	Although the maps $\ifi \colon M \to \nn_0$ and $\iii_n \colon M \to \zz$ (for every $n \in \nn$) introduced in the previous subsection are respectively denoted as the maps $\ifi \colon F[M]^* \to \nn_0$ and $\iii_n \colon W_n \to \zz$ introduced in this subsection, this should not cause any ambiguity given that the domains of the equally-denoted maps have different nature.
\end{remark}

We proceed to show that the maps $\ifi \colon F[M]^* \to \nn_0$ and $\iii_n \colon W_n \to \zz$ (for every $n \in \nn$) also satisfy some convenient properties. 

\begin{lemma} \label{lem:f indicator maps} %\label{lem:f indicator infty sum} %\label{lem:f indicator n sum}
	The following statements hold.
	\begin{enumerate}
		\item $\ifi(fg) \ge \ifi(f) + \ifi(g)$ for all $f,g \in F[M]^*$. 
		\smallskip
		
		\item $\iii_n(fg) = \iii_n(f) + \iii_n(g)$ for all $n \in \nn$ and $f,g \in W_n$ such that $\deg(fg) \le \frac{q_n}{100}$.
	\end{enumerate}
\end{lemma}

\begin{proof}
	(1) Fix $f,g \in F[M]^*$. Take $s \in \text{supp} \, fg$ such that $\ifi(s) = \ifi(fg)$, and then take $q \in \text{supp} \, f$ and $r \in \text{supp} \, g$ such that $s = q+r$. It follows from part~(1) of Lemma~\ref{lem:indicator infty} that
	\[
		\ifi(s) \ge \ifi(q) + \ifi(r) \ge \ifi(f) + \ifi(g)
	\]
	and, as a consequence, $\ifi(fg) = \ifi(s) \ge \ifi(f) + \ifi(g)$, as desired.
	\smallskip
	
	(2) Fix $n \in \nn$ and $f,g \in F[M]^*$ such that $\deg fg \le \frac{q_n}{100}$. Now take $s \in \text{supp} \, fg$ with $\iii_n(fg) = \iii_n(s)$, and then pick $q \in \text{supp} \, f$ and $r \in \text{supp} \, g$ such that $s = q+r$. Then we see that
	\[
		\iii_n(fg) = \iii_n(s) = \iii_n(q+r) = \iii_n(q) + \iii_n(r) \le \iii_n(f) + \iii_n(g),
	\]
	 where the last equality follows from part~(2) of Lemma~\ref{lem:indicator n} as $q+r \le \frac{q_n}{100}$. 
%	 As a consequence,
%	 \begin{align*}
%	 	\iii_n(fg) &\ge \max \{ \iii_n(q) + \iii_n(r) : q \in \text{supp} \, f \text{ and } r \in \text{supp} \, g \} \\
%	 				   &= \max\{ \iii_n(q) : q \in \text{supp} \, f \} + \max\{ \iii_n(r) : r \in \text{supp} \, g \} \\
%	 				   &= \iii_n(f)+\iii_n(g)
%	 \end{align*}
	 To argue the reverse inequality, set
	\[
		q_m := \max \big\{ q \in \text{supp} \, f : \iii_n(q) = \iii_n(f) \big\}
	\]
	and
	\[
		r_m := \max \big\{ r \in \text{supp} \, g : \iii_n(r) = \iii_n(g) \big\}.
	\]
	We claim that $q_m + r_m \in \text{supp} \, fg$; that is, $x^{q_m+r_m}$ cannot be canceled out when unfolding $fg$. To argue this, take $q \in \text{supp} \, f$ and $r \in \text{supp} \, g$ such that $q_m + r_m = q+r$. If $q > q_m$, then it follows from the maximality of $q_m$ that $\iii_n(q) < \iii_n(q_m)$ and so the fact that $\iii_n(r) \le \iii_n(r_m)$, along with part~(2) of Lemma~\ref{lem:indicator n}, would imply that
    \[
        \iii_n(q+r) = \iii_n(q) + \iii_n(r) < \iii_n(q_m) + \iii_n(r_m) = \iii_n(q_m + r_m),
    \]
    which is not possible. Thus, $q \le q_m$. In the same way, we can argue that $r \le r_m$. Therefore the equality $q_m + r_m = q+r$ implies that $q = q_m$ and $r = r_m$. Hence our claim follows: $q_m + r_m \in \text{supp} \, fg$. %indeed, if $q_m + r_m = q + r$ for some $q \in \text{supp} \, f$ and $r \in \text{supp} \, g$, then $\iii_n(q) = \iii_n(f)$ and $\iii_n(r) = \iii_n(g)$, and so the maximality in our choice of both $q_m$ and $r_m$ ensures that $q = q_m$ and $r = r_m$. This implies that $q_m + r_m \in \text{supp} \, fg$.
    As a consequence,
	\[
		\iii_n(fg) \ge \iii_n(q_m + r_m) = \iii_n(q_m) + \iii_n(r_m) = \iii_n(f) + \iii_n(g),
	\] 
	where the first equality follows from part~(2) of Lemma~\ref{lem:indicator n}. %because $q_m + r_m \le \frac{q_n}{100}$. % write this proof in terms of $r,s,t$.
\end{proof}

%\begin{lemma}\label{lem:f indicator n sum}
%	If $\deg(fg) \le q_n/100$, then $\iii_n(fg) =\iii_n(f) + \iii_n(g)$.
%\end{lemma}
%
%\begin{proof}
%	From lemma 8 we can see that $\iii_n(fg) \le \iii_n(f)+\iii_n(g)$. Now we prove the reverse side. Take $\hat{a} = \max \big\{ a \in \supp\, f \mid \iii_n(a) = \iii_n(f) \big\}, \hat{b} = \max \big\{ b \in \supp\, g \mid \iii_n(b) = \iii_n(f) \big\}$, then we see that $\hat{a} + \hat{b} \in \text{supp}(fg)$. This is because the term $x^{\hat{a}+\hat{b}}$ cannot be canceled out (if $\hat{a} + \hat{b} = a + b$, then $\iii_n(a) = \iii_n(f), \iii_n(b) = \iii_n(g)$, but we take $\hat{a}$ and $\hat{b}$ to be the maximum among this choices). So $\iii_n(fg) \ge \iii_n(\hat{a} + \hat{b}) = \iii_n(\hat{a}) + \iii_n(\hat{b}) = \iii_n(f) + \iii_n(g)$ by \ref{lem:indicator n sum}.
%\end{proof}

This is the last auxiliary result we need before proving our last main theorem.

\begin{lemma} \label{lem:f indicator ge 0 and eq 0}%\label{lem:certain f indicator n = 0} %\label{lem:any f indicator n = 0}
	For any $n \in \nn$ and $f \in W_n$, the following statements hold.
	\begin{enumerate}
		\item If $\ifi(f) = 0$, then $\iii_k(f) \ge 0$ for every $k \in \nn$ with $k \ge n$. %for all $n$ such that $\deg f \le q_n/100$.
		\smallskip
		
		\item %For each $f \in F[M]^*$, there 
		There exists $k_0 \in \nn$ such that $\iii_k(f) = 0$ for every $k \ge k_0$.
	\end{enumerate}
\end{lemma}

\begin{proof}
	(1) Suppose that $\ifi(f) = 0$. Take $q \in \text{supp} \, f$ such that $\ifi(q) = 0$, which is equivalent to the fact that no element of $B$ divides $q$ in $M$. This implies that $q \in \langle A \rangle$. Now fix $k \in \nn$ with $k \ge n$. As the sequence $(q_i)_{i \ge 1}$ is strictly increasing, it follows that $q \le \deg f \le \frac{q_n}{100} \le \frac{q_k}{100}$, and so part~(3) of Lemma~\ref{lem:indicator n} guarantees that $\iii_k(q) = c_k(q)$, whence $\iii_k(f) \ge \iii_k(q) \ge 0$.
	\smallskip
	
	(2) For each $q \in \text{supp} \, f$, one can fix $n_q \in \nn$ such that, for every $n \ge n_q$, the inequality $v_{p_n}(q) \ge 0$ holds, which means that $\iii_n(q) = 0$. Thus, after taking $k_0 \in \nn$ larger than both $n$ and $\max\{n_q : q \in \text{supp} \, f \}$, we see that whenever $k \ge k_0$ the inequality $\iii_k(f) = \max\{\iii_k(q) : q \in \text{supp} \, f\} = 0$ must hold.
    %It suffices to observe that if $\alpha_n$ does not appear in the decomposition of any element in $\text{supp} \, f$ and $\deg f \le \frac{q_n}{100}$, then the equality $\iii_n(f) = 0$ most hold.
\end{proof}

%\begin{lemma}\label{lem:certain f indicator n = 0}
%	If $\ifi(f) = 0$, then $\iii_n(f) \ge 0$ for all $n$ such that $\deg f \le q_n/100$.
%\end{lemma}
%
%\begin{proof}
%	This is because $a \in \text{supp} \, f$ for some $a \in \langle A \rangle$, in which case $\iii_n(a) \ge 0$.
%\end{proof}

We are finally in a position to prove that the monoid algebra $F[M]$ satisfies the almost ACCP for any field $F$.

\begin{theorem}\label{thm:rank 1 half-ACCP}
	For each field $F$, the monoid algebra $F[M]$ satisfies the almost ACCP.
\end{theorem}

\begin{proof}
	Let $F$ be a field, and set $R := F[M]$. Fix any nonempty finite subset $S$ of~$R$. Since $\inf B > \frac12$, there exists a maximum $\ell \in \nn$ such that we can take $m_1, \dots, m_\ell$ with $x^{\beta_{m_1} + \dots + \beta_{m_\ell}} \mid_R f$ for all $f \in S$. Since $B \subseteq \mathcal{A}(M)$ (by Proposition~\ref{prop:M atomic}), the monomial $x^{\beta_{m_i}}$ is irreducible in $F[M]$ for every $i \in \ldb 1, \ell \rdb$, and so after replacing $S$ by $S/x^{\beta_{m_1} + \dots + \beta_{m_\ell}}$, we can assume that the monomial $x^{\beta_n}$ is not a common divisor of $S$ for any $n \in \nn$. %If there exists $\beta_i$ such that $x^{\beta_i} \mid_{R} f$ for all $f \in S$, then we divide every elements in $S$ by $x^{\beta_i}$. We repeat this process, it will end in finitely many steps, because $\beta_i > \frac12$ for all $i \in \nn_0$. Now we notice that each $x^{\beta_i}$ is an atom in $R$, so we can assume there does not exist $\beta_i$ such that $x^{\beta_i} \mid_R f$ for all $f \in S$. 
	
	We proceed to show that there exists a polynomial expression $f \in S$ such that $f$ satisfies the ACCP in $F[M]$. To do so, we first set
	\[
		t_0 := \min \{ \ifi(f) : f \in S \}.
	\]
	 Observe that $t_0 \in \{0,1\}$; this is because if $t_0 \ge 2$, then part~(2) of Lemma~\ref{lem:indicator infty} would guarantee the existence of $n \in \nn_0$ such that $x^{\beta_n}$ is a common divisor of $S$ in $F[M]$, which is not possible. We split the rest of the proof into two cases, and we will argue in each of the cases that there exists an element of $S$ satisfying the ACCP.
	\smallskip
	
	\textsc{Case 1:} $t_0 = 0$. Take $f \in S$ such that $\ifi(f) = 0$. We claim that $f$ satisfies the ACCP. Suppose, by way of contradiction, that $(f_i R)_{i \ge 0}$ is an ascending chain of principal ideals of $R$ that does not stabilize with $f_0 = f$. Furthermore, we can assume that $f_i/f_{i+1}$ is not a unit of $F[M]$ for any $i \in \nn_0$. Then from part~(1) of Lemma~\ref{lem:f indicator maps} we obtain that $(\ifi(f_i))_{i \ge 0}$ is a decreasing sequence. The equality $\ifi(f) = 0$ now implies that, for every $i \in \nn_0$,
    \[
        \ifi(f_i) = 0.
    \]
    As $\lim_{n \to \infty} q_n = \infty$, in light of part~(2) of Lemma~\ref{lem:f indicator ge 0 and eq 0} we can fix $n_0 \in \nn$ large enough so that $\deg f \le \frac{q_{n_0}}{100}$ and $\iii_n(f) = 0$ for every $n \in \nn$ with $n \ge n_0$, whence $f_i \in W_n$ for every $i,n \in \nn_0$ with $n \ge n_0$. Now, for each $i \in \nn_0$, the fact that $\ifi(f_i) = 0$, along with part~(1) of Lemma~\ref{lem:f indicator ge 0 and eq 0}, ensures that $\iii_n(f_i) \ge 0$ for every $n \ge n_0$. Thus, for each $n \ge n_0$, the sequence $(\iii_n(f_i))_{i \ge 0}$ is decreasing by part~(2) of Lemma~\ref{lem:f indicator maps}. Now, in light of part~(2) of Lemma~\ref{lem:f indicator ge 0 and eq 0}, after replacing $n_0$ by a larger integer, one can further assume that for all $i,n \in \nn_0$ with $n \ge n_0$,
    \[
        \iii_n(f_i) = 0.
    \]
 
    %As a consequence, for each $i \in \nn_0$, the fact that $\ifi(f_i) = 0$, along with part~(1) of Lemma~\ref{lem:f indicator ge 0 and eq 0}, ensures that $\iii_n(f_i) = 0 $ for every $n \in \nn$ with $n \ge n_0$ 
	\medskip

	\noindent \textsc{Claim 1.} $\deg f_i \ge \alpha_{n_0 - 1} + \deg f_{i+1}$ for every $i \in \nn_0$.
	\smallskip
	
	\noindent \textsc{Proof of Claim 1.} Fix $i \in \nn_0$, and then write $f_i = f_{i+1}g$ for some $g \in F[M]$. As $\ifi(f_{i+1}) = 0$, we see that $\ifi(g) = \ifi(g) + \ifi(f_{i+1}) \le \ifi(f_i) = 0$ by part~(1) of Lemma~\ref{lem:f indicator maps} and, therefore, the fact that $\ifi(g) \in \nn_0$ guarantees that $\ifi(g) = 0$. %\min\{\ifi(q) : q \in \text{supp} \, g\} = 0$. 
    Also, if $\max\{\ifi(q) : q \in \text{supp} \, g\} \ge 1$, then one of the elements in $\text{supp} \, g$ is divisible by some element of $B$ in~$M$ and the inclusion $B \subset \big(\frac12,1\big]$ implies that $\deg g > \frac12$, whence
	\begin{equation*}
		\deg f_i = \deg g + \deg f_{i+1} > \frac12 + \deg f_{i+1} > \frac1{100} + \deg f_{i+1} >\alpha_{n_0 - 1}  + \deg f_{i+1},
	\end{equation*}
	as desired. Therefore we can assume that $\ifi(s) = 0$ for every $s \in \text{supp} \, g$. This, together with the fact that $g \notin F[M]^\times$, ensures the existence of a nonzero element $q \in \text{supp} \, g$ such that $\ifi(q) = 0$. Hence it follows from part~(1) of Lemma~\ref{lem:f indicator ge 0 and eq 0} that $\iii_n(q) \ge 0$ when $n \ge n_0$. Moreover, observe that if $n \ge n_0$, then $\iii_n(q) = 0$ as, otherwise, $\iii_n(f_i) \ge \iii_n(g) \ge \iii_n(q) > 0$. Since $\ifi(q) = 0$, it follows that $q \in \langle A \rangle \subseteq N$, and so $q$ has a canonical decomposition as in~\eqref{eq:canonical lifting decomposition}, namely, 
    \[
        q = \nu(q) + \sum_{n \in \nn} c_n(q) \alpha_n,
	\]
	where $\nu(q) \in \nn_0$ while the assignment $n \mapsto c_n(q)$ defines a map of finite support with $c_n(q) \in \ldb 0, p_n - 1 \rdb$ for every $n \in \nn$. As $q \in \langle A \rangle$ and $q \le \deg g \le \frac{q_{n_0}}{100}$, it follows from part~(3) of Lemma~\ref{lem:indicator n} that $c_n(q) = \iii_n(q) = 0$ for every $n \ge n_0$. Thus, if $\nu(q) \neq 0$ then $q \ge 1$, and if $\nu(q) = 0$ then $q \ge \alpha_j$ for some $j \in \ldb 1, n_0 - 1 \rdb$. In any case, the fact that the sequence $(\alpha_n)_{n \ge 1}$ is decreasing with $1 > \alpha_1$ ensures that $q \ge \alpha_{n_0 - 1}$. Hence
	\begin{equation*}
		\deg f_i = \deg g + \deg f_{i+1} \ge q + \deg f_{i+1} \ge \alpha_{n_0 - 1}+\deg f_{i+1},
	\end{equation*}
	and so Claim~1 is established.
	\smallskip
	
	Now set $d := \left\lceil \frac{\deg f}{\alpha_{n_0 - 1}}\right\rceil +1$ and observe that, as a consequence of Claim~1, the following inequalities $\deg f_d \le \deg f - d \alpha_{n_0 - 1} \le - \alpha_{n_0 - 1} < 0$ hold, which contradicts the fact that $f_d$ is nonzero.
	% \begin{equation*}
	% 	\deg f_{\left\lfloor \frac{\deg f}{\alpha_{n_0 - 1}}\right\rfloor +1}\le \deg f - \left(  \right)\alpha_{n_0 - 1} < 0,
	% \end{equation*}
%	which leads to a contradiction.
	\medskip

    % \textcolor{red}{Part (3) of Lemma 5.6:} If $\ifi(b)=1$ for some $b \in M \setminus B$, then for all sufficiently large $n \in \nn$ we have $\beta_n \mid_M b$.
	
	\textsc{Case 2:} $t_0=1$. It follows from part~(2) (resp., part~(3)) of Lemma~\ref{lem:indicator infty} that for all $f \in S$ with $\ifi(f) \ge 2$ (resp., $\ifi(f) = 1$) and $s \in \text{supp} \, f$ (resp., $s \in \text{supp} \, f \setminus B$) the divisibility relation $\beta_n \mid_M s$ holds for all sufficiently large $n \in \nn$. This, along with the fact that no monomial of the form $x^{\beta_n}$ is a common divisor of $S$ in $F[M]$, guarantees the existence of $f \in S$ such that $B \cap \text{supp} \, f$ is a nonempty set. We will show that $f$ satisfies the ACCP. Assume, towards a contradiction, that $(f_nR)_{\ge 0}$ is an ascending chain of principal ideals of $R$ that does not stabilize such that $f_0 = f$. As in the previous case, we can further assume that $f_i/f_{i+1}$ is not a unit of $F[M]$ for any $i \in \nn_0$, and we can infer that $(\ifi(f_i))_{i \ge 0}$ is a decreasing sequence from part~(1) of Lemma~\ref{lem:f indicator maps}. Now observe that if there exists $i \in \nn$ such that $\ifi(f_i)=0$, then we can repeat for $f_i$ the argument given in Case~1 to obtain that~$f_i$ satisfies the ACCP. Therefore, given the fact that $\ifi(f_0) = 1$, we can assume $\ifi(f_i) = 1$ for every $i \in \nn_0$. Now fix $k \in \nn_0$ such that $\beta_k \in \text{supp} \, f$.
	\medskip
	
	\noindent \textsc{Claim 2.} $\beta_k \in \text{supp} \, f_n$ for every $n \in \nn_0$.
	\smallskip
	
	\noindent \textsc{Proof of Claim 2.} We proceed by induction on~$n$. The case $n = 0$ follows from our choice of $k$. Now fix $n \in \nn$ such that $\beta_k \in f_i$ for every $i < n$. Since $\beta_k$ is an atom of $M$, if $\beta_k \notin \text{supp} \, f_n$, then from the two facts $\beta_k \in \text{supp} \, f_{n-1}$ and $f_n \mid_R f_{n-1}$ we can infer that $0 \in \text{supp} \, f_n$, which implies that $\ifi(f_n) = 0$, a contradiction. Hence Claim~2 has been established.
	\smallskip

% \color{red}
%  \begin{lemma} \label{lem:f indicator ge 0 and eq 0}%\label{lem:certain f indicator n = 0} %\label{lem:any f indicator n = 0}
% 	For any $n \in \nn$ and $f \in W_n$, the following statements hold.
% 	\begin{enumerate}
% 		\item If $\ifi(f) = 0$, then $\iii_k(f) \ge 0$ for every $k \in \nn$ with $k \ge n$. %for all $n$ such that $\deg f \le q_n/100$.
% 		\smallskip
		
% 		\item %For each $f \in F[M]^*$, there 
% 		There exists $k_0 \in \nn$ such that $\iii_k(f) = 0$ for every $k \ge k_0$.
% 	\end{enumerate}
% \end{lemma}
% \smallskip
% \color{black}

	Now take $n_0 \in \nn_{\ge k+2}$ such that $\deg f \le \frac{q_{n_0}}{100}$ and $\iii_n(f) = 0$ for all $n \ge n_0$, which can be done by virtue of part~(2) of Lemma~\ref{lem:f indicator ge 0 and eq 0}.
	\medskip
	
	\noindent \textsc{Claim 3.} $\iii_n(f_i) = 0$ and $\iii_n(f_{i-1}/f_i) = 0$ for all $i,n \in \nn$ with $n \ge n_0$. %, and $\iii_n(f_{i-1}/f_i) = 0$ for all $i,n \in \nn$ with $n \ge n_0$.
	\smallskip
	
	\noindent \textsc{Proof of Claim 3.} Fix $n \in \nn$ with $n \ge n_0$, and let us proceed by induction on~$i$. The case $i = 0$ follows from assumption. Now fix $i \in \nn$ and assume that both equalities hold for any index in $\ldb 0, i-1 \rdb$. Write $f_{i-1} = f_i g$ for some $g \in F[M]$. Observe that $g \in W_{n_0}$ because $\deg g \le \deg f$. We want to show that $\iii_n(f_i) = \iii_n(g) = 0$. %for all $n \ge n_0$. 
	Because $1 = \ifi(f_{i-1}) = \ifi(f_i)$, it follows from part~(1) of Lemma~\ref{lem:f indicator maps} that $\ifi(g) = 0$. This, along with part~(1) of Lemma~\ref{lem:f indicator ge 0 and eq 0}, guarantees that $\iii_n(g)\ge 0$. Also, since $\beta_k\in \text{supp} \, f_i$, it follows that $\iii_n(f_i)\ge \iii_n(\beta_k) = 0$. %(from our choice of $n_0$). 
    Thus, part~(2) of Lemma~\ref{lem:f indicator maps}, together with our induction hypothesis, $\iii_n(f_i) + \iii_n(g) = \iii_n(f_{i-1}) = 0$, which in turn implies that $\iii_n(f_i) = \iii_n(g) = 0$. Hence Claim~3 has been established.
	\smallskip

 	\noindent \textsc{Claim 4.} $\deg f_i \ge \alpha_{n_0 - 1} + \deg f_{i+1}$ for every $i \in \nn_0$.
	\smallskip
	
	\noindent \textsc{Proof of Claim 4.} Fix $i \in \nn_0$. After writing $f_i = f_{i+1}g$ for some $g \in F[M]$, we can proceed as in the proof of Claim~1 to obtain the inequality $\deg f_i > \alpha_{n_0 - 1}  + \deg f_{i+1}$ and the existence of a nonzero element $q \in \text{supp} \, g$ such that $\ifi(q) = 0$, and so $q \in \langle A \rangle$. Now it follows from Claim~3 that $\iii_n(q) = 0$ for every $n \in \nn$ with $n \ge n_0$. Now we can proceed as we did in the proof of Claim~1 to complete our argument, so Claim~4 is established.
    \smallskip
 
	% \noindent \textsc{Claim 4.} For each $i \in \nn_0$, the inequality $\deg f_i \ge \alpha_{n_0 - 1} + \deg f_{i+1}$ holds.
	% \smallskip
	
	% \noindent \textsc{Proof of Claim 4.} \textcolor{red}{[TODO: I'll try to merge this claim with claim 1.]} Fix $i \in \nn_0$, and then write $f_i = f_{i+1}g$ for some $g \in F[M]$. Then $\ifi(g) = 0$. In addition, if $\max\{\ifi(b) : b \in \text{supp} \, g\} \ge 1$, then $\deg g > \frac12$, whence
	% \begin{equation*}
	% 	\deg f_i = \deg g + \deg f_{i+1} > \frac12 + \deg f_{i+1} > \frac1{100} + \deg f_{i+1} > \alpha_{n_0 - 1} + \deg f_{i+1}.
	% \end{equation*}
	% As a result, we can assume that $\ifi(b) = 0$ for every $b \in \text{supp} \, g$ and, therefore, we can pick a nonzero element $b \in \text{supp} \, g \cap \langle A \rangle$ because $g$ does not belong to $F[M]^\times$. From Claim~3 we see $\iii_n(b) = 0$ for every $n \in \nn$ with $n \ge n_0$. Thus, if we write $q = c_1 \alpha_1 + \dots + c_\ell \alpha_\ell$ for some $c_1, \dots, c_\ell \in \nn_0$, we obtain that $c_j = 0$ for every index $j \in \ldb n_0, \ell \rdb$ as, otherwise, $\iii_j(q) = c_j \ge 1$. Hence $q \ge \alpha_j$ for some $j \in \ldb 1, n_0 - 1 \rdb$. This, along with the fact that $(\alpha_n)_{n \ge 1}$ is a decreasing sequence, ensures that $q \ge \alpha_{n_0 - 1}$. Now
	% \begin{equation*}
	% 	\deg f_i = \deg g + \deg f_{i+1} \ge b + \deg f_{i+1} \ge \alpha_{n_0 - 1}+\deg f_{i+1},
	% \end{equation*}
	% and so Claim~4 is established.
	% \smallskip

    With Claim~4 playing the role of Claim~1, we can produce the desired contradiction by mimicking the last paragraph of Case~1, and this completes our proof.
     %
	% It now follows from Claim~4 that if $d := \left\lfloor \frac{\deg f}{\alpha_{n_0 - 1}}\right\rfloor +1$, then $\deg f_d \le \deg f - d \alpha_{n_0 - 1} < 0$; however, this contradicts the fact that $f_d$ is nonzero.
%		
%	Therefore
%	\begin{equation*}
%		\deg f_{\left\lfloor \frac{\deg f}{\alpha_{n_0 - 1}}\right\rfloor +1} \le \deg f - \left( \left\lfloor \frac{\deg f}{\alpha_{n_0 - 1}}\right\rfloor + 1 \right)\alpha_{n_0 - 1} < 0,
%	\end{equation*}
%	which leads to a contradiction.
\end{proof}

We conclude with the following corollary, explicitly adding a class of one-dimensional monoid algebras to the constructions of atomic domains not satisfying the ACCP given in~\cite{aG74,aZ82,mR93,BC19,GL23,GL23a}.

\begin{cor}
    For each field $F$, the monoid algebra $F[M]$ is an atomic domain that does not satisfy the ACCP.
\end{cor}

\bigskip
%%%%%%%%%%%%%%%
%%%%%%%%%%%%%%%
\section*{Acknowledgments}

This collaboration took place as part of PRIMES-USA at MIT, and the authors would like to thank the directors and organizers of the program for making this research experience possible. While working on this paper, the second author was kindly supported by the NSF awards DMS-1903069 and DMS-2213323.

\bigskip
%%%%%%%%%%%%%%
%%%%%%%%%%%%%%

\end{document}